\documentclass[a4paper]{article}

\usepackage{a4wide}

\RequirePackage{amsthm,amsmath,amsfonts,dsfont,color}
\RequirePackage{amssymb}
\RequirePackage[numbers]{natbib}
\RequirePackage[colorlinks,citecolor=blue,urlcolor=blue]{hyperref}
\RequirePackage{graphicx}
\usepackage{booktabs}

\theoremstyle{plain}
\theoremstyle{plain}

\newtheorem{theorem}{Theorem}[section]
\newtheorem{rem}{Remark}[section]

\newtheorem{definition}{Definition}[section]
\newtheorem{proposition}[theorem]{Proposition}
\newtheorem{lemma}[theorem]{Lemma}
\DeclareMathOperator\card{card}
\DeclareMathOperator\supp{supp}

\title{Parametric convergence rate of some nonparametric estimators
  in mixtures of power series distributions}

\author{
  Fadoua Balabdaoui\thanks{Department of Mathematics, ETH Zurich, Zurich, Switzerland, email: fadouab@ethz.ch}
  \and
  Harald Besdziek\thanks{Department of Mathematics, ETH Zurich, Zurich, Switzerland, email: harald.besdziek@stat.math.ethz.ch}
  \and
  Yong Wang\thanks{Department of Statistics, University of Auckland, Auckland, New Zealand, email: yongwang@auckland.ac.nz}
}

\date{July 30, 2025}

\begin{document}

\maketitle

\begin{abstract}
  We consider the problem of estimating a mixture of power series
  distributions with infinite support, to which belong very well-known
  models such as Poisson, Geometric, Logarithmic or Negative Binomial
  probability mass functions. We consider the nonparametric maximum
  likelihood estimator (NPMLE) and show that, under very mild
  assumptions, it converges to the true mixture distribution $\pi_0$
  at a rate no slower than $(\log n)^{3/2} n^{-1/2}$ in the Hellinger
  distance. Recent work on minimax lower bounds suggests that the
  logarithmic factor in the obtained Hellinger rate of convergence can
  not be improved, at least for mixtures of Poisson
  distributions. Furthermore, we construct nonparametric estimators
  that are based on the NPMLE and show that they converge to $\pi_0$
  at the parametric rate $n^{-1/2}$ in the $\ell_p$-norm
  ($p \in [1, \infty]$ or $p \in [2, \infty])$: The weighted least
  squares and hybrid estimators. Simulations and a real data
  application are considered to assess the performance of all
  estimators we study in this paper and illustrate the practical
  aspect of the theory. The simulations results show that the NPMLE
  has the best performance in the Hellinger, $\ell_1$ and $\ell_2$
  distances in all scenarios. Finally, to construct confidence
  intervals of the true mixture probability mass function, both the
  nonparametric and parametric bootstrap procedures are
  considered. Their performances are compared with respect to the
  coverage and length of the resulting intervals.
\end{abstract}


\section{Introduction}
\subsection{General scope and existing literature}

Mixture models are commonly used in a wide range of applications, for
example biology, economics, engineering, finance, insurance, medicine
and the social sciences, to name only a few.  We refer the reader to
the excellent works \cite{Lindsay}, \cite{McLachlan} and
\cite{Titterington} for an overview of classical results. The success
of mixture models can be explained by their flexibility in fitting
different types of datasets and the fact that they underpin many
statistical techniques such as clustering, empirical Bayes procedures,
discriminant and image analysis. An important special case are
mixtures of discrete distributions, which serve as a popular tool for
analyzing count data, see e.g. \cite{Simar}, \cite{hell},
\cite{CheeWang2016}, \cite{giguelay2018} and \cite{balabdkulagina}. In
this work, we focus on an important subclass that includes a wide
range of discrete distributions: The class of power series
distributions (PSD). To define a PSD, consider
\begin{eqnarray*}
b(\theta):= \sum_{k=0}^{\infty}b_k\theta^k,
\end{eqnarray*}
for $b_k \geq0,$ to be a power series with radius of convergence
$R$. Let $\Theta:=[0,R]$ if $b(R)<\infty$ and $\Theta:=[0,R)$ if
$b(R)=\infty$, and define the support set
$\mathbb{K}:=\{k:b_k>0\}$. Famous examples are the Poisson, Geometric,
Logarithmic and Negative Binomial distributions. In these examples,
$\mathbb K$ is either equal to the set of all non-negative integers
$\mathbb N$ or is of the form $\{r, r+1, \ldots \}$ for some known
integer $r > 0 $. In the latter case, we can consider the
corresponding PSD with $\tilde{b}_k = b_{k + r}, k \in \mathbb N$,
whose normalizing constant is given by
$\tilde{b}(\theta) = \theta^{-r} b(\theta)$. In fact, by definition
$\tilde{b}(\theta) = \sum_{k \in \mathbb N} \tilde{b}_k \theta^k$ and
hence
\begin{eqnarray*}
  \tilde{b}(\theta) &=& \sum_{k \in \mathbb N} b_{k+r} \theta^{k}=  \sum_{k=r}^{\infty} b_{k} \theta^{k-r} = \theta^{-r}b(\theta).
\end{eqnarray*}
Therefore, we will assume in the sequel that $\mathbb K = \mathbb N$.
Note that for PSDs with a finite support set, i.e., with
$\card(\mathbb K) < \infty$, it is already known that the
nonparametric maximum likelihood estimator converges to the truth with
the fully parametric rate of $n^{-1/2}$ in the Hellinger distance and
hence in all the $\ell_p$ distances for $p \in [1, \infty]$.  This is
one reason this case will not be treated in this paper. Define now a
PSD by setting
\begin{eqnarray*}
f_\theta(k):=\frac{b_k\theta^k}{b(\theta)}
\end{eqnarray*}
for any $\theta \in \Theta$ and any $k \in \mathbb{N}$.  We are
interested in distributions that result from mixing the parameter
$\theta$.  Let $Q_0$ be a general distribution on $\Theta$, and define
the corresponding mixture probability mass function (pmf) $\pi_0$ via
\begin{eqnarray*}
\pi_0(k) :=  \pi(k; Q_0) =\int_{\Theta} f_\theta(k) dQ_0(\theta),
\end{eqnarray*}
for $k \in \mathbb{N}$. Assume that we observe i.i.d. random variables
$X_1,X_2,\ldots,X_n$ distributed according to $\pi_0$. Let
$\widehat{Q}_n$ denote the nonparametric maximum likelihood estimator
(NPMLE) of $Q_0$ based on the sample $(X_1, \ldots, X_n)$. For Poisson
mixtures, it was proved in \cite{Simar} that for each $n$,
$\widehat{Q}_n$ is a unique distribution on $[0,\infty)$ which is
supported on a finite number of points and is strongly consistent in
the sense that with probability equal to 1, the estimator
$\widehat{Q}_n$ converges weakly to the true distribution $Q_0$. In
\cite{jewell} and other papers this result was extended to many other
discrete distributions, with the only requirement that $Q_0$ is
identifiable.  Let $\widehat{\pi}_n$ be the corresponding NPMLE of
$\pi_0$, that is
\begin{eqnarray*}
\widehat{\pi}_n(k):= \pi(k; \widehat{Q}_n)=\int_{\Theta} f_\theta(k) d\widehat{Q}_n(\theta),
\end{eqnarray*}
for $ k \in \mathbb{N}$.  In our setting, existence of
$\widehat{\pi}_n$ can be shown using Theorem 18 in \cite{lindsay1995}
(see Theorem \ref{ExistMLE} below and the supplementary material for a
proof).  In addition, let $\bar{\pi}_n$ denote the empirical
estimator, that is
\begin{eqnarray*}
\bar{\pi}_n(k):=  n^{-1}  \sum_{i=1}^n \mathds{1}_{\{X_i = k\}},
\end{eqnarray*}
for $ k \in \mathbb{N}$, the observed proportion of the data equal to
$k$.

Before describing the scope and main results, we first provide the
reader with an overview of the convergence rates obtained in mixtures
of PSDs.  Recall that for two probability measures $\pi_1$ and $\pi_2$
defined on $\mathbb N$, the (squared) Hellinger distance is defined as
\begin{eqnarray*}
  h^2(\pi_1,\pi_2):=\frac{1}{2} \sum_{k \in \mathbb{N}} \left( \sqrt{\pi_1(k)}-\sqrt{\pi_2(k)} \right)^2   = 1 - \sum_{k \in \mathbb N} \sqrt{\pi_1(k)\pi_2(k)}.
\end{eqnarray*}
In \cite{Patilea} it was shown that for a wide range of PSDs, the rate
of convergence of the NPMLE in the sense of the Hellinger distance is
$(\log n)^{1+\epsilon}/\sqrt{n}$, for any $\epsilon > 0$. To obtain
this result, the author made use of the assumption that the true
mixing distribution $Q_0$ is compactly supported on an interval
$[0,M]$, with $0 < M < 1 \le R$. This assumption is used to get
control on the tail behavior of $\pi_0$. However, it has the
unfortunate effect that if $M \ge 1$, it is unknown as of yet whether
a similar rate of convergence holds for a general distribution
$Q_0$. While \cite{Patilea} is the only work known to us that derives
the rate of the NPMLE in the Hellinger distance, faster rates for a
fixed $k \in \mathbb{N}$ can be obtained. In the case of Poisson
mixtures, it was shown in \cite{Lambert} that the scaled estimation
error $\sqrt{n}(\widehat \pi_n(k) - \pi_0(k))$ converges, as $n$ tends
to infinity, to a normal distribution with mean zero and variance
$\pi_0(k)(1-\pi_0(k))$ for all $k \in \mathbb{N}$. In other words:
$\widehat \pi_n(k)$ is asymptotically normal for all
$k \in \mathbb{N}$. This asymptotic normality holds also generally in
the multivariate case where a single $k$ is replaced by any finite
subset $J \subset \mathbb{N}$. To obtain this result, the authors of
\cite{Lambert} require that $Q_0$ is not only compactly supported but
that it exhibits in addition a certain behavior in the neighborhood of
the origin, which automatically excludes all distributions that are
supported only on a finite number of points. In \cite{Patilea2005}
this result was generalized from the Poisson case to arbitrary PSDs,
nevertheless still under nearly the same strong assumptions on $Q_0$,
thereby excluding the important class of all finite mixtures.

\subsection{Rates in mixtures of densities with respect to Lebesgue
  measure}

Although mixtures of densities with respect to Lebesgue measure are
unrelated to the kind of mixtures we consider here, we would like to
note that to the best of our knowledge, the $n^{-1/2}$-rate has not
been achieved in a global sense by some nonparametric estimator in
such mixtures. Let us invoke the best-studied mixtures of densities in
the literature on the absolutely continuous setting, namely mixtures
of Gaussian densities.  In \cite{GvdV01}, it was shown that the NPMLE
converges at the rate $(\log n)^{\kappa} n^{-1/2}$, in the Hellinger
distance, for some $\kappa \ge 1$ or $\kappa \ge 3/2$ depending on
whether the model is location or location-scale mixture. In the first
model for example, the mixing distribution of the location parameter
is assumed to be compactly supported with a slowly growing support
while the scale parameter is taken to be arbitrary between two fixed
bounds.  Note that these results provided a significant improvement
over the rates obtained in \cite{GW00} for the sieve MLE shown to
converge only at the rate $(\log n)^{(1+\delta)/6} n^{-1/6}$ for some
$\delta > 0$.  Under the assumption of an exponentially tailed mixing
distribution, \cite{Zha09} showed that the generalized NPMLE of
location-scale mixture of Gaussian densities converges at the rate
$(\log n)^\kappa n^{-1/2}$ for some $\kappa > 3/4$.  In all the
references mentioned above, the convergence rate is nearly parametric
but \textit{not} parametric.

It is important to note that the rate
$(\log n)^\kappa n^{-1/2}, \kappa > 0$ can be far from being achieved
in case the mixed kernel is not very smooth. Examples include
estimation of non-degenerate monotone and convex densities with
respect to Lebesgue measure. In the former problem, it is known that
the model is equivalent to a scale mixture of uniform densities while
the latter is equivalent to a scale mixture of triangular densities.
Note that uniform densities are step functions while triangular
densities are continuous but not continuously differentiable. The
NPMLE is known to converge at the rates $n^{-1/3}$ and $n^{-2/5}$
under the assumption that the first/second derivative is not equal to
$0$. Here, we can refer to the works of \cite{piet85}, \cite{durot07}
and \cite{GJW01}.  These results can be extended in the problem of
estimating a $k$-monotone density, where the rate of the NPMLE is
$n^{-k/(2k+1)}$; see e.g., \cite{Balabdaoui07} and
\cite{GaoWell09}. Thus, in the continuous setting, the convergence
rate of the NPMLE of the mixture distribution seems to depend on the
smoothness of the kernel to be mixed. The same smoothness does not
play a similar role in mixtures of discrete distributions.  For
example, the NPMLE of a monotone pmf was shown in \cite{hannajon} to
converge at the parametric rate $n^{-1/2}$ in the $\ell_p$-norms, for
any $p \in [2, \infty]$. The same parametric rate was obtained in
nonparametric estimation of a unimodal, convex, $k$-monotone and
completely monotone pmf; see e.g., \cite{fadouahanna},
\cite{balabd2017}, \cite{giguelay} and \cite{fadouagab}.

Thus, we believe that the obtained $n^{-1/2}$-rate of nonparametric
estimators in the mixtures of power series distributions considered in
this paper is mainly due to the discreteness of the sample space. This
discreteness impacts not only the size of the distribution class and
hence the corresponding entropy (this is also the case for mixtures of
very smooth kernels in the continuous setting) but also induces the
$n^{-1/2}$-rate for the empirical estimator of the true pmf.  As this
estimator is the basis of other more sophisticated estimators, these
can be shown easily to inherit this fast rate, particularly if they
are constructed via some $\ell_p$-projection. In the continuous case,
basic nonparametric estimators of a density with respect to Lebesgue
measure; e.g., kernel estimations, which converge at the parametric
rate in a global sense cannot be constructed. This is, in our opinion,
one major difference between the discrete and continuous setting.

\subsection{Organization of the manuscript}

The manuscript will be structured as follows. In
Section~\ref{sec:global-rate}, we show that for mixtures of many
well-known PSDs, the NPMLE converges in the Hellinger distance at a
nearly parametric rate. Herewith we mean that the parametric rate is
inflated by a power of a logarithmic term, as in \cite{Patilea}.
However, we differ here from the work of \cite{Patilea} in that we do
not constrain the largest point in the support of the mixing
distribution to be strictly smaller than $1$. Instead, we allow for a
nearly arbitrary true mixing distribution $Q_0$, with the only main
requirement that it is compactly supported. The proof, as in
\cite{Patilea}, relies on techniques from empirical processes,
particularly on finding good upper bounds for the bracketing entropy
of the class of mixtures under study. In the same section, we present
the minimax lower bounds in the Hellinger distance obtained recently
in \cite{polyanskiy2021sharp} for mixtures of Poisson
distributions. These lower bounds, derived for compactly supported and
subexponential mixing distributions, strongly suggest that the
logarithmic factor obtained here and in \cite{Patilea} can not be
improved upon.

In Section~\ref{sec:estimators}, we construct nonparametric estimators
that are based on the NPMLE and which converge to the true mixture at
the $n^{-1/2}$-rate in any $\ell_p$-distance for all
$p \in [1, \infty]$ or at least for $ [2, \infty]$: The weighted Least
Squares and hybrid estimators.  In Section~\ref{sec:computation}, we
support our theoretical work via simulations and present an
application to real data. In order to have a good overview of how the
estimators perform, several settings are considered with different PSD
families and mixing distributions. Our study shows clearly that the
NPMLE has the best performance. Moreover, we consider construction of
(asymptotic) confidence intervals of the true pmf using
bootstrap. Both the nonparametric and parametric re-sampling
procedures are considered and the coverage and length of the produced
confidence intervals are investigated. By the term \lq\lq
parametric\rq\rq \ we mean that a bootstrap sample is drawn from the
fitted NPMLE.  Towards the end of Section~\ref{sec:computation}, an
application to real earthquake data is considered where the dataset
consists of yearly counts of world major earthquakes with magnitude 7
and above for the years 1900--2021.  Finally,
Section~\ref{sec:discussion} provides a discussion as well as an
outlook to future research. Some proofs are kept in this main
manuscript, especially the ones that may help the reader understand
better the results that are being proved.  In case a proof or an
auxiliary result is relegated to the supplementary material, the
reader is notified.

\section{The global rate of the nonparametric maximum likelihood
  estimator in the Hellinger distance}
\label{sec:global-rate}

The NMPLE in mixture models has a long history which goes back to
\cite{kiefer1956}, where consistency is shown under certain regularity
conditions. In \cite{LindsayI} and \cite{LindsayII} a geometric
perspective was used to prove some very fundamental results about this
important estimator (existence, uniqueness, upper bound on the number
of support points, consistency, etc). In mixtures of discrete
distributions with an infinite support, one of the reasons that the
NPMLE is more appealing than the empirical estimator is that not only
does it preserve the model structure (existence of mixing) but it
copes much better with the lack of any information beyond the largest
order statistic. In fact, the NPMLE can be shown to have a superior
performance in the Hellinger, $\ell_1$- and $\ell_2$-distances than
the empirical estimator at the tail and over the whole support (we
refer the reader to our simulation results in
Sections~\ref{sec:estimators} and \ref{sec:computation}).

Several research works on the NPMLE or other minimum contrast/distance
estimators in mixtures of discrete distributions are known in the
literature; see e.g \cite{Simar}, \cite{Lambert}, \cite{Patilea},
\cite{NorrisIII1998391}, \cite{Patilea2005}, \cite{hell},
\cite{hannajon}, \cite{CheeWang2016}, \cite{balabd2017},
\cite{giguelay} and \cite{giguelay2018}, and \cite{bdF2019} to name
only a few. In these references, the focus is put on estimation of the
mixture distribution. For the problem of estimating the mixing
distribution, we refer the reader to our Section~\ref{sec:discussion}
where we recall the main results obtained in this area and the
possible connections that may exist with finding sharp lower bounds
for mixtures of PSDs.

\subsection{Assumptions on the mixture model} 

Consider a family of PSDs
$f_\theta(k) = b_k \theta^k /b(\theta), k \in \mathbb N,$ for
$\theta \in \Theta$, with $\Theta= [0, R]$ if $b(R) < \infty$ and
$[0, R)$ if $b(R)=\infty$. We assume that the true mixture $\pi_0$ is
of the form
\begin{eqnarray}\label{Model}
\pi_0(k)  = \int_{\Theta} f_\theta(k) dQ_0(\theta) , \  \ k \in \mathbb N,
\end{eqnarray} 
with $Q_0$ denoting the \emph{unknown} true mixing distribution. We
are interested in estimating $\pi_0$ based on $n$ i.i.d. observations
$X_1, \ldots, X_n \sim \pi_0$. In the following, we derive a global
rate of the NPMLE in the Hellinger distance. To achieve this, we will
need the following assumptions. \\

\par \noindent \textbf{Assumption (A1).}
\begin{itemize}
\item If $R < \infty$, then there exists $q_0 \in (0,1)$ such that the
  support of the true mixing distribution satisfies
  $\supp{Q_0} \subseteq [0,q_0R]$.
\item If $R = \infty$, then there exists $M > 0$ such that
  $\supp{Q_0} \subseteq [0,M]$.
\end{itemize}

\par \noindent \textbf{Assumption (A2).}
\begin{itemize}
\item If $Q_0(\{0\}) > 0$, then there exist $\eta_0 \in (0,1)$ and
  $\delta_0 \in (0, R)$ small such that $Q_0(\{0\}) \le 1-\eta_0$ and
  $\supp{Q_0} \cap (0,\delta_0) = \varnothing$.
\item If $Q_0(\{0\}) = 0$, then there exists $\delta_0 \in (0, R)$
  small such that $\supp{Q_0} \cap [0,\delta_0) = \varnothing$.
\end{itemize}
\par \noindent \textbf{Assumption (A3).}   There exists $V \in
\mathbb{N}$ such that $b_k/b_0 \ge k^{-k}$ \ for all $k \ge V$.  \\ 
\par \noindent \textbf{Assumption (A4).}  The limit $\lim_{k \to \infty} \{ b_{k+1}/b_k\} $ exists and belongs to $[0, \infty)$. \\

Some remarks about the assumptions above are in order. All the
constants in Assumptions (A1) and (A2) are \textit{unknown}.
Assumption (A1) hinders the mixture from putting a positive mass very
near the radius of convergence of the underlying PSD family. It is
clear anyway that the mixing distribution $Q_0$ has no support point
beyond the radius of convergence because then, the mixture would not
be well-defined. For the case where the radius of convergence is
infinite, the same assumption states that the support of the mixing
distribution has an upper limit $M$, though $M$ may be unknown to the
practitioner. This assumption is more general than the one made in
\cite{Patilea} where it was imposed that $Q_0$ is compactly supported
on $[0, M]$, with $M < 1$. Assumption (A2) is two-fold. First, it
excludes the trivial case where $Q_0$ is just a Dirac measure at
zero. Secondly, it impedes the mixture from being supported on the
interval $(0,\delta_0)$, for $\delta_0 > 0$.  Since $\delta_0$ can be
taken arbitrarily close to zero, this assumption is not very
restrictive in practice.  Note that in the case where
$Q_0(\{0\}) > 0$, the mixing distribution $Q_0$ can be viewed itself
as a mixture of a Dirac at $0$ and another distribution with support
on $[\delta_0, q_0R]$ or $[\delta_0, M]$ depending on finiteness of
$R$. Assumptions (A3) and (A4) are properties of the PSD family alone
and do not at all concern the mixing distribution $Q_0$. Note that
Assumption (A4) implies that
\begin{eqnarray}\label{ImplicationA4}
  \lim_{k \to \infty}  \frac{b_{k+1}}{b_k}  =  \frac{1}{R},  \ \textrm{if $R < \infty$}, \ \textrm{and} \   \lim_{k \to \infty}  \frac{b_{k+1}}{b_k} =0, \ \textrm{if $R=\infty$}.
\end{eqnarray}
Assumptions (A3) and (A4) are satisfied by all well-known PSDs. To
provide concrete examples, consider the Geometric and Poisson
families. Other families, like the Logarithmic and Negative Binomial
distribution, also fulfill these assumptions, which can be shown
analogously.

\begin{itemize}
\item \emph{The Geometric family:}
  $f_{\theta}(k) = (1-\theta) \theta^k, \theta \in [0, 1)$, with
  radius of convergence $R=1$. Then, $b_k =1$ for all
  $k \in \mathbb N$.  Thus, $b_k/b_0 \ge k^{-k}$ for all 
  $k \in \mathbb N$ and $ b_{k+1}/b_k =1, k \in \mathbb N$. \\

\item \emph{The Poisson family:}
  $f_{\theta}(k) = e^{-\theta} \theta^k/k!, \theta \in [0, \infty)$,
  with radius of convergence $R = \infty$. Then, $b_k = 1/k!$. Thus,
  $b_k/b_0 \ge k^{-k}$ for all $ k \in \mathbb N$, and
  $\lim_{k \to \infty} b_{k+1}/b_k = 0$.
\end{itemize}

\subsection{Rate of convergence of the NPMLE in the Hellinger distance}
 
Throughout this section, we assume that we deal with a mixture of PSDs
with an infinite support set. Without loss of generality, and to make
the exposition clear, we will assume from now on that
$\mathbb K = \mathbb N$.  We also assume that Assumptions (A1) - (A4)
hold true.  The case of finite support is much more
straightforward. There, it can be shown that the NPMLE converges at
the parametric rate $n^{-1/2}$ in the Hellinger distance and hence in
all the $\ell_p$-distances, for $p \in [1, \infty]$.  For the sake of
completeness, a proof of this fast rate can be found in the
supplementary material (see Theorem 3.2).

Let $\widehat{\pi}_n$ be the NPMLE of $\pi_0$ based on $X_1, \ldots, X_n  \stackrel{i.i.d.}{\sim} \pi_0$. Before getting into any asymptotic result for $\widehat{\pi}_n$, we need to make sure that it exists. 

\medskip
\begin{theorem}\label{ExistMLE}
  Consider a family of PSDs
  $f_\theta(k) = b_k \theta^k /b(\theta), k \in \mathbb N,$ for
  $\theta \in \Theta$, with $\Theta= [0, R]$ if $b(R) < \infty$ and
  $[0, R)$ if $b(R)=\infty$. Let the true mixture $\pi_0$ be of the
  form
  \begin{eqnarray*}
    \pi_0(k)  = \int_{\Theta} f_\theta(k) dQ_0(\theta) , \  \ k \in \mathbb N,
  \end{eqnarray*} 
  with $Q_0$ denoting the true mixing distribution. Then, the NPMLE
  for the mixing distribution $\widehat{Q}_n$ exists and is
  unique. The same holds true for $\widehat{\pi}_n$, the corresponding
  NPMLE for the mixture.
\end{theorem}
To show Theorem \ref{ExistMLE}, we can appeal to Theorem 18 in Chapter
5 of \cite{lindsay1995}. The main difficulty is to show that the
likelihood curve is compact. To circumvent the cases where the
original likelihood curve is not closed, one can augment it with the
vector of zero's so that compactness is obtained. A detailed proof can
be found in the supplementary material.
 
In the sequel, we will need the following quantities:
\begin{eqnarray}\label{t0tildetheta}
 t_0 =  \frac{q_0+1}{2} \mathds{1}_{\{R < \infty\}}  +  \frac{1}{2} \mathds{1}_{\{R = \infty\}}, \  \   \tilde \theta =  (q_0 R) \mathds{1}_{\{R < \infty\}} + M \mathds{1}_{\{R = \infty\}},
\end{eqnarray}
\begin{eqnarray}\label{UV}
U =  \Big \lfloor \tilde \theta \sup_{\theta \in (0, \tilde \theta)} \frac{b'(\theta)}{b(\theta)} \Big \rfloor +1,  \   \ W =  \min \left \{ w \ge 3:  \max_{k \ge w} \frac{b_{k+1}}{b_k}  \le \frac{t_0}{\tilde \theta}   \right \},
\end{eqnarray}
and
\begin{eqnarray}\label{N0}
&& N(t_0, \tilde \theta, \delta_0, \eta_0) \notag \\
&& =   \bigg \lfloor \exp\left \{\log(t_0^{-1/2}) \left( U \vee V \vee W \vee \frac{b(\delta_0)}{b_0 \eta_0} \vee \frac{1}{\delta_0} \right)   \right \} \vee  \frac{1}{t_0^{W-1}(1-t_0)} \bigg \rfloor +1, \notag \\
&&
\end{eqnarray} 
where $\lfloor z \rfloor$ denotes the integer part of some real number $z$.

\medskip

\par \noindent We now state our main convergence theorem for the NPMLE. 

\medskip

\begin{theorem} \label{Rate}
  Let $L > 2$. Also, let $t_0$, $U$ and $W$
  be the same constants defined in (\ref{t0tildetheta}) and
  (\ref{UV}).  Under Assumptions (A1)-(A4), there exists a universal
  constant $A > 0$ such that
\begin{eqnarray*}
P\left(h(\widehat \pi_n, \pi_0)  > L  \frac{(\log n)^{3/2}}{\sqrt n} \right)   & \le   & \frac{1}{(L^2/2 - 2)^2 (\log n)^2}  \\&& \  +  \ \frac{A}{L} \frac{1}{\log (1/t_0)^{3/2}} \left(1 +  \frac{1}{\log (1/t_0)^{3/2}}  \right)
\end{eqnarray*}
for all $n \ge N(t_0, \tilde \theta, \delta_0, \eta_0)$ where
$N(t_0, \tilde \theta, \delta_0, \eta_0)$ is the same integer as in
(\ref{N0}).  In particular, it follows that
$$h(\widehat \pi_n, \pi_0)   = O_{\mathbb P} \left( \frac{(\log n)^{3/2}}{\sqrt n}\right).$$

\end{theorem}

\medskip

The first statement of Theorem \ref{Rate} shows how the probability
that the MLE is outside the $L (\log n)^{3/2}/\sqrt n$- Hellinger ball
centered $\pi_0$ decays with $L$ and $n$.  It can be seen that the
constant $L$ has to be larger for smaller values of $\log(1/t_0)$ or
equivalently for values of $q_0$ that are close to $1$ in the case
$R < \infty$.  In other words, the $O_{\mathbb P}$ in the convergence
rate deteriorates if the right endpoint of the support of the true
mixing distribution $Q_0$ gets closer to the radius of convergence
$R$. More importantly, Theorem \ref{Rate} provides a lower bound on
the required sample size as a function of the parameters in
(A1)-(A3). As this allows us to investigate uniformity of the
established convergence over given classes of mixtures, we add the
following remark.

\smallskip

\begin{rem}
  Suppose that $R < \infty$. Let $\bar q \in (0,1)$,
  $\underline{\delta} \in (0, R)$, $\underline{\eta} \in (0,1)$, and
  consider
  $\mathcal{Q}_{\bar q, \underline{\delta}, \underline{\eta}}$ to be
  the class of mixing distributions $Q_0$ satisfying $q_0 \le \bar q$,
  $\delta_0 \ge \underline{\delta}$, and
  $\eta_0 \ge \underline{\eta}$. Then,
  $\tilde \theta = q_0 R \le \bar q R$ and hence
\begin{eqnarray*}
U \le \bar U: =  \Big \lfloor \bar q R \sup_{\theta \in (0, \bar q R)} \frac{b'(\theta)}{b(\theta)} \Big \rfloor +1.
\end{eqnarray*}
Also,
\begin{eqnarray*}
\frac{t_0}{\tilde \theta} =  \frac{q_0 +1}{2q_0 R}  \ge  \frac{\bar q +1}{2 \bar q R}.
\end{eqnarray*}
By definition of $W$, it is easy to see that $W \le \bar W$ with
$$
\bar W =  \min \left \{ w \ge 3:  \max_{k \ge w} \frac{b_{k+1}}{b_k}  \le \frac{\bar q +1}{2 \bar q R} \right \}.
$$
It follows from increasing monotonicity of $\theta \mapsto b(\theta)$
and $\delta_0 \le q_0 R \le \bar q R$ that
$b(\delta_0) \le b(\bar q R)$. Finally,
$t_0 = (q_0 +1)/2 \ge (\delta_0 + R)/(2R) \ge (\underline{\delta} +
R)/(2R)$ and $1-t_0 \ge (1-\bar q)/2$. Thus,
\begin{eqnarray*}
  \log(t_0^{-1/2})  \le \frac{1}{2} \log\left(\frac{2R}{\underline{\delta} + R} \right), \ \   \frac{1}{t_0^{W-1}(1-t_0)}  \le \left(\frac{2R}{\underline{\delta} + R}\right)^{W-1} \frac{2}{1-\bar q}
\end{eqnarray*}
and hence,
\begin{eqnarray*}
  && N(t_0, \tilde \theta, \delta_0, \eta_0) \\
  && \le  \bigg \lfloor \exp\left \{\frac{1}{2} \log\left(\frac{2R}{\underline{\delta} + R} \right) \left( \bar U \vee V \vee \bar W \vee \frac{b(\bar q R)}{b_0 \underline{\eta}} \vee \frac{1}{\underline{\delta}} \right)   \right \} \vee \left(\frac{2R}{\underline{\delta} + R}\right)^{W-1} \frac{2}{1-\bar q} \bigg \rfloor \\
  && \ \ + \ 1 \\
  && := \bar N.
\end{eqnarray*}
Then, Theorem \ref{Rate} implies that for all $n \ge \bar N$ and $L > 2$
\begin{eqnarray}\label{UnifRfinite}
  && \sup_{ Q_0 \in \mathcal{Q}_{\bar q, \underline{\delta}, \underline{\eta}}} P_{Q_0}\left(h\big(\widehat \pi_n, \pi(\cdot, Q_0)\big)  > L  \frac{(\log n)^{3/2}}{\sqrt n} \right) \notag  \\
  && \le   \frac{1}{(L^2/2 - 2)^2 (\log n)^2}  +  \ \frac{A}{L} \frac{1}{\log (2/(\bar q +1))^{3/2}} \left(1 +  \frac{1}{\log (2/(\bar q +1))^{3/2}}  \right). \notag \\
  &&
\end{eqnarray}
Now, consider the case $R = \infty$. For $\bar M > 0$,
$\underline \delta > 0$ and $\underline \eta > 0 $ define
$\mathcal Q_{\bar M, \underline \delta, \underline \eta}$ to be the
class of mixing distribution functions $Q_0$ for which $M \le \bar M$,
$\delta_0 \ge \underline \delta$ and $\eta_0 \ge \underline \eta$. In
this case, we have $t_0 = 1/2$ and $\tilde \theta = M \le \bar
M$. Also, $\delta_0 \le M \le \bar M$ and hence
$b(\delta_0) \le b(\bar M)$. Let
\begin{eqnarray*}
\bar U =  \Big \lfloor \bar M  \sup_{\theta \in (0, \bar M)} \frac{b'(\theta)}{b(\theta)} \Big \rfloor +1,  \ \ \bar W = \min \left \{ w \ge 3:  \max_{k \ge w} \frac{b_{k+1}}{b_k}  \le \frac{1}{2 \bar M} \right \},
\end{eqnarray*}
and define
\begin{eqnarray*}
  \bar N = \bigg \lfloor \exp\left \{\log(\sqrt 2) \left( \bar U \vee V \vee \bar W \vee \frac{b(\bar M)}{b_0 \underline \eta} \vee \frac{1}{\underline \delta} \right)   \right \} \vee  \frac{1}{2^W} \bigg \rfloor +1.
\end{eqnarray*}
Then, Theorem \ref{Rate} implies that for all $n \ge \bar N$ and $L > 2$
\begin{eqnarray}\label{UnifRinfinite}
 && \sup_{ Q_0 \in \mathcal{Q}_{\bar M, \underline{\delta}, \underline{\eta}}} P_{Q_0}\left(h\big(\widehat \pi_n, \pi(\cdot, Q_0)\big)  > L  \frac{(\log n)^{3/2}}{\sqrt n} \right) \notag  \\
 && \le   \frac{1}{(L^2/2 - 2)^2 (\log n)^2}  +  \ \frac{A}{L} \frac{1}{\log (2)^{3/2}} \left(1 +  \frac{1}{\log (2)^{3/2}}  \right). 
\end{eqnarray}

\end{rem}

\medskip \medskip \medskip
In the following, we provide the reader
with the most relevant elements that go into the proof of Theorem
\ref{Rate}.  The main argument relies on finding a good upper bound
for the bracketing entropy of the class of pmf's to which $\pi_0$
belongs. Since the support set is infinite, the tail behavior of
$\pi_0$ will be determinant in deriving such a bound.  But before
doing so, we first need some preparatory lemmas, which can be regarded
as standalone results. The following lemma gathers some properties
satisfied by the power series distribution, and hence does not involve
the estimation procedure nor the data. Its proof can be found in the
supplementary material.

\medskip

\begin{lemma} \label{Prop1}
 Let $t_0, \tilde \theta, U$ and $W$ be the same constants as in (\ref{t0tildetheta}) and (\ref{UV}). Then, the following properties hold.
 \medskip
 
\begin{enumerate}
\item For all $k \ge U$, the mapping $\theta \mapsto f_\theta(k)$ is
  non-decreasing on $[0, \tilde{\theta}]$.
\item For all $k \ge W$, we have
\begin{eqnarray} \label{W}
b_{k+1}  \le  \frac{t_0 }{\tilde{\theta}} b_k.
\end{eqnarray}
\item For all $K \ge \max(U, W)$, we have that
\begin{eqnarray} \label{TailBound}
\sum_{k \ge K+1}  \pi_0(k)  \le A t_0^K,
\end{eqnarray}
where 
\begin{eqnarray*}
  A =  \frac{b_W \tilde \theta^W}{b(\tilde{\theta})} \frac{1}{t^{W-1}_0(1-t_0)} =  \frac{f_W(\tilde \theta)}{t^{W-1}_0(1-t_0)}.  \\
\end{eqnarray*}

\item The map $k \mapsto \pi_0(k)$ is strictly decreasing for $k \ge W$. 
\end{enumerate}
\end{lemma}

\medskip

We now move to the key part of this manuscript, which is about finding
a good upper bound for the bracketing entropy of the class of mixtures
under study. In the sequel, we use the standard notation from
empirical process theory.
\begin{itemize}

\item  $\mu$: The counting measure on $\mathbb N$.

\item   $\mathbb{P}$: The true probability measure; i.e., $d\mathbb{P}/d\mu=\pi_0$.

\item  $\mathbb{P}_n$:  The empirical measure; i.e., $\mathbb{P}_n:=\frac{1}{n} \sum_{i=1}^n \delta_{X_i}$, with $\delta_{X_i}, i \in \{1,\ldots,n\},$ the Dirac measures associated with the observed sample.

\end{itemize}

Let $\mathcal Q$ be the set of mixing distributions defined on
$\Theta$. Also define
\begin{eqnarray}\label{ClassM}
\mathcal M  & =   &   \Big \{\pi: \pi(k)=  \pi(k, Q) =\int_{\Theta}  \frac{b_k \theta^k}{b(\theta)}   dQ(\theta), \ \  \textrm{for $k \in \mathbb N$ and $Q \in \mathcal Q$}  \Big \}.
\end{eqnarray}
Set now
\begin{eqnarray}\label{Kn}
K_n:=  \min \bigg\{K \in \mathbb N:  \sum_{\{k > K\}} \pi_0(k)  \le \frac{(\log n)^{3}}{n}\bigg \},
\end{eqnarray}
and 
\begin{eqnarray}\label{taun}
\tau_n  :=  \inf_{0 \le k \le K_n}  \pi_0(k).
\end{eqnarray}
Existence of $K_n$ follows clearly from the non-increasing monotonicity of the map $K \mapsto \sum_{\{k > K\}} \pi_0(k)$. We now provide an upper bound for a particular combination of $K_n$ and $\tau_n$. The proof can be found in the supplementary material.

\medskip

\begin{lemma} \label{Kntaun}
  Suppose the assumptions (A1)-(A3) are
  satisfied. Then, for $n \ge N(t_0, \tilde \theta, \delta_0, \eta_0)$
  where $N(t_0, \tilde \theta, \delta_0, \eta_0)$ is the same integer
  defined in (\ref{N0}), it holds that
\begin{eqnarray*}
(K_n + 1) \log(1/\tau_n)  \le  \frac{81  (\log n)^{3}}{\log (1/t_0)^{3}}.
\end{eqnarray*}
\end{lemma}

\medskip
\medskip

For $\delta > 0$, consider the class
\begin{eqnarray}\label{Gndelta}
\mathcal{G}_n (\delta) &:= & \Big \{ k \mapsto g(k) = \frac{\pi(k)  - \pi_0(k)}{\pi(k)  +  \pi_0(k)} \mathbb{I}_{\{0 \le k \le K_n\}}:  \pi \in \mathcal M \  \ \textrm{such that} \notag \\
&&  \hspace{6.5cm} \ \  h(\pi, \pi_0) \le \delta   \Big\}, 
\end{eqnarray}
where $\mathcal M$ is defined in (\ref{ClassM}). For a given
$\nu > 0$, denote by $H_B(\nu,\mathcal{G}_n(\delta),\mathbb{P)}$ the
$\nu$-bracketing entropy of $\mathcal{G}_n (\delta)$ with respect to
$L_2(\mathbb P)$; i.e., the logarithm of the smallest number of pairs
of functions $(l, u)$ such that $l \le u$ and
$\int (l-u)^2 d\mathbb P \le \nu^2$ needed to cover
$\mathcal{G}_n(\delta)$. Also define the corresponding bracketing
integral
\begin{eqnarray} \label{BracketingIntegral}
\widetilde{J}_B(\delta, \mathcal{G}_n(\delta), \mathbb P)  :=  \int_{0}^\delta  \sqrt{1 + H_B(u, \mathcal{G}_n(\delta), \mathbb P)} du.  
\end{eqnarray}
We now provide an upper bound for this bracketing integral. But before
doing so, we would like to explain the intuition behind our approach.
Each element of the class $\mathcal{G}_n(\delta)$ is forced to have
its support included in $[0, K_n]$. If $K_n$ were not depending on
$n$, then the $\epsilon$-bracketing entropy of the class would be of
order $1/\epsilon$ as in any parametric model. In fact, in the case
where the true pmf has a finite support, with cardinality $ K \ge 2$,
the model is fully parametric with dimension equal to $ K - 1 $ and
the rate of convergence of the NPMLE can be shown to be $1/\sqrt
n$. This rate is rather independent of whether $\pi_0$ is the mass
function of a mixture distribution or not. Here, we deal with the more
difficult case of infinite support.  To mimic the situation with
finite support, the true support is recovered progressively through
$[0, K_n]$ as $n$ grows. In choosing $K_n$, one has to strike a
balance between having a small probability at the tail and small
entropy for the class, which clearly go in opposite directions.
However, even when this balance is achieved, the parametric rate
$1/\sqrt n$ cannot be obtained in this case as the entropy is inflated
by a logarithmic factor due to the $\log n$-term in $K_n$.

\medskip

\begin{proposition} \label{EntropyBound}
Let $t_0$ and $N(t_0, \tilde \theta, \delta_0, \eta_0)$ be the same quantities defined in (\ref{t0tildetheta}) and (\ref{N0}) respectively. Then for $n \ge N(t_0, \tilde \theta, \delta_0, \eta_0)$, we have that
\begin{eqnarray*}
\widetilde{J}_B(\delta, \mathcal{G}_n(\delta), \mathbb P) \le  \frac{27  \delta  (\log n)^{3/2}}{\log (1/t_0)^{3/2}}.
\end{eqnarray*}
\end{proposition}

\begin{proof}
  We make use of the following inequality, which is also the
  inequality (4.4) in \cite{Patilea}:
\begin{eqnarray}\label{IneqPatilea}
\left(\sum_{k  \in \mathbb N}  \left( \frac{\pi(k)  - \pi_0(k)}{\pi(k) + \pi_0(k)}  \right)^2 \pi_0(k) \right)^{1/2} \le 2 h(\pi, \pi_0).
\end{eqnarray}
In particular, this implies that if $\pi_0(k) \ge \kappa_n$, for some
threshold $\kappa_n > 0$, we have for all $k \in \mathbb N$ that
\begin{eqnarray*}
\frac{\vert \pi(k)  - \pi_0(k)  \vert}{\pi(k) + \pi_0(k)} \mathbb{I}_{\{\pi_0(k) \ge \kappa_n \}}  \le  \frac{2 h(\pi, \pi_0)}{\sqrt{\kappa_n}}.
\end{eqnarray*}
Thus, for any element $g \in \mathcal{G}_n(\delta)$ and for all
$k \in \{0, \ldots, K_n \}$, we have that
\begin{eqnarray*}
 g(k)  \in \left[-\frac{2 \delta}{\sqrt{\tau_n}}, \frac{2 \delta}{\sqrt{\tau_n}}\right],
\end{eqnarray*}
with $\tau_n$ defined in (\ref{taun}) . We now partition this interval
into $N$ sub-intervals of the same size $s$ (depending on $\delta$),
which must satisfy $s N = 4 \delta/\sqrt{\tau_n}$. For any
$k \in \{0, \ldots, K_n\}$, there exists $i_k \in \{0, \ldots, N-1\}$
such that
\begin{eqnarray*}
l_i(k)  := -\frac{2 \delta}{\sqrt{\tau_n} } +  i_k s   \le  g(k)   \le    u_i(k) :=  -\frac{2 \delta}{\sqrt{\tau_n}}  +  (i_k+1) s.
\end{eqnarray*}
Note that 
\begin{eqnarray*}
\sum_{k \le K_n} (u_i(k)  -  l_i(k))^2  \pi_0(k)    =  s^2 \sum_{k \le K_n}  \pi_0(k)  \le s^2.
\end{eqnarray*}
Thus, we can take $\nu = s$ so that $[l_i(k), u_i(k)]$ is a
$\nu$-bracket, implying that
\begin{eqnarray*}
N =  \frac{4\delta}{\sqrt{\tau_n} \nu}.
\end{eqnarray*}
The number of brackets needed to cover $\mathcal{G}_n(\delta)$ is at
most $N^{(K_n+1)}$. Hence, an upper bound for the $\nu$-bracketing
entropy is
\begin{eqnarray*}
H_B(\nu,\mathcal{G}_n(\delta),\mathbb{P}) & \le & (K_n +1)  \log N   =    (K_n+1)  \log \left(\frac{4 \delta}{\sqrt{\tau_n} \nu}  \right)  \\
                          &  \le  &  (K_n+1) \log 4  + \frac{1}{2} (K_n+1)  \log\left(\frac{1}{\tau_n}\right)   +  (K_n+1)  \log \left(\frac{\delta}{\nu}\right)  \\
&   \le  &   (K_n+1)  \log\left(\frac{1}{\tau_n}\right)   +  (K_n+1)  \log \left(\frac{\delta}{\nu}\right)
\end{eqnarray*}
for $n$ large enough, where we used the fact that
$\lim_{n \to \infty} \tau^{-1}_n = \infty$. Using
$\sqrt{x+y} \le \sqrt{x} + \sqrt y$ for all $x, y \in [0,\infty)$, we
get
\begin{eqnarray*}
\int_{0}^\delta   \sqrt{H_B(u, \mathcal{G}_n(\delta), \mathbb P)} du 
& \le &   \sqrt{K_n +1} \sqrt{\log\left(\frac{1}{\tau_n}\right) }  \delta   +  \sqrt{K_n +1} \int_{0}^\delta   \sqrt{\log \left(\frac{\delta}{u}\right)} du.
\end{eqnarray*}
By elementary calculus, we can bound the integral in the second term
by $\delta$. Hence, we obtain for $n$ large enough that
\begin{eqnarray*}
\int_{0}^\delta   \sqrt{H_B(u, \mathcal{G}_n(\delta), \mathbb P)} du \le  \sqrt{K_n +1} \left( \sqrt{\log\left(\frac{1}{\tau_n}\right) }  \delta   +  \delta \right) \le 2\delta \sqrt{K_n +1}  \sqrt{\log\left(\frac{1}{\tau_n}\right)}.
\end{eqnarray*}
Thus, for $n \ge N(t_0, \tilde \theta, \delta_0, \eta_0)$ defined in
(\ref{N0}), we obtain by definition of the bracketing integral and the
inequality $\sqrt{x+y} \le \sqrt{x} + \sqrt y$ that
\begin{eqnarray*}
\widetilde{J}_B(\delta, \mathcal{G}_n(\delta), \mathbb P)
\le \delta + \int_{0}^\delta  \sqrt{H_B(u, \mathcal{G}_n(\delta), \mathbb P)} du 
& \le  &   3 \delta \sqrt{K_n +1}  \sqrt{\log\left(\frac{1}{\tau_n}\right)}   \\
& \le  &  \frac{27  \delta   (\log n)^{3/2}}{\log (1/t_0)^{3/2}},
\end{eqnarray*}
where Lemma \ref{Kntaun} was applied in the last step.
\end{proof}

\medskip

Now we are ready to prove Theorem \ref{Rate}, our main theorem for
this section. For this aim, we shall make use of the following basic
inequality, which is re-adapted from Lemma 4.5 of \cite{sara}.

\medskip

\begin{lemma} \label{BasicIneq} Let $\pi_0 \in \mathcal M$, where
  $\mathcal M$ was defined above in (\ref{ClassM}), and
  $\widehat{\pi}_n$ the NPMLE of $\pi_0$.  Then, it holds that
\begin{eqnarray}\label{ineqbasique}
h^2(\widehat \pi_n, \pi_0)  \le \int \frac{\widehat{\pi}_n -  \pi_0}{\widehat{\pi}_n +  \pi_0}  d(\mathbb P_n -  \mathbb P).
\end{eqnarray}
\end{lemma}

\medskip

The proof of the basic inequality can be found in \cite{sara}, but the
reader can find it also in the supplementary material for the sake of
completeness.  Note that the class $\mathcal M$ can be replaced by any
convex class of pmf's provided that the NPMLE exists.

We will now combine Proposition \ref{EntropyBound} with Lemma
\ref{BasicIneq} and the so-called peeling device, a well-known
technique from empirical process theory, to show that the NPMLE
converges at a rate that is no slower than $(\log n)^{3/2}/\sqrt n$.

\medskip

\noindent \textbf{Proof of Theorem \ref{Rate}.}  Let $L > 2$ and
$\mathcal M$ be as in (\ref{ClassM}). Consider the sequence
$\{\delta_n\}_{n \ge 1}$:
\begin{eqnarray*}
\delta_n  := \frac{(\log n)^{3/2}}{\sqrt n}.
\end{eqnarray*}
It follows from Lemma \ref{BasicIneq} that
\begin{eqnarray*}
&& P(h(\widehat \pi_n, \pi_0)   >  L \delta_n) \\
&&   \le  P\left( \sup_{\pi \in \mathcal M: h(\pi, \pi_0)  > L \delta_n}  \left\{\int \frac{\pi - \pi_0}{\pi + \pi_0}  d(\mathbb P_n - \mathbb P)  -  h^2( \pi, \pi_0) \right \}  \ge 0 \right) 
\end{eqnarray*}
and therefore
\begin{eqnarray*}
&& P(h(\widehat \pi_n, \pi_0)   >  L \delta_n) \\
&& \le P \left(   \sup_{\pi \in \mathcal M: h(\pi, \pi_0)  > L \delta_n}  \left\{\int_{\{\pi _0 < \tau_n\}} \frac{\pi - \pi_0}{\pi + \pi_0}  d(\mathbb P_n - \mathbb P)  -  \frac{1}{2} h^2( \pi, \pi_0) \right \}  \ge 0 \right)   \\
&&   + \:  P \left(  \sup_{\pi \in \mathcal M: h(\pi, \pi_0)  > L \delta_n}  \left\{\int_{\{\pi_0 \ge \tau_n\}} \frac{\pi - \pi_0}{\pi + \pi_0}  d(\mathbb P_n - \mathbb P)  -  \frac{1}{2} h^2( \pi, \pi_0) \right \}  \ge 0  \right) \\
&& =: P_1  + P_2.
\end{eqnarray*}
Next, we will upper bound the probabilities $P_1$ and $P_2$.  We have
that
\begin{eqnarray*}
  \int_{\{\pi_0 < \tau_n\}} \frac{\pi - \pi_0}{\pi + \pi_0}  d(\mathbb P_n - \mathbb P) &=  &  \int \mathds{1}_{\{\pi_0 < \tau_n\}} \frac{\pi - \pi_0}{\pi + \pi_0}  d(\mathbb P_n - \mathbb P)\\
                                                                                        & = & \int \mathds{1}_{\{\pi_0 < \tau_n\}}   d(\mathbb P_n - \mathbb P)  + \int \mathds{1}_{\{\pi_0 < \tau_n\}} \frac{2 \pi_0}{\pi + \pi_0}  d\mathbb P \\
                                                                                        && \ - \int \mathds{1}_{\{\pi_0 < \tau_n\}} \frac{2 \pi_0}{\pi + \pi_0}  d\mathbb P_n
\end{eqnarray*}
and hence
\begin{eqnarray*}
  \int_{\{\pi_0 < \tau_n\}} \frac{\pi - \pi_0}{\pi + \pi_0}  d(\mathbb P_n - \mathbb P) &  \le  &    \int \mathds{1}_{\{\pi_0 < \tau_n\}}d(\mathbb P_n - \mathbb P)+ 2 \int \mathds{1}_{\{\pi_0 < \tau_n\}} \frac{\pi_0}{\pi + \pi_0}  d\mathbb P  \\
                                                                                        & = &  \int \mathds{1}_{\{\pi_0 < \tau_n\}}d(\mathbb P_n - \mathbb P) +  2 \sum_{k \in \mathbb N}  \pi_0(k) \mathds{1}_{\{\pi_0(k) < \tau_n\}} \frac{\pi_0(k)}{\pi(k) + \pi_0(k)}  \\
                                                                                        & \le & \left \vert  \int \mathds{1}_{\{\pi_0 < \tau_n\}}d(\mathbb P_n - \mathbb P) \right \vert + 2 \sum_{k \in \mathbb N}  \pi_0(k) \mathds{1}_{\{\pi_0(k) < \tau_n\}}  
\end{eqnarray*}
using the fact that $\pi_0 \le \pi_0 + \pi$.  Now, applying the
definitions of $\tau_n$, $K_n$ and $\delta_n$, it follows that
\begin{eqnarray}\label{Ineqrev}
\left \vert \int_{\{\pi_0 < \tau_n\}} \frac{\pi - \pi_0}{\pi + \pi_0}  d(\mathbb P_n - \mathbb P)  \right \vert  & = & \left \vert  \int \mathds{1}_{\{\pi_0 < \tau_n\}}d(\mathbb P_n - \mathbb P) \right \vert   +  2 \sum_{k > K_n}  \pi_0(k) \notag \\
& \le & \left \vert  \int \mathds{1}_{\{\pi_0 < \tau_n\}}d(\mathbb P_n - \mathbb P) \right \vert  +2 \delta_n^2.
\end{eqnarray}
Furthermore,
$$
 \sup_{\pi \in \mathcal M: h(\pi, \pi_0)  > L \delta_n}  \left\{\int_{\{\pi _0 < \tau_n\}} \frac{\pi - \pi_0}{\pi + \pi_0}  d(\mathbb P_n - \mathbb P)  -  \frac{1}{2} h^2( \pi, \pi_0) \right \}  \ge 0 
$$
implies that
\begin{eqnarray*}
\sup_{\pi \in \mathcal M}  \left\{\int_{\{\pi _0 < \tau_n\}} \frac{\pi - \pi_0}{\pi + \pi_0}  d(\mathbb P_n - \mathbb P) \right \} &\ge & \sup_{\pi \in \mathcal M: h(\pi, \pi_0)  > L \delta_n}  \left\{\int_{\{\pi _0 < \tau_n\}} \frac{\pi - \pi_0}{\pi + \pi_0}  d(\mathbb P_n - \mathbb P) \right \} \\
&\ge & \frac{L^2}{2} \delta^2_n.
\end{eqnarray*}
Using the inequality established in (\ref{Ineqrev}), we can write that 
\begin{eqnarray*}
P_1  & \le & P \left(    \sup_{\pi \in \mathcal M}  \right \vert \int_{\{\pi _0 < \tau_n\}} \frac{\pi - \pi_0}{\pi + \pi_0}  d(\mathbb P_n - \mathbb P)  \left \vert    \ge \frac{L^2}{2}  \delta^2_n  \right) \\
& \le   & P\left(   \sqrt n \left \vert  \int \mathds{1}_{\{\pi_0 < \tau_n\}}d(\mathbb P_n - \mathbb P) \right \vert \ge  (L^2/2 - 2) \sqrt n \delta_n^2 \right)\\
& \le &  \frac{\sum_{k \in \mathbb N} \pi_0(k) \mathds{1}_{\{\pi_0(k) < \tau_n\}}}{(L^2/2 - 2)^2 n \delta_n^4} 
 \le  \frac{\delta_n^2}{(L^2/2 - 2)^2 n \delta_n^4}  =  \frac{1}{(L^2/2 - 2)^2 n \delta_n^2}.
\end{eqnarray*}
Now, we turn to finding an upper bound for $P_2$. This will be done using the so-called peeling device. First, note that $h(\pi,\pi_0) \le 1$ for all $\pi \in \mathcal M$. Set $S:=  \min \{s \in \mathbb N:  2^{s+1} L \delta_n \ge 1\}$. We have that
\begin{eqnarray*}
\{\pi: h(\pi, \pi_0) > L \delta_n \} = \bigcup_{s=0}^S     \{\pi: 2^s L \delta_n < h(\pi, \pi_0) \le 2^{s+1} L \delta_n \}.
\end{eqnarray*}
Now, for $s =0, \ldots, S$, the event 
$$
 \sup_{\pi \in \mathcal M: 2^s L \delta_n < h(\pi, \pi_0) \le 2^{s+1} L \delta_n}  \left\{\int_{\{\pi _0 < \tau_n\}} \frac{\pi - \pi_0}{\pi + \pi_0}  d(\mathbb P_n - \mathbb P)  -  \frac{1}{2} h^2( \pi, \pi_0) \right \}  \ge 0 
$$
implies that
$$
 \sup_{\pi \in \mathcal M: 2^s L \delta_n < h(\pi, \pi_0) \le 2^{s+1} L \delta_n}  \left\{\int_{\{\pi _0 < \tau_n\}} \frac{\pi - \pi_0}{\pi + \pi_0}  d(\mathbb P_n - \mathbb P)   \right \}  \ge \frac{2^{2s} L^2 \delta_n^2}{2} 
$$
and hence
$$
 \sup_{\pi \in \mathcal M:  h(\pi, \pi_0) \le 2^{s+1} L \delta_n}  \left\{\int_{\{\pi _0 < \tau_n\}} \frac{\pi - \pi_0}{\pi + \pi_0}  d(\mathbb P_n - \mathbb P)   \right \}  \ge \frac{2^{2s} L^2 \delta_n^2}{2}. 
$$
Using the union bound, it follows that  
\begin{eqnarray*}
P_2 & \le  &  \sum_{s=0}^S P \left( \sup_{\pi \in \mathcal M: h(\pi, \pi_0)  \le 2^{s+1} L  \delta_n} \sqrt n \left \vert \int  \mathds{1}_{\{\pi_0  \geq \tau_n\}} \frac{\pi - \pi_0}{\pi + \pi_0}  d(\mathbb P_n - \mathbb P) \right\vert  \ge   \frac{1}{2} \sqrt n 2^{2s}  L^2  \delta_n^2    \right) \\
& = &  \sum_{s=0}^S  P \left(  \sup_{g \in \mathcal{G}_n(2^{s+1} L \delta_n)}  \vert \mathbb G_n g \vert  \ge   \frac{1}{2} \sqrt n 2^{2s} L^2  \delta_n^2   \right)  ,
\end{eqnarray*}
using property 4 of Lemma \ref{Prop1}. Here,
$\mathbb G_n f = \sqrt n (\mathbb P_n - \mathbb P) f$ is the standard
notation for the value of the empirical process at a function $f$ and
$\mathcal G_n(\delta)$ for a given $\delta > 0$ is as defined in
(\ref{Gndelta}). By the Markov's inequality, it follows that
\begin{eqnarray*}
P_2  &\le &  \sum_{s=0}^S  \frac{2 \mathbb E\left[\Vert \mathbb G_ n \Vert_{\mathcal{G}_n(2^{s+1} L \delta_n)}\right]}{\sqrt n 2^{2s}   L^2 \delta_n^2}, \  \ \ \textrm{with $\Vert \mathbb G_n \Vert_{\mathcal F} = \sup_{f \in \mathcal F}  \vert \mathbb G_n f \vert$}.
\end{eqnarray*}
Now, note that each element of the class $\mathcal{G}_n(2^{s+1}L \delta_n)$ is bounded from above by $1$. Furthermore, for any $g \in \mathcal{G}_n(2^{s+1}L \delta_n)$, we have that
\begin{eqnarray*}
\mathbb P g^2  =  \sum_{0 \le k \le K_n} \left(\frac{\pi(k)  - \pi_0(k)}{\pi(k)  + \pi_0(k)}\right)^2 \pi_0(k)  \le 4 \ 2^{2s+2} L^2 \delta^2_n ,
\end{eqnarray*}
using that $h(\pi, \pi_0) \le 2^{s+1} L  \delta_n$ and the following inequality (which is the same as inequality 4.4 from \cite{Patilea}):
\begin{eqnarray*}
\sum_{k \in \mathbb N} \left(\frac{2\pi(k)}{\pi(k) + \pi_0(k)} - 1 \right)^2 \pi_0(k) = \sum_{k \in \mathbb N} \left(\frac{\pi(k)  - \pi_0(k)}{\pi(k)  + \pi_0(k)}\right)^2 \pi_0(k) \le 4 h^2(\pi, \pi_0).    
\end{eqnarray*}
Thus, we are in the position to apply Lemma 3.4.2 of \cite{aadbook},
which together with Proposition \ref{EntropyBound} implies that for
some universal constant $C > 0$ and for
$n \ge N(t_0, \tilde \theta, \delta_0, \eta_0)$ defined in (\ref{N0})
\begin{eqnarray*}
&& \mathbb E\left[\Vert \mathbb G_ n \Vert_{\mathcal{G}_n(2^{s+1} L \delta_n)}\right] \\
&&  \le  C    \widetilde{J}_B(2^{s+1}L \delta_n, \mathcal{G}_n(2^{s+1} L  \delta_n), \mathbb P)   \left(1 +  \frac{\widetilde{J}_B(2^{s+1} L \delta_n, \mathcal{G}_n(2^{s+1} L \delta_n), \mathbb P)}{2^{2s+2} L^2 \delta^2_n \sqrt n} \right) \\ 
&& =  C \: 2^{s+1}L \delta_n \frac{27}{\log (1/t_0)^{3/2}} (\log n)^{3/2}  \cdot \Big(1 + \frac{ 2^{s+1}L \delta_n \frac{27}{\log (1/t_0)^{3/2}} (\log n)^{3/2}}{2^{2s+2} L^2 \delta^2_n \sqrt n} \Big) \\
&& =  C \: 2^{s+1}L \delta_n^2 \sqrt n \frac{27}{\log (1/t_0)^{3/2}} \Big(1 + \frac{27}{L 2^{s+1} \log (1/t_0)^{3/2} } \Big) \\
&& =  C \: \Big( 2^{s+1}L \delta_n^2 \sqrt n \frac{27}{\log (1/t_0)^{3/2}} + \delta_n^2 \sqrt n \frac{27^2}{\log (1/t_0)^{3}} \Big).
\end{eqnarray*}
With $D =  2 \cdot 27^2 C$, it follows that
\begin{eqnarray*}
P_2 & \le & \frac{D}{\log (1/t_0)^{3/2} L} \sum_{s=0}^S \frac{1}{2^{s}} + \frac{D}{\log (1/t_0)^{3} L^2} \sum_{s=0}^S \frac{1}{2^{2s}} \\
& \le & \frac{2 D}{L} \left( \frac{1}{\log (1/t_0)^{3/2} }  + \frac{1}{\log (1/t_0)^{3}} \right), \  \textrm{since $L > 2$} \\
& \le & \frac{2D}{L} \frac{1}{\log (1/t_0)^{3/2}} \left(1 +  \frac{1}{\log (1/t_0)^{3/2}}  \right).
\end{eqnarray*}
Finally, we obtain that for all
$n \ge N(t_0, \tilde \theta, \delta_0, \eta_0)$
\begin{eqnarray*}
P(h(\widehat \pi_n, \pi_0)  > L \delta_n)   \le \frac{1}{(L^2/2 - 2)^2 (\log n)^2}  +  \frac{2D}{L} \frac{1}{\log (1/t_0)^{3/2}} \left(1 +  \frac{1}{\log (1/t_0)^{3/2}}  \right).
\end{eqnarray*} 
The right-hand side vanishes as $L \to \infty$ and $n \to \infty$.
$\hfill \Box$

\subsection{Minimax lower bounds: Existing results}\label{sec:minimax}

The obtained convergence rate $(\log n)^{3/2}/\sqrt n$ for the NPMLE
in the Hellinger distance, although fast, prompts the question whether
the logarithmic factor can be removed. Finding minimax lower bounds is
one way of looking for a possible answer.  In this section, we will
discuss the recently derived minimax lower bounds obtained in
\cite{polyanskiy2021sharp} for mixtures of Poisson
distributions. Before giving these bounds, we start with a brief
description, re-casted in our notation, of the Bayesian estimation
problem considered in \cite{polyanskiy2021sharp} and how it relates to
the current paper. Let $\Theta \subseteq \mathbb R$ be some measurable
real set. For $\theta \in \Theta$ denote by $f_\theta$ some density
with respect with some $\sigma$-finite dominating measure equipping a
sample space $\mathcal X$. If $\theta \sim Q_0$, for a given prior
distribution $Q_0$ supported on $\Theta$, then based on a realization
$x$ of
$$X \sim \pi_{Q_0} = \pi_0= \int_{\Theta} f_\theta dQ_0(\theta) $$
the Bayes estimator of $\theta$ is given by the conditional mean 
\begin{eqnarray}\label{CondMean}
\widehat{\theta}_{Q_0}(x) =  \frac{\int_\Theta \theta f_\theta(x) dQ_0}{\int_\Theta f_\theta(x) dQ_0} = \frac{\int_\Theta \theta f_\theta(x) dQ_0(\theta)}{\pi_{Q_0}(x)}
\end{eqnarray}
and its associated Bayes risk is 
\begin{eqnarray}\label{mmse}
\text{mmse}(Q_0) =\mathbb E_{Q_0}\left [\left(\widehat{\theta}_{Q_0}(X) - \theta \right)^2 \right].
\end{eqnarray}
Let $X_1, \ldots, X_n$ be i.i.d. $\sim \pi_{Q_0}$ and
$\theta^n := (\theta_1, \ldots, \theta_n)$ the vector of the
corresponding (unobserved) parameters such that
$\theta_i, i=1, \ldots, n$ are i.i.d $ \sim Q_0$.  The main goal in
\cite{polyanskiy2021sharp} is to obtain sharp bounds of the optimal
total regret over a given collection priors $\mathcal Q$:
\begin{eqnarray*}
\text{TotRegret}_n(\mathcal Q) := \inf_{\widehat{\theta}^n}\sup_{Q \in \mathcal Q} \left \{  \mathbb E_Q \left[ \Vert \widehat{\theta}^n(X_1, \ldots, X_n) - \theta^n \Vert^2   \right]  - n \cdot \text{mmse}(Q)\right \}  
\end{eqnarray*}
where $\Vert \cdot \Vert$ denotes the Euclidean norm,
$\widehat{\theta}^n = (\widehat{\theta}_1, \ldots,
\widehat{\theta}_n)$ is an estimator of $\theta^n$ based on
$X_1, \ldots, X_n$, and $\text{mmse}(Q)$ is defined similarly as in
(\ref{mmse}) by replacing $Q_0$ with an element $Q \in \mathcal
Q$. Moreover, the authors focus on the case where $f_\theta$ is either
the density of $\mathcal{N}(\theta,1)$ with respect to Lebesgue
measure or that of $\text{Poisson}(\theta)$ with respect to the
counting measure. For obvious reasons, we shall restrict attention to
the latter.  Denote by $\mathcal Q_{[0, M]}$ the collection of all
distributions which are supported on $[0, M]$ for some given $M >
0$. Also, denote by $\mathcal{Q}_{\text{SubE}(s)}$ the collection of
all $s$-subexponential distributions on $[0, \infty)$ for some
$s > 0$, that is the set of distributions $Q$ such that
$Q([t, \infty)) \le 2 \exp(-t/s)$ for all $t > 0$. Note that the
elements of $\mathcal Q_{[0, M]}$ satisfy our assumption (A1) for the
Poisson kernel (in this case the convergence radius is $R = \infty$).
It follows from \cite[Theorem 2]{polyanskiy2021sharp} that for $n$
large enough
\begin{eqnarray}\label{regretM}
c_1 \left(\frac{\log n}{\sqrt n \log(\log n)}\right)^2  \le \frac{1}{n} \text{TotRegret}_n(\mathcal Q_{[0, M]}) \le c_2 \left(\frac{\log n}{\sqrt n \log(\log n)}\right)^2 
\end{eqnarray}
for some constants $0 < c_1  \le c_2$ which depend on $M$, and 
\begin{eqnarray}\label{regrets}
c_3 \left(\frac{(\log n)^{3/2}}{\sqrt n} \right)^2 \le \frac{1}{n} \text{TotRegret}_n(\mathcal{Q}_{\text{SubE}(s)}) \le c_4  \left(\frac{(\log n)^{3/2}}{\sqrt n} \right)^2 
\end{eqnarray}
for some constants $0 < c_3 \le c_4$ which depend on $s$.  Note that
for any collection $\mathcal Q$
\begin{eqnarray*}
  \frac{1}{n} \text{TotRegret}_n(\mathcal{Q}) &=& \text{Regret}_n(\mathcal Q)\\
                                              &:=&  \inf_{\widehat{\theta}}\sup_{Q \in \mathcal Q} \left \{  \mathbb E_Q \left[ \left(\widehat{\theta}(X_1, \ldots, X_n) - \theta\right)^2   \right]  - \text{mmse}(Q)\right \},  
\end{eqnarray*}
the individual regret associated with estimating \lq\lq one\rq\rq \ $\theta$; see \cite[Lemma 5]{polyanskiy2021sharp}. \\

To obtain the lower bounds in (\ref{regretM}) and (\ref{regrets}), the
key results in \cite{polyanskiy2021sharp} are Proposition 7, Lemma 11
and Lemma 12, which yield a more concrete version of Assouad's Lemma
(see e.g. \cite{yu1997}) for estimating the means of a Poisson
distribution in the context of the above Bayesian paradigm. The crux
of the matter is to make a judicious choice of the prior $Q_0$ and the
associated collection of perturbations around it so that they satisfy
some orthogonality property; see (65) in \cite[Lemma
11]{polyanskiy2021sharp}. To apply \cite[Proposition 7, Lemma 11 \&
Lemma 12]{polyanskiy2021sharp}, the authors choose $Q_0$ to be the
distribution of a $\text{Gamma}(\alpha, \beta)$ for some $\alpha > 0$
and $\beta >0$ which possibly depend on $n$. Note that $Q_0$ is the
conjugate prior for the Poisson kernel and that the corresponding
(marginal) mixed pmf $\pi_0 = \pi_{Q_0}$ is that of a generalized
negative Binomial with parameters $\beta/(1+\beta)$ and $\alpha$. The
most difficult part in the proof is to construct meaningful
perturbation functions around the prior $Q_0$. If $r$ is some bounded
function on $\Theta = [0, \infty)$, then perturbing $Q_0$ in the
direction of $r$ amounts to defining the distribution function
$Q_\delta$ such that
\begin{eqnarray*}
dQ_\delta  =  \frac{1}{1 + \delta \int r dQ_0} (1+ \delta r) dQ_0 
\end{eqnarray*}
for some small $\delta > 0$.  Then, it can be shown that
$\pi_{Q_\delta}$ is linked to $\pi_{Q_0} = \pi_0$ via the identity
\begin{eqnarray}\label{Pertpi}
\pi_{Q_\delta}  = \pi_0 \ K \left(\frac{1 + \delta r}{1 + \delta \int r dQ_0}  \right) =  \pi_{Q_0} \ \frac{1+\delta K r}{1 + \delta \int r dQ_0}
\end{eqnarray}
where $K$ is the integral operator which assigns to a bounded function
$g$ on $\Theta$ the image $K g$ such hat
\begin{eqnarray}\label{K}
[K g](x) : =  \frac{\int g(\theta) f_\theta(x) dQ_0(\theta)}{\pi_0(x)} = \mathbb E_{Q_0}[g(\theta)|X=x]
\end{eqnarray}
for $x \in \mathbb N$; i.e., $[Kg](x)$ is the conditional mean of $g$
given $X=x$. If $g$ is the identity function, then with some abuse of
notation
$$K\theta =  \mathbb E_{Q_0}[\theta|X=x]$$
which is nothing but the Bayes estimator, $\widehat{\theta}_{Q_0}$ if
the prior $Q_0$ were perfectly known (see also (\ref{CondMean})).
Using the Bayes rule and the identity in (\ref{Pertpi}) it follows
that the Bayes estimator of $\theta$ associated with the perturbation
$Q_\delta$ is given by
\begin{eqnarray}\label{FinalOp}
\widehat{\theta}_{Q_\delta}(x) = [K \theta](x) + \delta \ [K_1 r](x) + \delta^2 \ \frac{[Kr](x) [K_1 r](x)}{1+ [Kr](x)}
\end{eqnarray}
with $K_1 r: = K(\theta r) - (K\theta)\cdot (K r)$ and $K$ is the same operator defined in (\ref{K}). 

The identity in (\ref{FinalOp}) shows that the dependence of
$\widehat{\theta}_{Q_\delta} - \widehat{\theta}_{Q_0}$ on the
perturbation direction, $r$, is highly non-linear.  This makes
application of the Assouad's lemma very challenging. In fact,
construction of the collection of the relevant perturbations requires
finding the eigenbasis of the self-adjoint operator $K^* K$. Using
highly technical calculations, it is shown through several equations
that the elements of the eigenbasis can be expressed in terms of
generalized Laguerre polynomials $\{L^{\nu}_k\}_{k \ge 0}$ for some
$\nu \in (-1, \infty)$; see e.g. \cite[Chapter 5]{szeg1939orthogonal}
for a definition. For the lower bound in (\ref{regrets}), the authors
show that they can take the prior to be an Exponential distribution,
that is $\alpha =1$ and show that in this case $\nu = 0$. This means
that in this case the elements of the eigenbasis can be written
explicitly as functions of the usual Laguerre polynomials.

Now, we come to the main point of this section: Minimax lower bounds
in the Hellinger distance for estimating a mixture of Poisson
distributions.  For a given collection of distributions $\mathcal Q$
let us define as in \cite{polyanskiy2021sharp} the minimax risk in the
Hellinger sense
\begin{eqnarray*}
  \mathcal{R}_n(\mathcal Q):= \inf_{\widehat{\pi}_n} \sup_{Q \in \mathcal Q} \mathbb E_Q h^2(\widehat{\pi}_n, \pi_Q)  
\end{eqnarray*}
where the infimum is taken over all possible estimators
$\widehat \pi_n$ based on the observed data
$X_1, \ldots, X_n \stackrel{d}{=} \pi_Q$ where
$\pi_Q= \int_0^\infty f_\theta dQ $ and $f_\theta$ the pmf of a
Poisson distribution with mean $\theta$.

The construction of the system of orthonormal perturbations through
\cite[Lemma 11 \& Lemma 12]{polyanskiy2021sharp} could be used again
by the authors to show that for $n$ large enough
\begin{eqnarray}\label{minmaxM}
\mathcal{R}_n(\mathcal Q_{[0, M]}) \ge c_0  \frac{1}{n} \frac{\log(\log n) }{\log n},
\end{eqnarray}
for some constant $c_0 > 0$ which depends on $M$,  and that 
\begin{eqnarray}\label{minmaxs}
\mathcal{R}_n(\mathcal Q_{\text{SubE}(s)}) \ge c_1  \frac{\log n }{n}
\end{eqnarray}
for some constant $c_1 > 0$ which depends on $s$; see \cite[Theorem
21]{polyanskiy2021sharp}. Finding the lower minimax lower bounds for
estimating the mixture distribution is much easier for at least two
reasons: (a) The road is already paved thanks to the readily existing
collection of suitable perturbations, (b) the relationship between the
mixture distribution $\pi_{Q_0}$ and the resulting perturbation
$\pi_{Q_\delta}$, described by (\ref{Pertpi}), is less complex than
for the mean.

The minimax lower bound in (\ref{minmaxM}) shows that, under our
assumption (A1), the convergence rate of the NPMLE in the Hellinger
distance for estimating a mixture of Poisson distributions can not be
parametric. On the other hand,
$(\log(\log n)/\log n)^{1/2} < < (\log n)^{3/2}$, which prompts the
question whether the bound in (\ref{minmaxM}) is too small to be
attained by the NPMLE.

The lower bounds established in (\ref{minmaxM}) and (\ref{minmaxs})
are to the best of our knowledge the only results on minimax lower
bounds in the Hellinger distance for some sub-classes of Poisson
mixtures. The highly involved calculations and the special
construction of an appropriate collection of perturbations that yields
non-trivial minimax lower bounds give a hint that what worked here
would not suitable for mixtures of other PSDs. Therefore, it will be
necessary to study the specificity of each kernel in order to come up
with the right choice of $Q_0$ and the perturbation functions. Since
the prior $Q_0$ was chosen in \cite{polyanskiy2021sharp} to be the
conjugate prior of a Poisson distribution; i.e., a
$\text{Gamma}(\alpha, \beta)$, we conjecture that it is most likely
that a $\text{Beta}(a, b), a, b \in (0, \infty)$ prior will be the
appropriate prior for mixtures of Geometric and Negative Binomial
distributions. We intend in a future work to start with the Geometric
mixtures and investigate the minimax lower bounds in the Hellinger
distance for classes of compactly supported mixing distributions.

\section{Estimators with $n^{-1/2}$-consistency in the
  $\ell_p$-distance}
\label{sec:estimators}

It follows from the results obtained in the previous sections that,
under our assumptions (A1)-(A4), the NPMLE converges at a nearly
parametric rate in the Hellinger distance, and that this convergence
rate is not parametric, at least for mixtures of Poisson
distribution. The natural question to be asked is whether the NPMLE
converges at the parametric rate in the other distances, e.g. $\ell_p$
for $p \in [2, \infty]$ or $p \in [1, \infty]$.  This question is
still open. In fact, it is not at all straightforward to re-use the
obtained Hellinger-rate in a way that the logarithmic factor does not
contribute anymore in the $\ell_p$-rate.

In this section, we investigate other estimators of the true mixture
$\pi_0$, that are based on the NPMLE, and for which it is possible to
show convergence in the fully parametric rate of $1/\sqrt{n}$.

\subsection{Weighted least squares estimators}

We present here a family of weighted least squares estimators which we
prove to converge to $\pi_0$ at $1/\sqrt n$. We will assume throughout
this section that the Assumptions (A1)-(A4) hold.  In addition, we
shall need the important fact that in our setting, the true mixing
distribution $Q_0$ is identifiable.  This result is a consequence of
Proposition 1 of \cite{Patilea2005} since in our setting
$\sum_{k \in \mathbb K: k > 0} k^{-1} = \sum_{k \ge 1} k^{-1} =
\infty$ and the support of $Q_0$ is a compact subset of $[0, R)$.

Theorem \ref{2convRes} below gives a nearly parametric rate of the
empirical estimator in the sense of weighted mean squared errors with
weights inversely proportional to $\pi_0$ or $\widehat \pi_n$. As
noted in \cite{Lambert},
$$
n^{1/2}\left(\sum_{k \in \mathbb{N}} \frac{(\bar \pi_n(k)-\pi_{0}(k))^2}{\pi_{0}(k)} \right)^{1/2}
$$
diverges to $\infty$ as $n \to \infty$. Thus, the parametric rate
$1/\sqrt n$ cannot be expected here. Nevertheless, the rate is of
smaller order of $n^{-1/2 + \epsilon}$ for an arbitrarily small
$\epsilon > 0$ and this will be used to derive the parametric rate of
our weighted LSEs.  We would like to note that the proof of Theorem
\ref{2convRes} goes along the same lines of that of Proposition 3.1
(i) and (ii) of \cite{Lambert}. In the supplementary material, we give
this proof again for the sake of completeness. In our proof, we
provide additional details as to how to obtain an almost sure upper
bound for the ratio $\widehat{\pi}_n/\pi_0$, which is a very crucial
step in obtaining the desired rate; see Lemma 2.1 of the supplementary
material.

\medskip

\begin{theorem}\label{2convRes}
For any $\epsilon>0$, it holds that
\begin{eqnarray} \label{0-Norm}
n^{1/2-\epsilon}\left(\sum_{k \in \mathbb{N}} \frac{(\bar \pi_n(k)-\pi_{0}(k))^2}{\pi_{0}(k)} \right)^{1/2} = o_{\mathbb P}(1),
\end{eqnarray}
and
\begin{eqnarray} \label{N-Norm}
n^{1/2-\epsilon}\left(\sum_{k \in \mathbb{N}} \frac{(\bar \pi_n(k)-\pi_{0}(k))^2}{\widehat \pi_n(k)}\right)^{1/2}=  o_{\mathbb P}(1).
\end{eqnarray}
\end{theorem}

\smallskip

Recall that in our setting, $\pi_0$ satisfies
$\sum_{k \in \mathbb N} \sqrt{\pi_0(k)} < \infty$, a consequence of
Proposition \ref{empirical} above. Then, for all $\alpha \in [0, 1/2]$
we have that
\begin{eqnarray}\label{convalpha}
\left(\sum_{k \in \mathbb N}  \frac{(\bar{\pi}_n(k)  - \pi_0(k))^2}{\pi^{\alpha}_0(k)}\right)^{1/2}   = O_{\mathbb P}\left(\frac{1}{\sqrt n}\right).
\end{eqnarray}
In fact, for all $\alpha \in [0, 1/2]$ and $k \ge 0$, $\pi^{\alpha}_0(k)  \ge \sqrt{\pi_0(k)}$, and 
\begin{eqnarray*}
\mathbb E\left[\sum_{k \in \mathbb N}  \frac{(\bar{\pi}_n(k)  - \pi_0(k))^2}{\sqrt{\pi_0(k)}} \right]  & = &  \frac{1}{n}  \sum_{k \in \mathbb N}  \frac{\pi_0(k) (1- \pi_0(k))}{\sqrt{\pi_0(k)}}    \\
& \le &  \frac{1}{n}  \sum_{k \in \mathbb N}  \sqrt{\pi_0(k)}  =  O\left(\frac{1}{n} \right).
\end{eqnarray*}

\medskip

We continue with the following definition.

\medskip

\begin{definition}
  For $\alpha \in [0,1)$, the weighted LSE with weights
  $\widehat{\pi}^{-\alpha}_n(k), k \in \mathbb N$ is
\begin{eqnarray*}
\widetilde{\pi}_{n, \alpha}  = \textrm{argmin}_{\pi \in \mathcal M} \sum_{k \ge 0}  \frac{(\bar{\pi}_n(k)  -  \pi(k))^2}{\widehat{\pi}^\alpha_n(k)}.
\end{eqnarray*}
\end{definition}

\medskip

Next, we need to show that $\widetilde{\pi}_{n, \alpha}$ does indeed
exist. For $\alpha=0$, the proof is rather easy and
$\widetilde{\pi}_{n, 0}$ can be shown to exist for every sample size
$n \ge 1$. For $\alpha \in (0, 1)$, we will be only able to show that
$\widetilde{\pi}_{n, \alpha}$ exists with probability tending to $1$
as $n$ grows to $\infty$.  As a first step, we need to show the
following result.  For the sake of having the least cumbersome
notation, let us write for a given $\alpha \in (0,1)$ and a sequence
$x \equiv (x(k))_{k \ge 0} \in \mathbb R^{\mathbb N}$
\begin{eqnarray*}
\Vert x \Vert_{n,\alpha}   =  \left(\sum_{k \ge 0}  \frac{x^2(k)}{\widehat{\pi}^\alpha_n(k)}\right)^{1/2}.
\end{eqnarray*}
Also, let $\ell_{2, \alpha}(\mathbb N)  = \left \{ x \in \mathbb R^{\mathbb N}:  \Vert x \Vert_{n,\alpha} < \infty \right \}$.  Note that in the notation $\ell_{2, \alpha}(\mathbb N)$ the sample size $n$ was omitted but needs to be kept in mind. Note also that $\ell_{2, \alpha}(\mathbb N)$ depends also on $X_1, \ldots, X_n$ through $\widehat{\pi}_n$, which means that it contains random sequences. 

\medskip

\begin{proposition}\label{pi0inl2alpha}
Fix $\alpha \in [0, 1)$. Then, as $n \to \infty$, 
\begin{eqnarray*}
\pi_0  \in \ell_{2, \alpha}(\mathbb N)
\end{eqnarray*}
with probability tending to $1$.
\end{proposition}

\par \noindent A proof of Proposition \ref{pi0inl2alpha} can be found
in the supplementary material. In the sequel, let us write
\begin{eqnarray*}
Q_{n, \alpha}(\pi)   =   \sum_{k \ge 0}  \frac{(\bar{\pi}_n(k)  -  \pi(k))^2}{\widehat{\pi}^\alpha_n(k)}
\end{eqnarray*}
for $\pi \in \mathcal M$.  Proposition \ref{pi0inl2alpha} shows that
for $n$ large enough it makes sense to search for a minimizer of
$Q_{n, \alpha}$.  In the next result, we show that a minimizer exists
with increasing probability.

\bigskip

\begin{proposition}\label{existpitilde}
Let $\alpha \in (0, 1)$.  As $n \to \infty$, $Q_{n, \alpha}$ admits a unique minimizer with probability tending to $1$. In other words, the estimator $\widetilde{\pi}_{n, \alpha}$ exists with probability tending to $1$.

\end{proposition}

\medskip

\begin{proof} 
  Consider minimization of $Q_{n, \alpha}$ over the space
  $\mathcal{M} \cap \ell_{2, \alpha}(\mathbb N) $.  By Proposition
  \ref{pi0inl2alpha}, this space is not empty with probability tending
  to $1$.  Furthermore, it is a closed and convex subset of the
  Hilbert space $\ell_{2, \alpha}(\mathbb N)$.  By the projection
  theorem, we conclude that the convex criterion admits a
  minimizer. Uniqueness follows from strict convexity of
  $Q_{n, \alpha}$.
\end{proof}

\medskip

In the following result, we show that the weighted LSE's converge to
$\pi_0$ uniformly in $\alpha \in [0,1/2]$ at the $n^{-1/2}$-rate.

\medskip

\begin{theorem}\label{ratetildepialpha}
  Suppose that Assumptions (A1)-(A4) are satisfied. Then, for every
  $\alpha \in [0, 1/2]$ it holds that
\begin{eqnarray*}
\sup_{\alpha \in [0,1/2]}  \left(\sum_{k \ge 0}  \frac{(\widetilde{\pi}_{n, \alpha}(k)  - \pi_0(k))^2 }{\widehat{\pi}^\alpha_{n}(k)} \right)^{1/2}  = O_{\mathbb P}\left(\frac{1}{\sqrt n}\right)
\end{eqnarray*}
and
\begin{eqnarray*}
\sup_{\alpha \in [0,1/2]}  \left(\sum_{k \ge 0} \vert \widetilde{\pi}_{n, \alpha}(k)- \pi_0(k)  \vert^p\right)^{1/p}= O_{\mathbb P}\left(\frac{1}{\sqrt n}\right).
\end{eqnarray*}
for all $ 2 \le p \le \infty$.
\end{theorem}

\medskip

Intuitively, one can expect the estimators
$\widetilde{\pi}_{n, \alpha}, \ \alpha \in [0, 1/2]$, to be finite
mixtures like the NPMLE.  However, this seems to be harder to show
than expected. The main obstacle is the fact that we minimize the
criterion $Q_{n, \alpha}$ over the space of mixtures of PSDs whose
mixing positive measure has total mass equal to $1$.  Nevertheless, we
conjecture that for any $\alpha \in [0,1/2]$ the estimator
$\widetilde{\pi}_{n, \alpha}$ has no more than $N = X_{(n)} +1$
components, with $X_{(n)} = \max_{1 \le i \le n } X_i$.  To support
our conjecture, we present in the appendix of the supplementary
material a proof that the minimizer of $Q_{n, \alpha}$ over the bigger
space of mixtures of PSDs with a mixing measure that is positive and
finite has indeed finitely many components whose number cannot exceed
$N$.

\subsection{A hybrid estimator}\label{hyb}
Before we describe the hybrid estimator, consider the following two
mixtures of Poisson, which are also used in the two first simulation
settings of Figure \ref{fig:sim1}:
\begin{eqnarray}\label{Ex1}
\pi_0(k)  = \frac{4}{9} \frac{e^{-1}}{k!}  + \frac{5}{9} \frac{e^{-2} 2^k}{k!}
\end{eqnarray}
and 
\begin{eqnarray}\label{Ex2}
\pi_0(k)  = \frac{1}{5(1- (4/5)^8)} \sum_{j=1}^8 0.8^{j-1} \frac{e^{-j} j^k }{k!}  
\end{eqnarray}
for $k \in \{0,1,2, \ldots \}$. For both scenarios, we computed the
empirical estimator $\bar{\pi}_n$ and the NPMLE
$\widehat{\pi}_n$. Using 100 replications, we report in Table
\ref{tableratios} and Table \ref{tableratios2} the average and median
of the ratio of the Hellinger, $\ell_1$ and $\ell_2$ distances between
$\pi_0$ and $\bar{\pi}_n$ over the same distances between $\pi_0$ and
$\widehat{\pi}_n$ over the region
$\{X_{(n)} +1, X_{(n)} +2, \ldots \}$. Note that on this region the
empirical estimator assigns $0$ probability while the NPMLE gives
non-zero weights. More explicitly, the ratios are given by
\begin{eqnarray*}
 \frac{\sqrt{\sum_{k =X_{(n)} +1}^\infty \pi_0(k)}}{\sqrt{\sum_{k =X_{(n)} +1}^\infty (\sqrt{\widehat{\pi}_n(k)} -  \sqrt{\pi_0(k)})^{2}}} 
\end{eqnarray*}
for the Hellinger distance, and
\begin{eqnarray*}
 \frac{\sum_{k =X_{(n)} +1}^\infty \pi_0(k)}{\sum_{k =X_{(n)} +1}^\infty \vert \widehat{\pi}_n(k) -  \pi_0(k) \vert}, \ \  \frac{\sqrt{\sum_{k =X_{(n)} +1}^\infty \pi^2_0(k)}}{\sqrt{\sum_{k =X_{(n)} +1}^\infty (\widehat{\pi}_n(k) -  \pi_0(k))^2}}
\end{eqnarray*}
for the $\ell_1$- and $\ell_2$-distances respectively. 

\bigskip
\begin{table}[!htb]
\begin{center}
\begin{tabular}{c||c|c|c}
\hline
    Sample size  & \multicolumn{3}{c}{(mean, median)} \\
    $n$   &   &   & \\
    & Hellinger & $\ell_1$ & $\ell_2$ \\
    \midrule
    $100$     & (5.45,3.68)    & (3.15,1.89)  & (3.53,2.03)  \\
     $1000$     &  (9.74, 4.16)   & (5.59, 2.37)     &  (6.20, 2.50)   \\
    $10000$     &  (25.80,5.80) & (13.21,3.17)  & (12.85,  3.30)  \\  
    \hline
\end{tabular}
\end{center}
\caption{Mean and median for the ratios of the estimation error at the
  tail using the Hellinger, $\ell_1$- and $\ell_2$-distances for
  $\pi_0$ in (\ref{Ex1}).}
  \label{tableratios}
\end{table}

\begin{table}[!htb]
\begin{center}
\begin{tabular}{c||c|c|c}
\hline
    Sample size  & \multicolumn{3}{c}{(mean, median)} \\
    $n$   &   &   & \\
    & Hellinger & $\ell_1$ & $\ell_2$ \\
    \midrule
    $100$     & (5.08, 3.36)    & (3.07, 2.06)  & (3.47, 2.14)  \\
     $1000$     &  (8.88, 3.93)   & (5.33, 2.33)     &  (6.18, 2.51)   \\
    $10000$     &  (10.85, 4.42)  & (6.31, 2.45)  & (7.20, 2.71)  \\  
    \hline
\end{tabular}
\end{center}
\caption{Mean and median for the ratios of the estimation error at the
  tail using the Hellinger, $\ell_1$- and $\ell_2$-distances for
  $\pi_0$ in (\ref{Ex2}).}
  \label{tableratios2}
\end{table}

The results shown in \ref{tableratios} and \ref{tableratios2} show
that for the first and second examples of Poisson mixtures considered
in this paper, the NPMLE does much better than the empirical estimator
at the tail. We believe that this is one reason why the NPMLE has an
overall superior performance than that of the empirical estimator; see
the simulation results in Figure \ref{fig:sim1} and Figure
\ref{fig:sim2}. Thus, although the empirical estimator has excellent
asymptotic properties (pointwise asymptotic normality, global
$1/\sqrt n$-consistency in $\ell_p$ for all $p \in [2, \infty]$ or
even $p \in [1, \infty]$ as shown in Proposition \ref{empirical}), it
does not cope well with missing information at the tail for
distributions with infinite support.

The hybrid estimator is defined to show that $1/\sqrt n$ can be
achieved using a very simple approach, which has the additional
advantage of not assigning a zero-weight at the tail. From a purely
theoretical perspective, we believe that the hybrid estimator might
bring some good insights into future investigations of the convergence
rate of $\widehat \pi_n$ in the $\ell_p$-distances.

We continue next with the following proposition.

\medskip

\begin{proposition}\label{empirical}
Let $\pi_0(k) = \int_\Theta f_{\theta}(k) dQ_0(\theta), k \in \mathbb N,$ as defined above, and let $\bar{\pi}_n$ be the empirical estimator of $\pi_0$ based on i.i.d. random variables $X_1, \ldots, X_n \sim \pi_0$. Then, it holds that 
\begin{eqnarray*}
\sum_{k \in \mathbb N} \sqrt{\pi_0(k)} < \infty. 
\end{eqnarray*}
Moreover,  for all $p \in [1, \infty]$, we have that
\begin{eqnarray*}
\ell_p(\bar \pi_n, \pi_0)   =  O_{\mathbb P}(1/\sqrt{n}).
\end{eqnarray*}

\end{proposition}
 
\begin{proof}
  Let $U$ and $W$ the same constants defined in (\ref{UV}) and
  (\ref{W}) respectively. Using Property 1 and 2 of Lemma \ref{Prop1}
  (for simplicity, we denote again $\max(U,W)$ by $W$), it follows
  that
\begin{eqnarray*}
\int_{\Theta}  f_\theta(k)  dQ_0(k)  \le f_{\tilde{\theta}}(k)  \int_{\Theta} dQ_0(\theta)  =   f_{\tilde{\theta}}(k)
\end{eqnarray*}
and $b_k  \le (t_0/\tilde{\theta})^{k-W} b_W$ for all $k \ge W$.  Hence,
\begin{eqnarray*}
\sum_{k \ge W} \sqrt{\pi_0(k)}  &\le & \sum_{k \ge W} \frac{\sqrt{b_k} \tilde{\theta}^{k/2}}{\sqrt{b(\tilde{\theta})}} \\
& \le &  \sum_{k \ge W} \frac{\sqrt{b_W}}{\sqrt{b(\tilde{\theta})}} \left(\frac{t_0}{\tilde{\theta}}\right)^{(k-W)/2} \tilde{\theta}^{k/2}  =  C  \sum_{k \ge W} t^{(k-W)/2}_0 = \frac{C}{1-\sqrt{t_0}}< \infty,
\end{eqnarray*}
where the constant $C > 0$ depends only on $W$, $b_W$, $\tilde{\theta}$ and the value $b(\tilde{\theta})$. This proves the first assertion.

To show the second assertion, note that
$\vert \bar \pi_n(k) - \pi_0(k) \vert \ge \vert \bar \pi_n(k) -
\pi_0(k) \vert^p$ for all $p \ge 1$ and for all $k \in
\mathbb{N}$. Hence, it is enough to show the result for $p =1$.  By
Fubini's theorem and Jensen's inequality, we have
\begin{eqnarray*}
\mathbb{E}\Big[\sum_{k \in \mathbb N} \vert \bar \pi_n(k) - \pi_0(k) \vert\Big]
& \le &    \sum_{k \in \mathbb N}  \sqrt{\mathbb{E}\Big[(\bar \pi_n(k) - \pi_0(k))^2\Big]}
=  \sum_{k \in \mathbb N} \sqrt{\frac{1}{n} \pi_0(k) (1-\pi_0(k))} \\
& = & \frac{1}{\sqrt{n}}  \sum_{k \in \mathbb N} \sqrt{\pi_0(k) (1-\pi_0(k))}
 <      \frac{1}{\sqrt{n}}  \sum_{k \in \mathbb N} \sqrt{\pi_0(k)}.
\end{eqnarray*}
We conclude the proof by using Markov's inequality and the first assertion.
\end{proof}

\medskip

In the following proposition we introduce the hybrid estimator and
prove its convergence at the $n^{-1/2}$-rate.

\medskip

\begin{proposition}\label{hybrid}
Let $\widehat{\pi}_n$ denote again the NPMLE of $\pi_0 \in \mathcal{M}$.  Let $\tilde{K}_n > 0$ be the smallest integer   $K$ such that 
\begin{eqnarray*}
\sum_{k > K}  \widehat \pi_n(k)   \le \frac{1}{(\log n)^{3}}.
\end{eqnarray*}
Then, the hybrid estimator $\widetilde{\pi}_n$ defined as 
\begin{eqnarray*}
\widetilde{\pi}_n(k)  = \bar \pi_n(k) \mathds{1}_{\{k \le \tilde{K}_n\}}   +  \widehat \pi_n  \mathds{1}_{\{k > \tilde{K}_n\}}
\end{eqnarray*}
satisfies that
\begin{eqnarray*}
\ell_p(\widetilde{\pi}_n, \pi_0)   =  O_{\mathbb P}(1/\sqrt{n})
\end{eqnarray*}
for all $p \in [1, \infty]$.
\end{proposition}

\begin{proof}
It is enough to show that the result holds for $p =1$. We have
\begin{eqnarray}\label{zerleg}
  \vert \widetilde{\pi}_n(k)  -  \pi_0(k) \vert \le  \vert \bar{\pi}_n(k)  - \pi_0(k) \vert \mathds{1}_{\{k \le \tilde{K}_n\}}   +   \vert \widehat{\pi}_n(k)  - \pi_0(k) \vert \mathds{1}_{\{k > \tilde{K}_n\}}.
\end{eqnarray}
Also we can write
\begin{eqnarray*}
  && \sum_{k > \tilde{K}_n}\vert \widehat{\pi}_n(k)  - \pi_0(k) \vert 
     =  \sum_{k > \tilde{K}_n} \big \vert \sqrt{\widehat{\pi}_n(k)}  - \sqrt{\pi_0(k)}  \big \vert  \big(\sqrt{\widehat{\pi}_n(k)} +  \sqrt{\pi_0(k)} \big)  \\
  &&  \le \left [ \sum_{k > \tilde{K}_n} \big(\sqrt{\widehat{\pi}_n(k)}  - \sqrt{\pi_0(k)}  \big)^{2} \right ]^{1/2} \cdot  \left [\sum_{k > \tilde{K}_n} \big(\sqrt{\widehat{\pi}_n(k)} +  \sqrt{\pi_0(k)} \big)^2\right ]^{1/2}, \\
  &&  \  \  \  \  \textrm{using the Cauchy-Schwarz inequality}  \\
  & & \le  \sqrt 2  h(\widehat \pi_n, \pi_0)  \cdot  \sqrt 2   \left [\sum_{k > \tilde{K}_n} \widehat {\pi}_n(k)  + \sum_{k > \tilde{K}_n}  \pi_0(k)  \right]^{1/2},  \\
  &&  \  \  \ \ \textrm{using the fact that $(a+b)^2 \le 2 (a^2 + b^2)$}  \\
  & &  \le  2  h(\widehat \pi_n, \pi_0)  \cdot  \left [\sum_{k > \tilde{K}_n}  \vert \widehat {\pi}_n(k)  - \pi_0(k) \vert +  2  \sum_{k > \tilde{K}_n}  \widehat \pi_n(k)  \right ]^{1/2}  \\
  & &  \le  O_{\mathbb P}((\log n)^{3/2}/\sqrt{n})  \cdot  \left(O_{\mathbb P}((\log n)^{3/2}/\sqrt{n})   + (\log n)^{-3}  \right)^{1/2}  =  O_{\mathbb P}(1/\sqrt{n}),
\end{eqnarray*}
where we have applied our convergence result for the NPMLE, obtained
in Theorem \ref{Rate}. We conclude by using Proposition
\ref{empirical}, which implies that the sum over $k$ in the first term
of (\ref{zerleg}),
$\sum_{k \in \mathbb N} \vert \bar{\pi}_n(k) - \pi_0(k) \vert
\mathds{1}_{\{k \le \tilde{K}_n\}}$, is $ O_{\mathbb P}(1/\sqrt n)$.
\end{proof}

\medskip

As noted above, a key disadvantage of the empirical estimator is that
it puts zero mass in the tail. The following proposition shows that
this does not happen with the hybrid estimator with probability
tending to 1 as the sample size $n \to \infty$. To show this, we first
need the following proposition.

\medskip

\begin{proposition}\label{Nonzero}
  Let $\tilde{K}_n$ be defined as in Proposition \ref{hybrid}. Then,
  it holds that
\begin{eqnarray*}
(\tilde{K}_n +1) (1- \pi_0(\tilde{K}_n))^n = o_{\mathbb P}(1).
\end{eqnarray*}
Moreover, we have
\begin{eqnarray*}
\lim_{n \to \infty}  P\left( \min_{ 0 \le k \le \tilde{K}_n}   \bar{\pi}_n(k)  > 0  \right)   = 1.
\end{eqnarray*}
\end{proposition}

\medskip

\begin{proof}
  We only prove the second assertion; the proof of the first claim can
  be found in the supplementary material. It is clear that
  $n \bar{\pi}_n(k) \sim \text{Bin}(n, \pi_0(k))$ for any fixed
  $k \in \mathbb N$.  Then, for $n$ large enough
\begin{eqnarray*}
P\left( \min_{ 0 \le k \le \tilde{K}_n}   \bar{\pi}_n(k)  > 0  \right)  & =  &  1- P \left( \exists \ k \in \{0, \ldots, \tilde{K}_n \}:   \bar{\pi}_n(k) = 0  \right)   \\
& \ge &   1 -   \sum_{k =0}^{\tilde{K}_n}  P(\bar{\pi}_n(k) = 0)   
=   1 -   \sum_{k =0}^{\tilde{K}_n}  (1- \pi_0(k))^n  \\
& \ge   & 1- (\tilde{K}_n +1)  (1-\pi_0(\tilde{K}_n))^n,
\end{eqnarray*}
where in the last step we applied item 4 of Lemma \ref{Prop1}. We conclude by using the first assertion.
\end{proof}

\medskip

\section{Computations: Simulations and real data application}
\label{sec:computation}

\subsection{The algorithm}

Different algorithms to compute the NPMLE $\widehat \pi_n$ were
already proposed in the literature; we can refer here for example to
\cite{wang-2007} and \cite{wang-2010}. In the following, we describe
the algorithm used to compute $\widetilde{\pi}_{n, \alpha}$ for a
given $\alpha \in (0,1)$.  In the sequel, fix such an $\alpha$. The
objective function to be minimized takes the form
\begin{align*}
  D(Q) = \sum_{k=0}^\infty w_n(k) [\pi(k; Q) - \bar{\pi}_n(k) ]^2,
\end{align*}
where  $\pi(k; Q) = \int_{\Theta} f_\theta(k) dQ(\theta)$ with $Q$ a discrete mixing distribution with support points $\theta_j$ and weights $p_j$ for $j =1, \ldots, m$ and $m \in \mathbb N$. Additionally, $w_n(k) =   [\widehat{\pi}_n(k)]^{-\alpha}$.  

For numerical computation, the infinite sum of $D(Q)$ can first be
replaced with a finite one as follows:
\begin{align*}
  D_K(Q) = \sum_{k=0}^K \left[ \sum_{j=1}^m p_j w_n(k)^{\frac 1
  2} f_k(\theta_j) - w_n(k)^{\frac 1 2} \bar{\pi}_n(k)\right]^2,
\end{align*}
where $K < \infty$ can be chosen sufficiently large for computing
accuracy. Write
$$s_k = w_n(k)^{\frac 1 2} (f_{\theta_1}(k), \dots,
f_{\theta_m}(k))^T, $$
$$S = (s_0, \dots, s_K)^T,$$
$$b = (w_n(0)^{\frac 1 2} \bar{\pi}_n(0), \dots, w_n(K)^{\frac 1 2}
\bar{\pi}_n(K))^T,$$ 
and 
$$p = (p_1, \ldots, p_m)^T.$$
Then
\begin{align}
  \label{eq:DKQ-problem}
  D_K(Q) = \Vert S p - b \Vert^2,
\end{align}
subject to $p_j \ge 0$ for all $j = 1, \dots, m$ and
$\sum_{j=1}^m p_j = 1$. For any probability measure $Q$ with fixed
support points, the optimal mixing proportions $p_1, \dots, p_m$ can
be found by solving the constrained least squares problem
(\ref{eq:DKQ-problem}). Note that some $p_j$ may turn out to be
exactly equal to $0$. This can be resolved using the same approach
described in \cite{wang-2007} and \cite{wang-2010}.

The gradient function is given by
\begin{eqnarray*}
  d(\theta; Q)  &= &  \frac{\partial D_K[(1-\epsilon)Q + \epsilon \delta_\theta]}{\partial \epsilon}|_{\epsilon = 0} \\
  & =   & 2 \sum_{k=0}^K w_n(k) [\pi(k; Q) - \bar{\pi}_n(k) ]  [f_k(\theta) - \pi(k; Q)].
\end{eqnarray*}

Choose an initial discrete mixing distribution $Q_{(0)}$ with a finite
number of support points, and set $s = 0$.  Then, the algorithm goes
through the following steps.

\begin{enumerate}
  
\item Find all local minima of the gradient function
  $d(\theta; Q_{(s)})$. \label{item:1}
  
\item Expand the support set of $Q_{(s)}$ with the above local minima.
  Assign mass $0$ to each new support point. Denote this mixing
  distribution by $Q_{(s+\frac 1 2)}$, which is equivalent to
  $Q_{(s)}$.
  
\item Solve problem (\ref{eq:DKQ-problem}) for $p$ and then discard
  the support points with mass $0$. Denote the resulting mixing
  distribution by $Q_{(s+1)}$.

\item If $Q_{(s)} - Q_{(s+1)} \le$ tolerance, then stop; otherwise,
  set $s = s+1$ and return to \ref{item:1}.

\end{enumerate}

\subsection{Simulation studies}
\label{sec:sim}

To numerically investigate the asymptotic behavior of the estimators,
we carry out some simulation studies using the fast algorithm
described above. Mixtures of three component distribution families are
considered: the Poisson, Geometric and Negative Binomial
distributions. The sample size is set to $n = 100, 1000, \dots,
10^9$. The following 8  estimators are studied: The empirical
estimator (Emp), the maximum likelihood estimator (MLE), the hybrid
estimator (Hyb), and five weighted least squares estimators (LSE)
(with $\alpha = 0.0, 0.2, 0.4, 0.6, 0.8$, respectively). Three
performance measures scaled by $\sqrt n$ are calculated: $\sqrt n h$,
$\sqrt n \ell_2$ and $\sqrt n \ell_1$.

\begin{figure*}[!tbh]
  \centering
  \includegraphics[width=.48\textwidth]{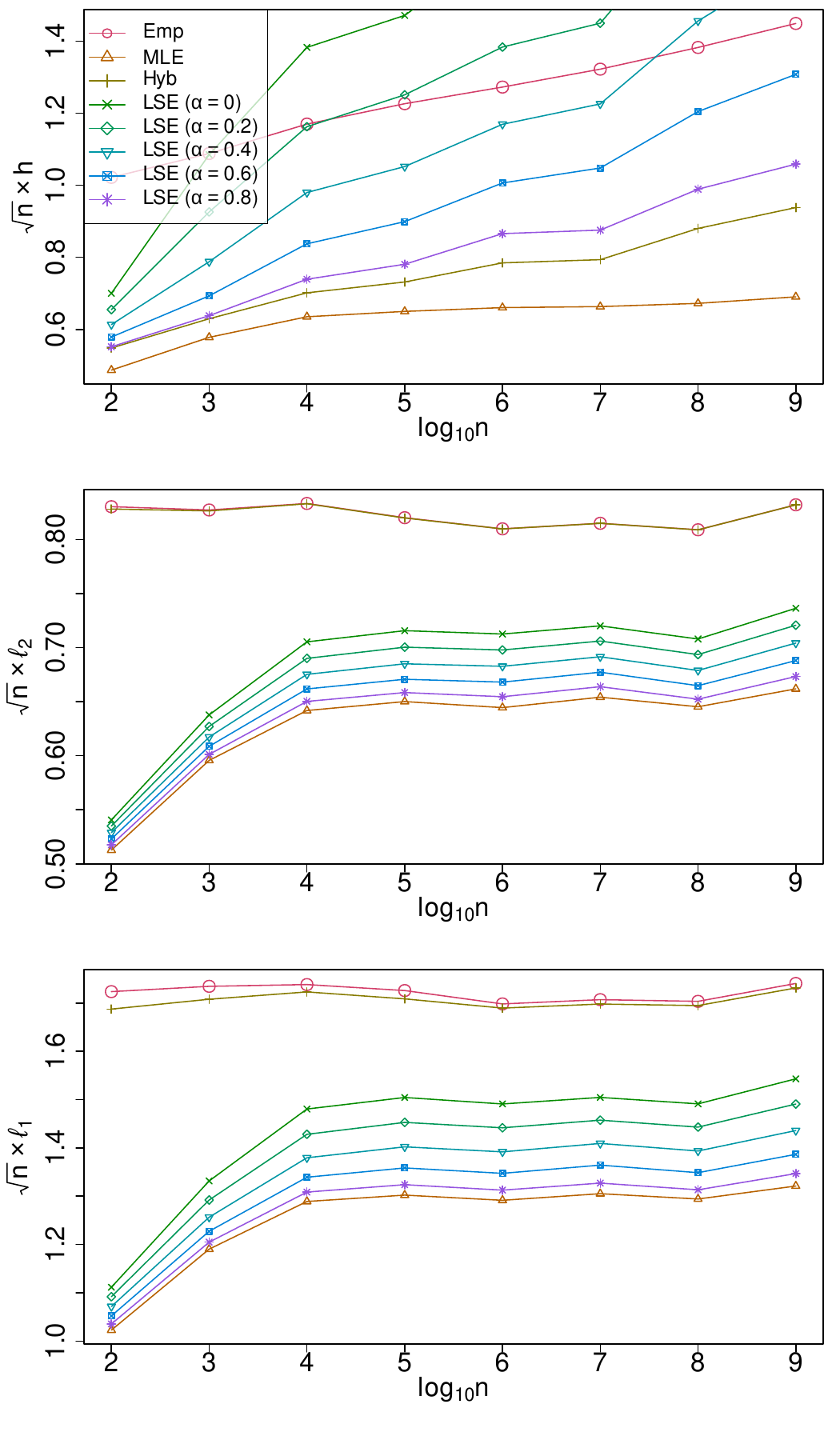}
  \includegraphics[width=.48\textwidth]{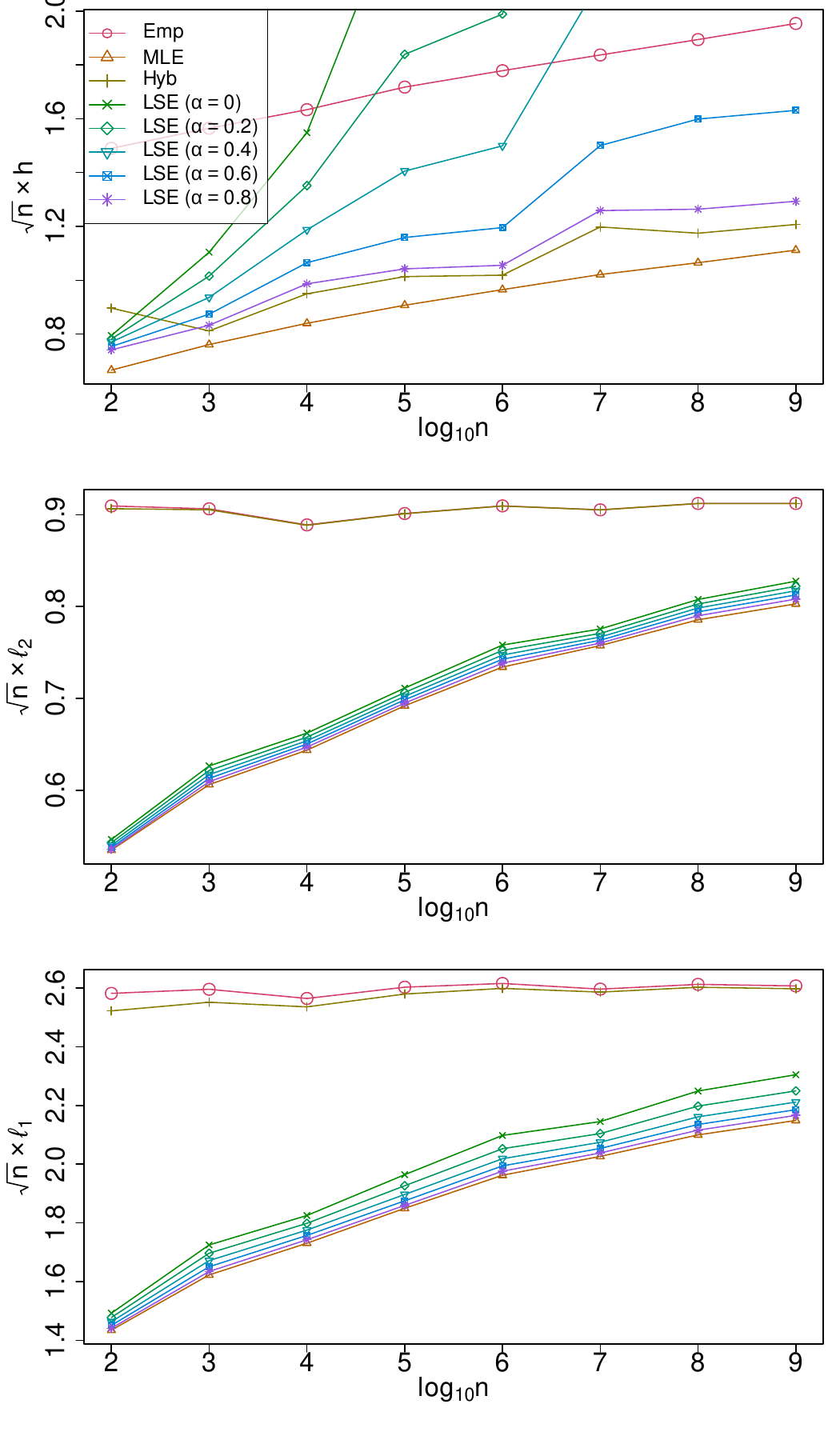} 
  \caption{Poisson mixtures: Finite with $m = 2$ components (left); with $m = 8$ components (right).}
  \label{fig:sim1}
\end{figure*}

\begin{figure*}[!tbh]
  \centering
  \includegraphics[width=.48\textwidth]{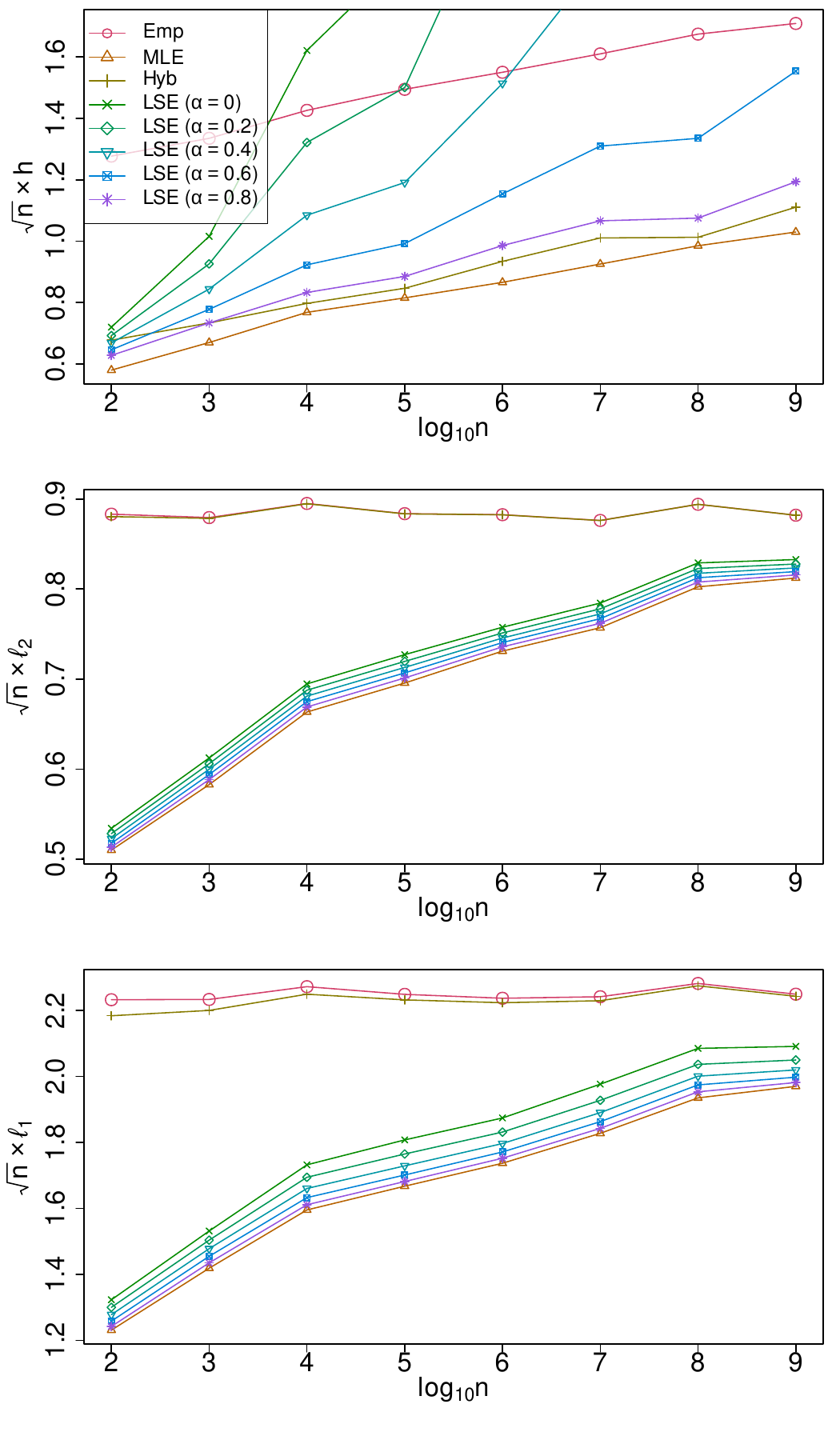}
  \includegraphics[width=.48\textwidth]{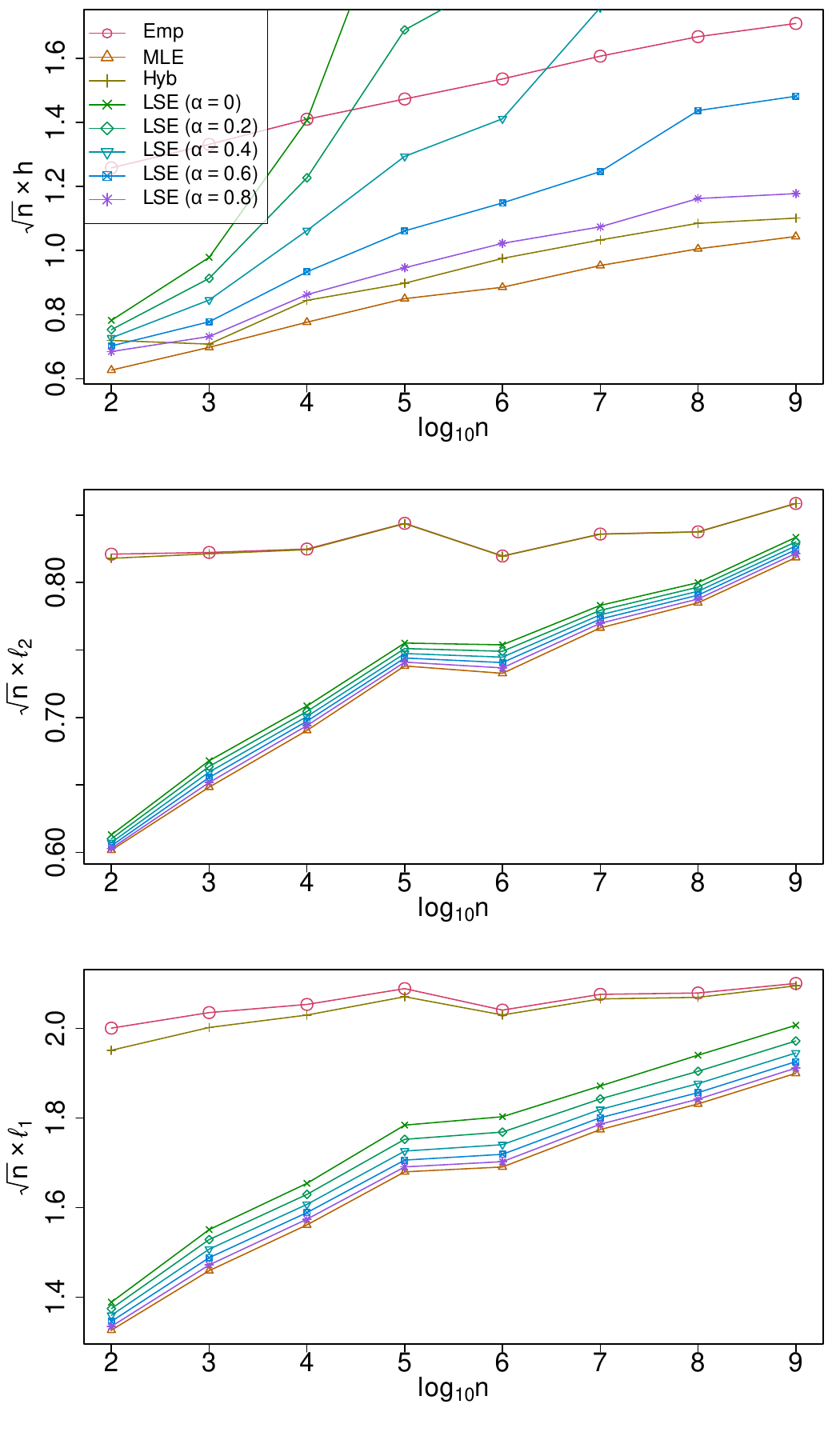} 
  \caption{Poisson mixtures: With the mixing distribution $\mathrm{U}(0.2, 5)$
    (left); with the mixing distribution
    $\frac 1 3 \delta_0 + \frac 2 3 \mathrm{U}(0.2,5)$ (right).}
  \label{fig:sim2}
\end{figure*}

\begin{figure*}[!tbh]
  \centering
  \includegraphics[width=.48\textwidth]{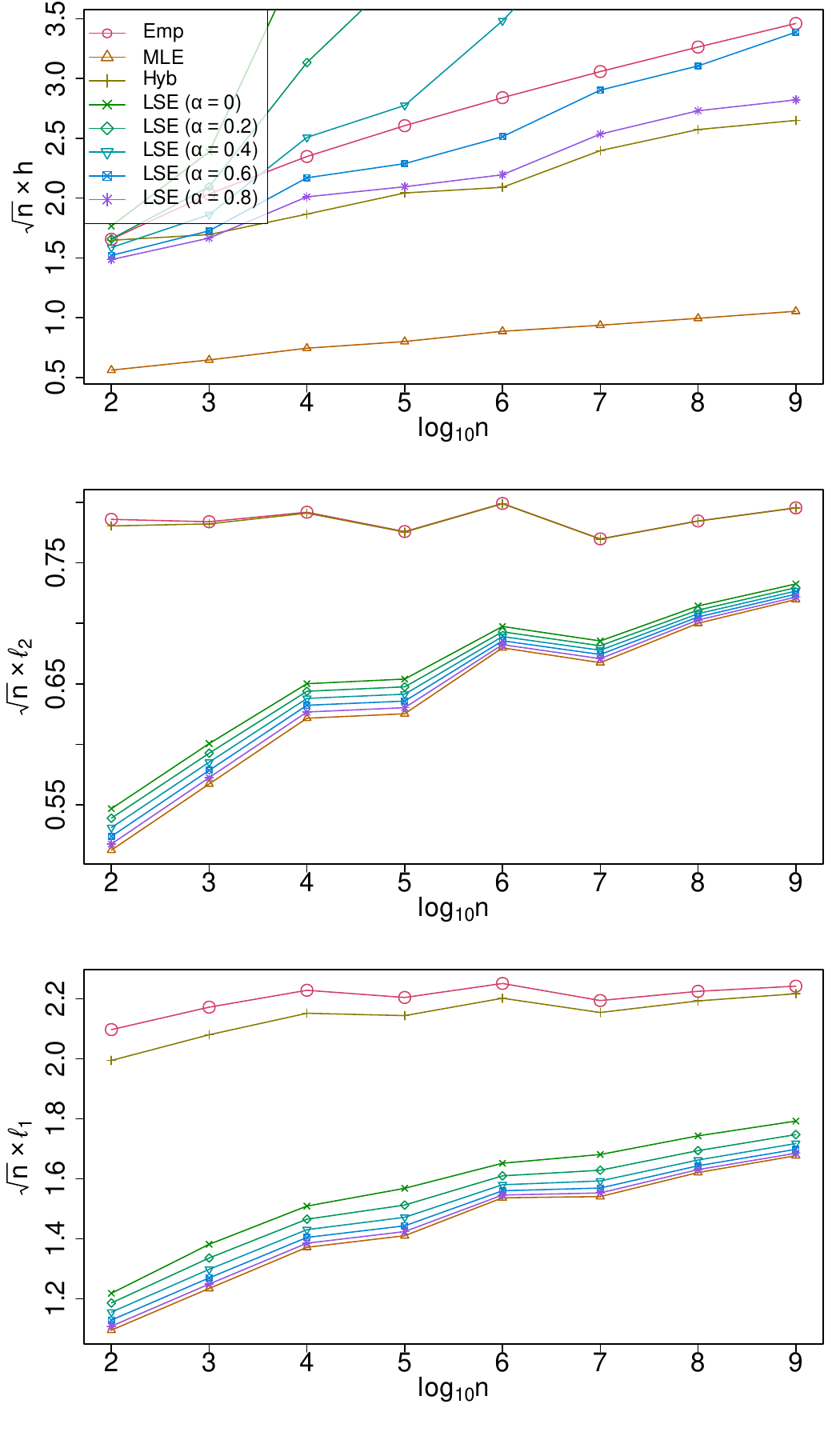}
  \includegraphics[width=.48\textwidth]{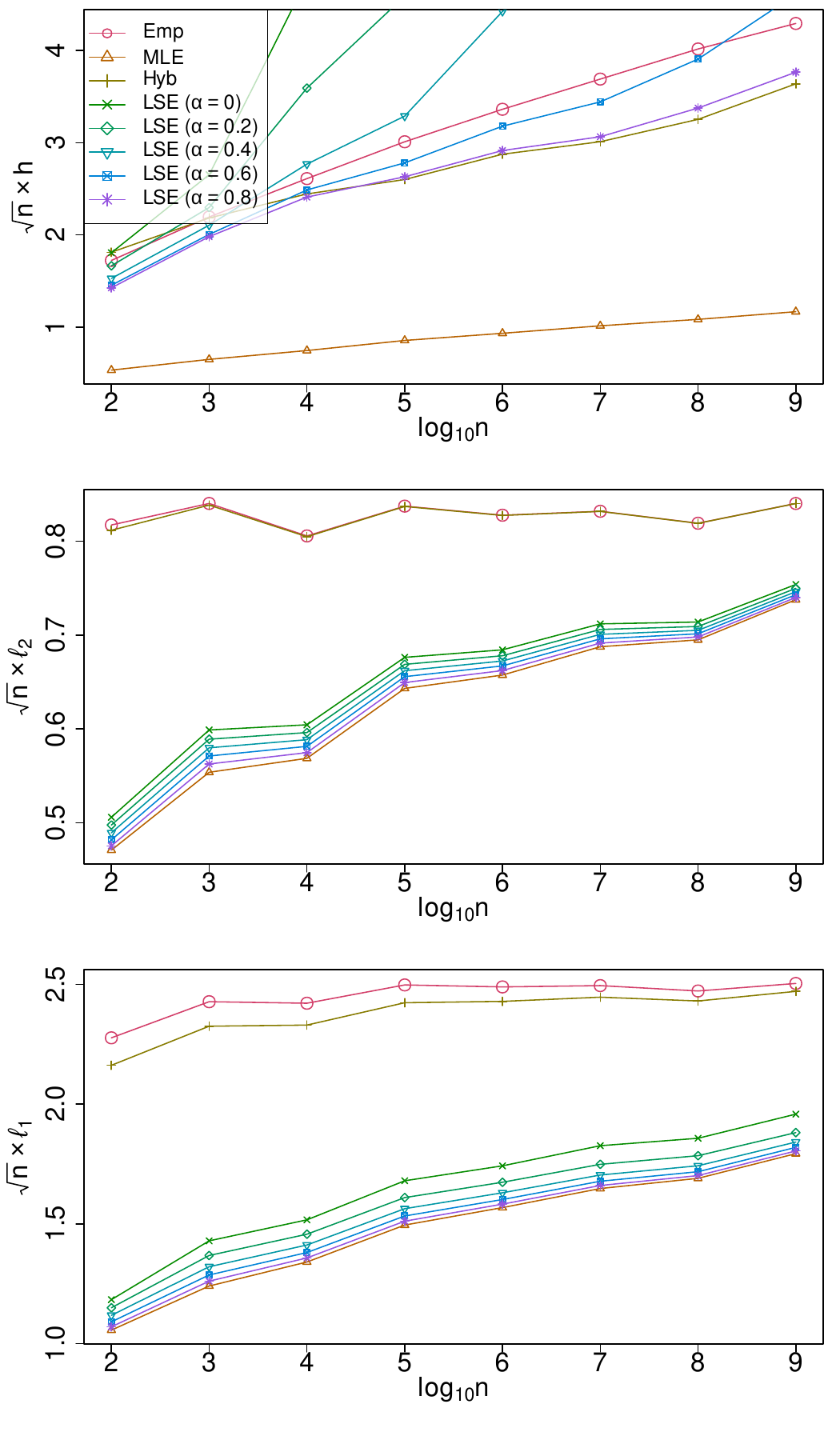}
  \caption{Geometric mixtures: Finite with $m = 7$ components (left); with the mixing
    distribution $\mathrm{Beta}(2,3)$ transformed to have support on
    $[0.1,0.9]$ (right).}
  \label{fig:sim3}
\end{figure*}

\begin{figure*}[!tbh]
  \centering
  \includegraphics[width=.48\textwidth]{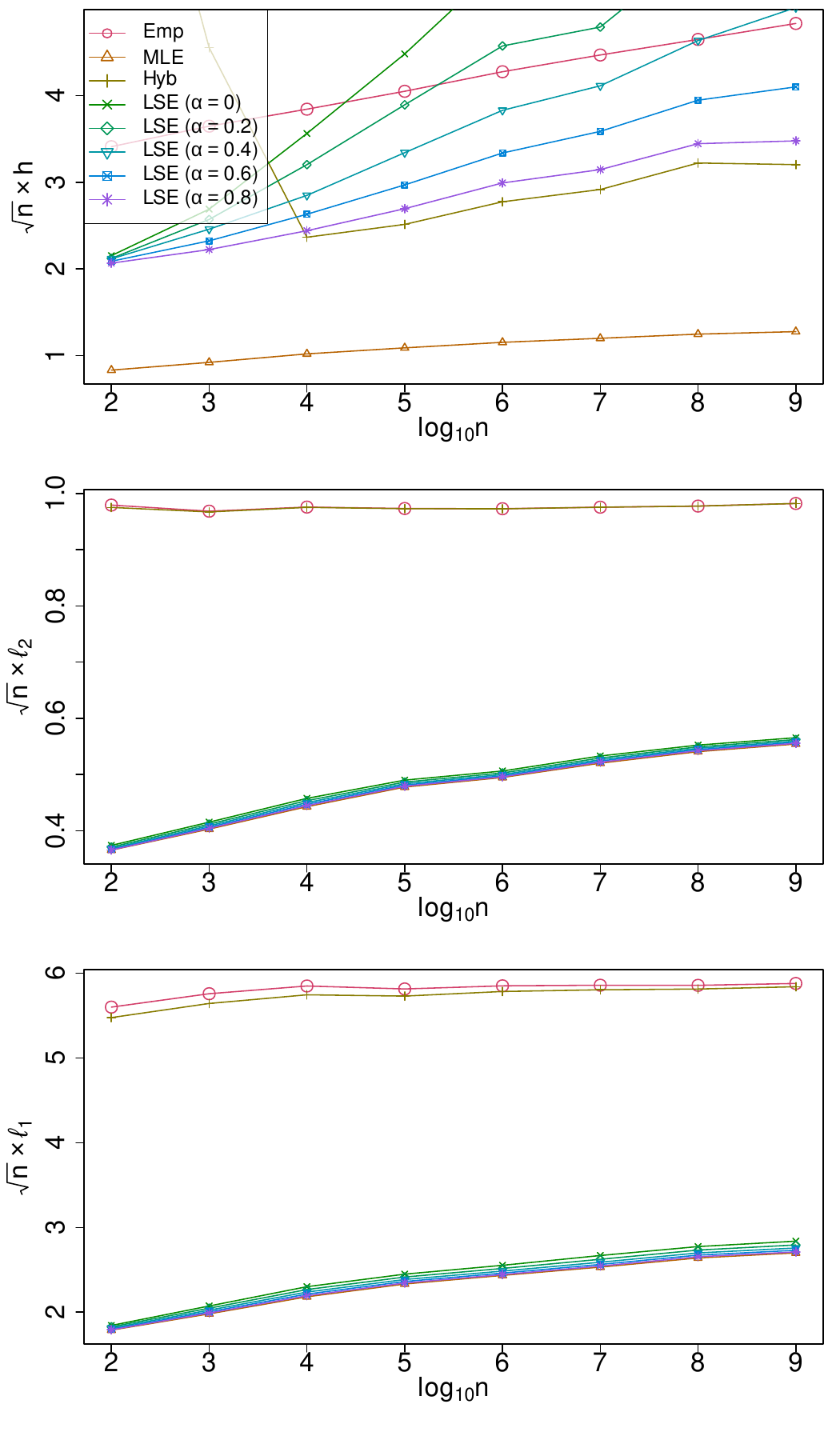}
  \includegraphics[width=.48\textwidth]{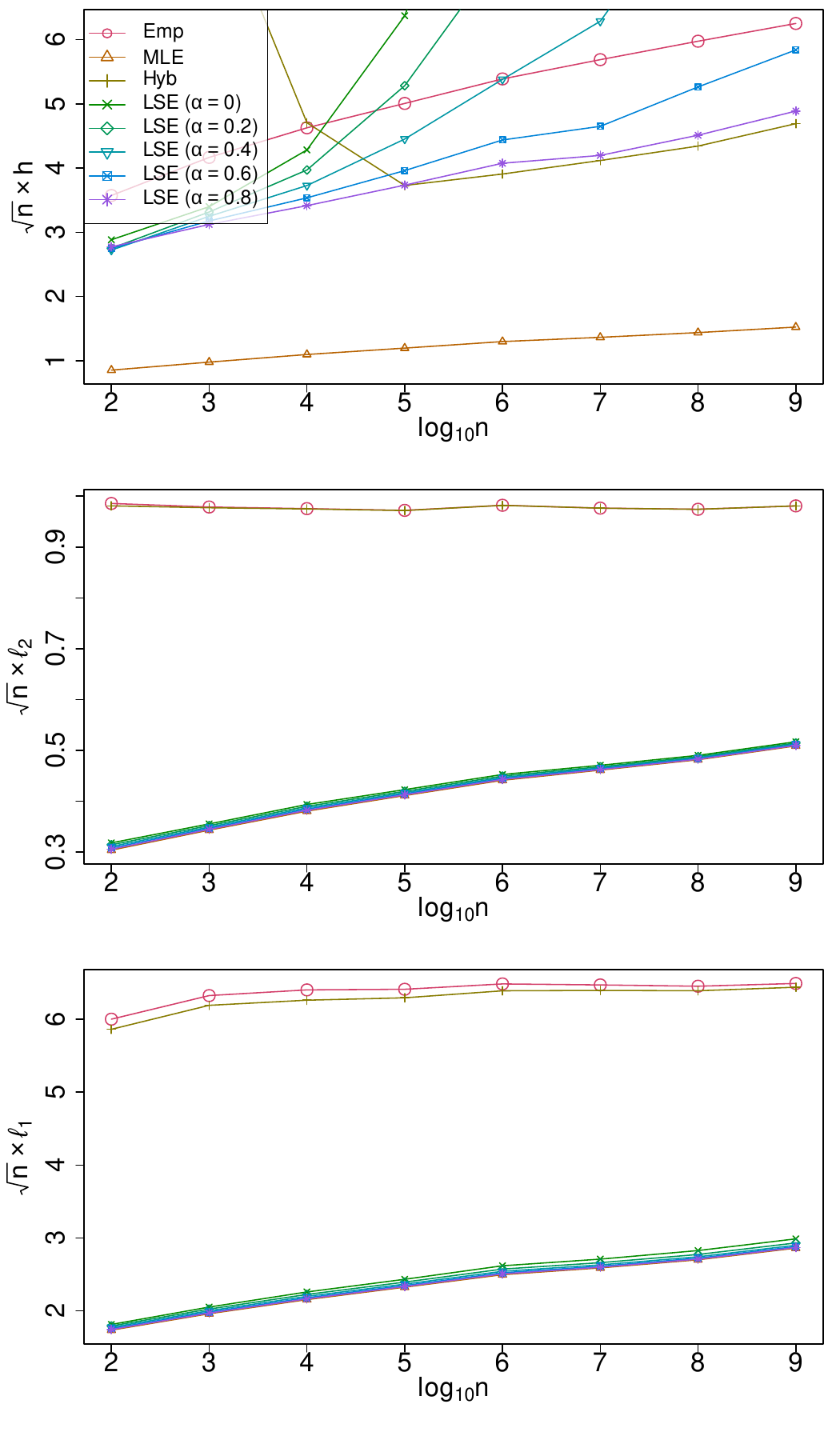}
  \caption{Negative Binomial mixtures: Finite with $m = 7$  components (left); with the mixing
    distribution $\mathrm{Beta}(2,3)$ transformed to have support on
    $[0.1,0.9]$ (right).}
  \label{fig:sim4}
\end{figure*}

The simulation results are summarized and presented in
Figures~\ref{fig:sim1}--\ref{fig:sim4}, where every marked point on a
curve is the mean value of a performance measure over $1000$
simulation runs. Figure~\ref{fig:sim1} shows the results for finite
Poisson mixtures, where the component means are chosen to be
$\theta = 1, 2, \dots, m$ and their associated mixing proportions
follow a geometrically decreasing sequence that is proportional to
$0.8^{\theta-1}$. The three plots on the left panel correspond to the finite
mixture with $m=2$ components, and those on the right panel to the one with
$m= 8$ components.

Figure~\ref{fig:sim2} shows the results for two Poisson mixtures in
continuous mixing settings. The left three plots correspond to a
mixing distribution uniform on $[0.2, 5]$, that is $\mathrm{U}(0.2, 5)$,
while the right three ones to a mixing distribution with mass
$\frac 1 3$ at $0$ and mass $\frac 2 3$ for $\mathrm{U}(0.2, 5)$.

Figure~\ref{fig:sim3} shows the results for two Geometric mixtures,
with $m = 7$ components (left panel) and a continuous mixing
distribution (right panel), respectively. The $7$-component finite
mixture has a mixing distribution with support points
$\theta = 0.2, 0.3, \dots, 0.8$ with the same mixing probability $1/7$, while the continuous mixing
distribution is the $\mathrm{Beta}(2, 3)$ distribution with its
support $[0,1]$ transformed to be $[0.1, 0.9]$.

Figure~\ref{fig:sim4} shows the results for two Negative Binomial
mixtures, with $m = 7$ components (left panel) and a continuous mixing
distribution (right panel), respectively. The finite and continuous
mixing distributions are the same as used for the Geometric
mixtures. The size for the Negative Binomial mixtures is set to
$r = 10$.

No matter whether we choose Poisson, Geometric or Negative Binomial
mixtures, with a discrete or continuous mixing distribution, the
simulations confirm our theoretical results. The hybrid estimator
seems indeed to be $n^{-1/2}$-consistent as well in the $\ell_1$- as
in the $\ell_2$-distance. The weighted least squares estimator is also
quite stable here. Note that the $n^{-1/2}$-consistency of the
weighted least squares estimator seems also to hold for the
$\ell_1$-distance and for $\alpha > 1/2$ (recall that our theory
covers only convergence in the $\ell_2$-norm and $\alpha \in
[0,1/2]$). For the Hellinger distance, we observe that both the hybrid
and the weighted least squares estimators blow up with increasing
sample sizes. The NPMLE seems to be
$n^{-1/2}$-consistent in the $\ell_1$- and $\ell_2$-distance, and the
same graphs might even suggest the stronger result that the NPMLE is
$\sqrt n$-consistent in the Hellinger distance. In case this holds
true, a new line of proof needs to be found to get rid of the
logarithmic factor in Theorem \ref{Rate}.  We would like also to draw
the reader's attention to the fact that in terms of the performance
measured by $\ell_1$- and $\ell_2$-norms one takes advantage in
considering the NPMLE, the hybrid or the weighted LSE's for
$\alpha > 0$ instead of the empirical estimator.

\subsection{Confidence intervals}
\label{sec:ci}

\begin{figure*}[!tbh]
  \centering
  \includegraphics[width=.7\textwidth]{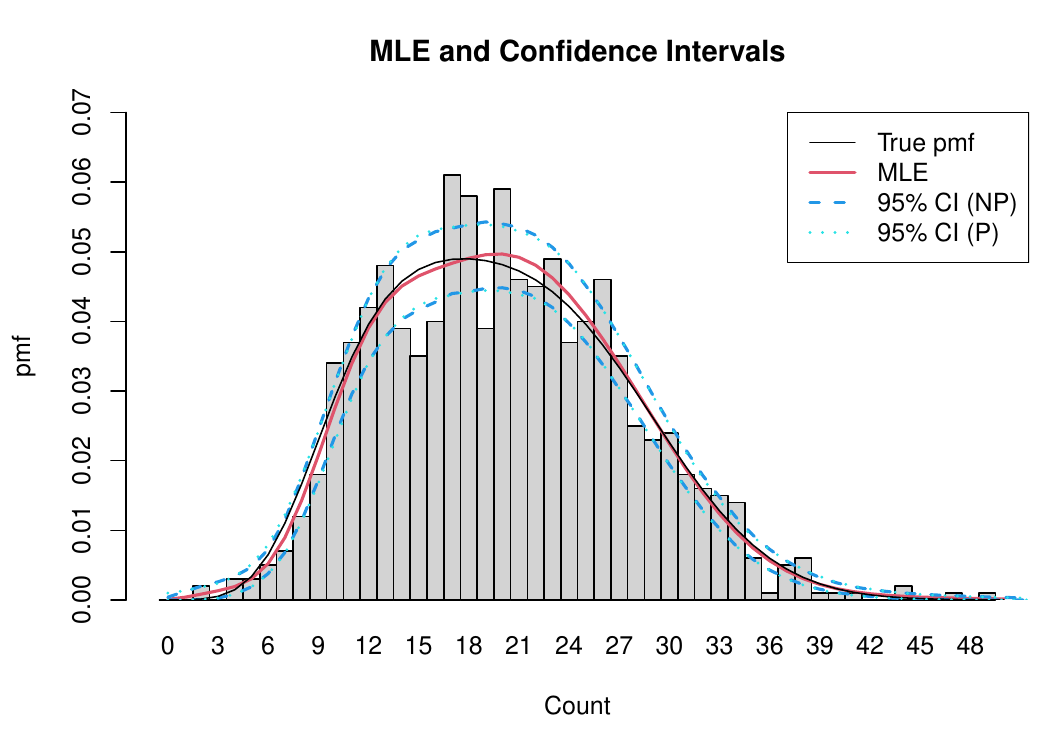} 
  \caption{The NPMLE with the nonparametric and parametric bootstrap
    confidence intervals for a generic sample drawn from a Poisson mixture
    with the mixing distribution $\mathrm{U}(10, 30)$.}
  \label{fig:boot1}
\end{figure*}

One can further construct confidence intervals for the NPMLE at each
value of $k$ using bootstrap.  There are two bootstrap approaches that
can be adopted here: Nonparametric and parametric. With the
nonparametric bootstrapping, one draws independent $B$ bootstrap
samples with replacement from the set of original observations,
computes the NPMLE for each of the bootstrap samples and at each
$k$-value constructs a confidence interval based on the quantiles of
the $B$ obtained values of the NPMLE at $k$. There are quite a few
bootstrapping methods available for constructing a confidence interval
\citep{davison-hinkley-1997}. Here, we chose to simply adopt the basic
percentile method, which uses the empirical $2.5\%$ and $97.5\%$
quantiles of the $B$ obtained values as the lower and upper endpoints
of a $95\%$ confidence interval. With the parametric bootstrapping,
one first computes the NPMLE, $\widehat \pi_n$, and draws independent
$B$ bootstrap samples from the $\widehat{\pi}_n$. Using each bootstrap
sample in the same way as above, one can thus construct parametric
confidence intervals.

Figure\ref{fig:boot1} shows the true pmf and the NPMLE computed for a
generic sample of size $1000$ drawn from a Poisson mixture with a
mixing distribution $\mathrm{U}(10, 30)$, along with the two $95\%$
confidence intervals using the nonparametric and parametric bootstrap
methods ($B=1000$), respectively. The mixing distribution
$\mathrm{U} (10, 30)$ is chosen so that the range of observations is
similar to that in the real data application presented in
Section~\ref{sec:real-data}. It can be seen that the two confidence
intervals, produced with the nonparametric and parametric re-sampling
procedures, are very similar.

Based on simulations, we can further investigate the performance of
the two bootstrapping methods, in particular the coverage probability
and mean length of the confidence intervals constructed. By
replicating the above process $1000$ times, we can obtain $1000$
confidence intervals using either the nonparametric or parametric
bootstrap method ($B=1000$). This allows us to obtain the estimated
coverage probability and mean length. Four mixtures are considered for
this study:
\begin{itemize}
\item Finite Poisson mixture with $m = 2$ components, as used in
  Section~\ref{sec:sim}.
\item Poisson mixture, with a mixing distribution $\mathrm{U}(10,30)$.
\item Finite Geometric mixture with $m = 7$ components, as used in
  Section~\ref{sec:sim}.
\item Geometric mixture with a continuous mixing distribution, as used
  in Section~\ref{sec:sim}.
\end{itemize}

\begin{figure*}[!tbh]
  \centering
  \includegraphics[width=.24\textwidth]{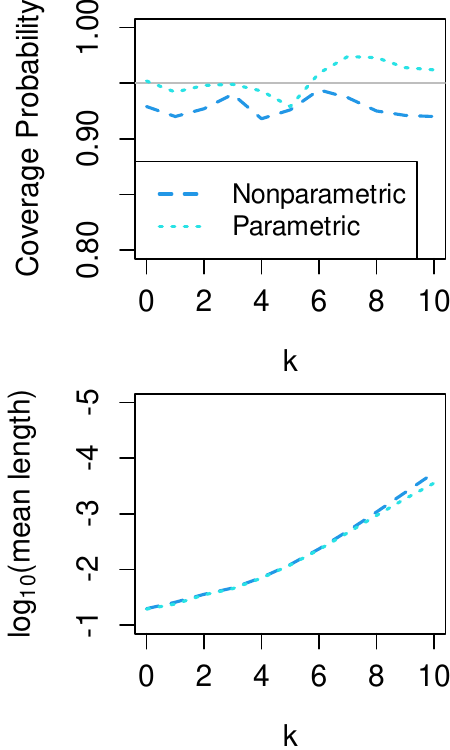} 
  \includegraphics[width=.24\textwidth]{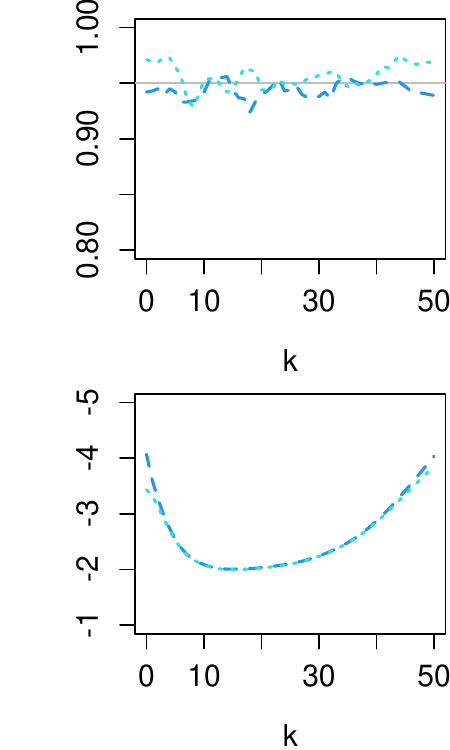} 
  \includegraphics[width=.24\textwidth]{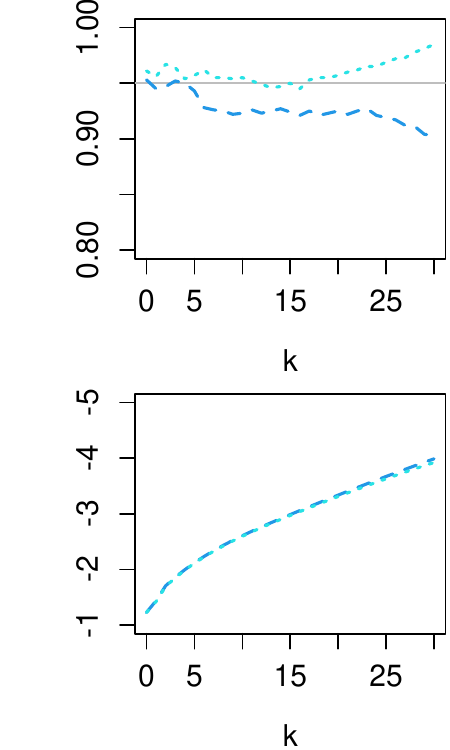} 
  \includegraphics[width=.24\textwidth]{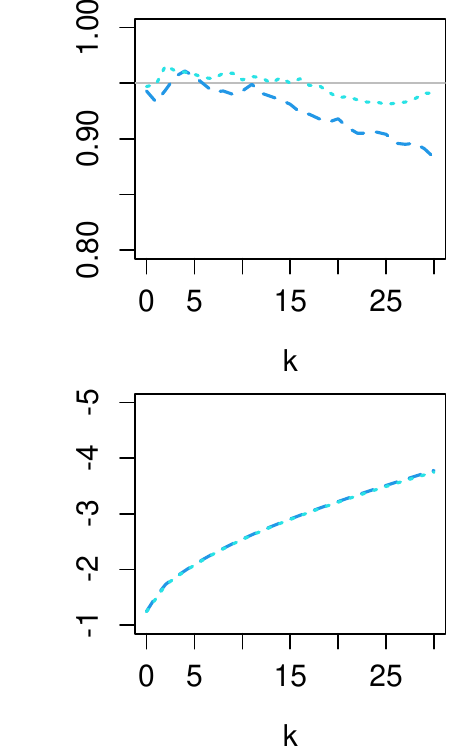} 
  \caption{$95\%$ bootstrap confidence intervals: Finite Poisson mixture with
    $m = 2$ components (leftmost); Poisson mixture with a mixing
    distribution $\mathrm{U}(10, 30)$ (second from left); Finite Geometric
    mixture with $m = 7$ components (second from right); Geometric mixture with a
    mixing distribution $\mathrm{Beta}(2,3)$ transformed to have
    support on $[0.1,0.9]$ (rightmost).}
  \label{fig:boot2}
\end{figure*}

Using a sample size $n = 1000$ in all situations, the results are
obtained and shown in Figure~\ref{fig:boot2}. In each of these plots,
a sufficiently wide range of $k$-values is chosen, beyond which there
is virtually no observation (for a sample of size $n=1000$). In all
cases, both methods yield similar mean lengths. However, the
parametric bootstrap procedure seems to provide a better coverage
probability which stays close to $95\%$.  As expected, it is to be
noted that in both methods, there is a gradually deteriorating trend
in terms of the coverage probability as $k$ moves away from the range
of the observations.

\subsection{A real data application}
\label{sec:real-data}

In this section, we illustrate our method using a real dataset.
Table~\ref{tab:quakes} lists the yearly counts of world major
earthquakes with magnitude 7 and above for the years 1900--2021
\citep{zucchini-macdonald-langrock-2016,usgs}. For this particular
dataset, mixtures of Poisson distributions seem to be a very
reasonable model. Hence, we consider fitting such a mixture to their
observed frequencies, using the same estimators as in
Section~\ref{sec:sim}. The cross-validation technique is used here for
the real dataset, as the true distribution is unknown. The results are
given in Table~\ref{tab:cv}, where each value is the mean of a
performance measure over $1000$ runs of a $2$-fold
cross-validation. For each $2$-fold cross-validation, all earthquake
events are randomly divided into two groups of equal size. Then one
group is used to find each estimate of the mixture while the remaining
one is retained to obtain the empirical estimate, followed by
computing a distance measure between the two. This computation is
repeated after switching the roles played by the two groups.  Note
that the empirical estimator is clearly the worst among all the ones
considered. This shows again the great advantage of using estimators
which include information about the statistical model.  Given that the
NPMLE shows overall a better performance compared to the other
estimators, we recommend its use in practice.

\begin{table}
\centering
\begin{tabular}{rrrrrrrrrrr} 
\hline
& \multicolumn{1}{c}{0} & \multicolumn{1}{c}{1} & \multicolumn{1}{c}{2} & \multicolumn{1}{c}{3} & \multicolumn{1}{c}{4} & \multicolumn{1}{c}{5} & \multicolumn{1}{c}{6} & \multicolumn{1}{c}{7} & \multicolumn{1}{c}{8} & \multicolumn{1}{c}{9} \\
\hline
1900+ & 13 & 14 & 8 & 10 & 16 & 26 & 32 & 27 & 18 & 32 \\ 
1910+ & 36 & 24 & 22 & 23 & 22 & 18 & 25 & 21 & 21 & 14 \\ 
1920+ & 8 & 11 & 14 & 23 & 18 & 17 & 19 & 20 & 22 & 19 \\
1930+ & 13 & 26 & 13 & 14 & 22 & 24 & 21 & 22 & 26 & 21 \\
1940+ & 23 & 24 & 27 & 41 & 31 & 27 & 35 & 26 & 28 & 36 \\ 
1950+ & 39 & 21 & 17 & 22 & 17 & 19 & 15 & 34 & 10 & 15 \\ 
1960+ & 22 & 18 & 15 & 20 & 15 & 22 & 19 & 16 & 30 & 27 \\ 
1970+ & 29 & 23 & 20 & 16 & 21 & 21 & 25 & 16 & 18 & 15 \\ 
1980+ & 18 & 14 & 10 & 15 & 8 & 15 & 6 & 11 & 8 & 7 \\ 
1990+ & 18 & 16 & 13 & 12 & 13 & 20 & 15 & 16 & 12 & 18 \\
2000+ & 15 & 16 & 13 & 15 & 16 & 11 & 11 & 18 & 12 & 17 \\
2010+ & 24 & 20 & 14 & 19 & 12 & 19 & 16 & 7 & 17 & 10 \\
2020+ & 9 & 19 \\
\hline
\end{tabular}
\medskip
\caption{Counts of world major earthquakes (magnitude 7 and above) for
  the years 1900--2021}
\label{tab:quakes}
\end{table}

\begin{table}[!htb]
  \centering
  \begin{tabular}{c|c|c|c|ccccc} \hline
    & Emp &MLE & Hybrid & \multicolumn5c{LSE} \\
    & & & & $\alpha = 0.0$ & $\alpha = 0.2$ & $\alpha = 0.4$ & $\alpha = 0.6$ & $\alpha = 0.8$ \\ \hline
    $h$      & 0.455 & 0.104 & 0.078 & 0.642 & 0.631 & 0.609 & 0.580 & 0.550 \\
    $\ell_2$ & 0.175 & 0.014 & 0.031 & 0.014 & 0.014 & 0.014 & 0.014 & 0.014 \\
    $\ell_1$ & 0.762 & 0.547 & 0.765 & 0.558 & 0.556 & 0.554 & 0.552 & 0.550 \\
    \hline
  \end{tabular}
\vspace{0.2cm}
  \caption{Cross-validation results for the world earthquake dataset}
  \label{tab:cv}
\end{table}

\begin{figure*}[!tbh]
  \centering
  \includegraphics[width=.7\textwidth]{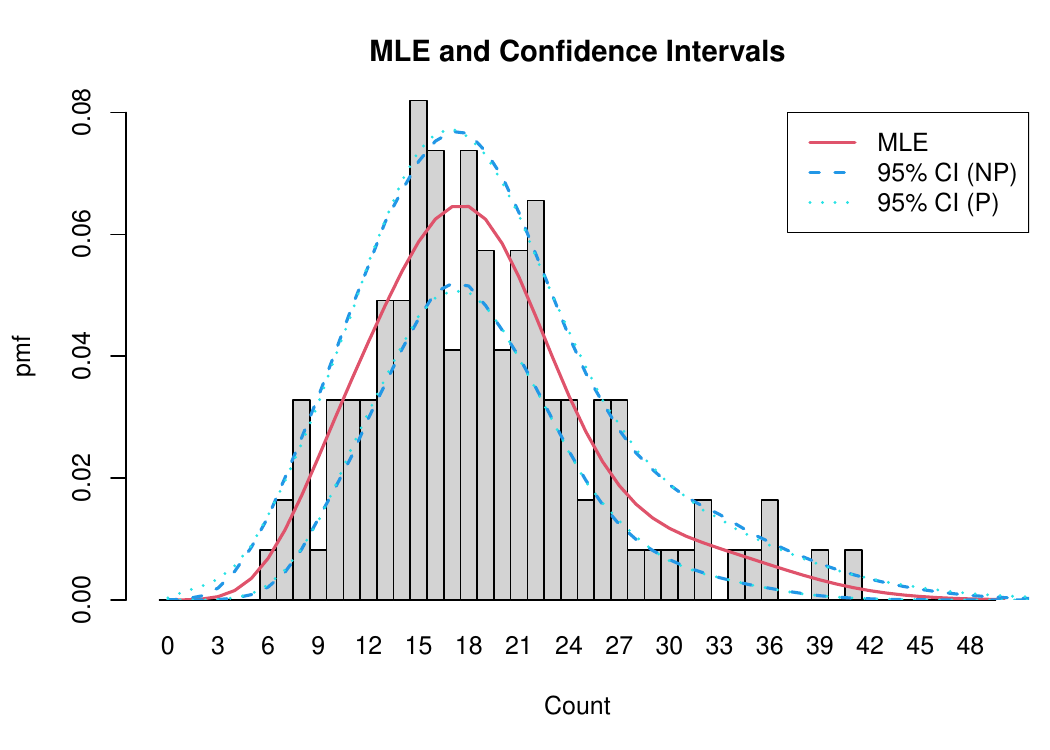} 
  \caption{The NPMLE with the 95\% nonparametric and parametric bootstrap confidence intervals.}
  \label{fig:boot-earthquake}
\end{figure*}

Figure~\ref{fig:boot-earthquake} also shows the NPMLE, along with two
95\% bootstrap confidence intervals which were produced
nonparametrically and parametrically from $B = 1000$ bootstrap
samples, with the methods described in Section~\ref{sec:ci}.  To
highlight the great advantage of our method over the empirical
estimator, we would like to note that it is always possible to obtain
a much more meaningful confidence interval for the true probability of
any count of interest.  Consider for example the extreme counts $x=0$
and $50$.  While the empirical estimator assigns the value $0$ to
these unobserved values, we get the asymptotic $95\% $-confidence
intervals for the true probabilities $\pi_0(0)$ and $\pi_0(50)$, as
given in Table~\ref{tab:quakes-ci}.  Note that the confidence
intervals obtained with the parametric approach are wider than those
with the nonparametric one.  Based on the better coverage probability
results obtained with the parametric bootstrap approach, the
confidence intervals obtained for $\pi_0(0)$ and $\pi_0(50)$ with this
approach are more trustworthy.

\begin{table}
\centering
\begin{tabular}{c|c|c} \hline
\textit{Method}  & $\pi_0(0)$ & $\pi_0(50)$ \\ \hline
  Nonparametric bootstrap & $[1.532 \times 10^{-7}, 1.643  \times 10^{-5}]$ & $[3.339 \times 10^{-6}, 3.978 \times 10^{-4}]$ \\
  Parametric bootstrap & $[1.771 \times 10^{-7}, 2.732 \times 10^{-4}]$
                           & $[1.483 \times 10^{-6}, 7.350 \times
                             10^{-4}]$ \\\hline
\end{tabular}
\medskip
\caption{Bootstrap confidence intervals for the earthquake dataset}
\label{tab:quakes-ci}
\end{table}

\section{Discussion and outlook}
\label{sec:discussion}

In this work we have derived the rate of convergence of the NPMLE for
a wide range of mixtures of power series distributions. We proved that
the NPMLE converges at a rate which is no slower than
$(\log n)^{3/2}/\sqrt n$ in the Hellinger distance under mild
conditions which are satisfied by all well-known power series
distributions. In (\ref{UnifRfinite}) and (\ref{UnifRinfinite}) we
show that the rate $(\log n)^{3/2}/\sqrt n$ is uniform over some
classes of mixing distributions. Although this uniformity is proved
for the exceedance probability risk, the recent results in
\cite{polyanskiy2021sharp} on minimax lower bounds in the Hellinger
distance in the sense of mean square risk strongly suggests that the
logarithmic factor in the obtained rate can not be suppressed, at
least for mixtures of Poisson distributions.  Based on the NPMLE, we
constructed alternative nonparametric estimators which we proved to be
$n^{-1/2}$-consistent in the $\ell_p$-distances for
$p \in [2, \infty]$ or even $[1, \infty]$.  Our simulations clearly
show that the NPMLE should be also $n^{-1/2}$-consistent in the
$\ell_1$-distance but a proof is yet to be developed. In
\cite{Patilea2005} it was shown that the NPMLE of a mixture of PSDs is
asymptotically equivalent to the empirical estimator and hence the
vector of its values at finitely many integers converges jointly at
the parametric rate to a multivariate Gaussian distribution. However,
this asymptotic normality was proved under the condition that the true
mixing distribution $Q_0$ satisfies
$Q((\theta, \theta+ \tau]) \ge d\tau^\gamma$ for all
$\theta, \tau \in (0, \epsilon)$ with
$d > 0, \gamma > 0, \epsilon > 0$ some fixed constants. Although this
assumption seems to exclude the important class of finite mixtures,
the authors in \cite{Patilea} show empirical evidence that this
continues to hold even for this setting. This is also consistent with
the observed $n^{-1/2}$-consistency of the NPMLE in some of our
simulation results. Thus, even if there is a substantial numerical
evidence that the NPMLE converges at the parametric rate in all the
$\ell_p$-distances for $p \in [1, \infty]$, it is not at all
straightforward to get rid the logarithmic factor obtained in its
Hellinger convergence rate when working with any of the $\ell_p$
distances. As already mentioned above in Section~\ref{hyb}, the hybrid
estimator might offer an interesting perspective when investigating
this difficult problem. In fact, this estimator is shown to be
$n^{-1/2}$-consistent and does only involve the empirical estimator
and the NPMLE itself.  Another possible approach would be to use a
sieve technique by first approximating $\pi_0$ by the sequence
$(\pi_{0, K})_{K}$, where
\begin{eqnarray*}
\pi_{0, K}(k)   =    \int_{\Theta}  f_{\theta, K}(k)  dQ_0(\theta),   \   k \in \{0, \ldots, K \}
\end{eqnarray*}
with
\begin{eqnarray*}
f_{\theta, K}(k)  =  \frac{f_\theta(k)}{\sum_{j=0}^K f_\theta(j)},  \   k \in \{0, \ldots, K \}.
\end{eqnarray*}
For a fixed $K$, the NPMLE $\widehat{\pi}_{n, K}$ converges to
$\pi_{0, K}$ at the $n^{-1/2}$-rate since the support of the true pmf
is finite. The main challenge is to be able to show that this
convergence result continues to hold as $K$ grows to $\infty$.  The
first step would be to explore the success of such an idea in showing
pointwise convergence before considering some global behavior in the
sense of the $\ell_1$- or $\ell_2$-distance.

Although the focus in this paper was on estimation of the mixed
distribution, estimation of the mixing distribution is an important
problem from both the theoretical and practical perspectives. Several
papers were devoted to finding minimax lower bounds as well as
estimators that attain them. For mixtures with a finite but unknown
number of components, and under certain regularity conditions, it was
established in \cite{chen1995} that the convergence rate for
estimating the mixing distributions can not be faster than $n^{-1/4}$.
Since the current work is on mixtures of PSDs, \cite{loh1996global}
and \cite{heng97} bring in very interesting insights as they treat
mixtures of of discrete distributions which belong to an exponential
family (which the case of a PDS). For examples, for mixtures of
Negative Binomials, and if the true mixing distribution is assumed to
admit a density $q_0$ with respect to Lebesgue measure, then it is
shown in \cite{loh1996global} that the optimal convergence rate under
the weighted $\ell_p$ loss ($1\le p \le \infty$) is
$\frac{1}{(\log n)^\alpha}$ where $\alpha > 0$ denotes the degree of
smoothness of $g_0$. The proofs are based on Fourier methods. For
mixtures of Poisson distributions with compactly supported mixing
distributions admitting a density $g_0$ which is either
$\alpha$-Lipschitz or belongs to the Sobolev space
$\{ g: \int [g^{(r)}(t)]^2 dt < M \}$ for $M > 0$, it follows from
\cite{heng97} that the mean integrated square estimation error (after
taking the square root) for a suitable orthonormal polynomial demixing
estimator converges at the rate
$ \left(\log n / \log(\log n)\right)^{\alpha}$ and
$ \left( \log n / \log(\log n) \right)^{r}$ in the first and second
case respectively. Furthermore, the rate
$ \left( \log n / \log(\log n)\right)^{r}$ was shown to be optimal in
the case of Sobolev densities. Note that these rates are somehow
reminiscent of the minimax lower bound for the total regret derived in
\cite{polyanskiy2021sharp}, although the authors of that work do not
make smoothness assumption about the mixing distribution. See also
(\ref{regretM}).  In \cite{roueff2005} a general approach based on
projection estimators was proposed for estimating the density of a
mixing distribution in mixtures of discrete distributions, including
mixtures of PSDs; see \cite[Section 5]{roueff2005}.  In particular,
the optimal rate $\frac{1}{(\log n)^\alpha}$ for $\alpha$-smooth
mixing densities in mixtures of Negative Binomials can be again
retrieved from \cite[Corollary 1]{roueff2005}.

The results reviewed above show that convergence rates in the problem
of estimating the mixing distribution in a mixture of PDSs are very
slow. However, a close inspection of the arguments employed to obtain
the optimal rates reveals that it might be possible to use the
construction of the optimal estimator (at which the minimax lower
bound is attained) to get sharp lower bounds in the direct problem;
i.e., that of estimating the mixture distribution. A similar
observation was made above in Section~\ref{sec:minimax} about
derivation of the minimax lower bounds for mixtures of Poisson in
\cite{polyanskiy2021sharp} by re-using the orthonormal eigenbasis
constructed for the problem of estimating the Poisson
means. Investigating such links for other mixtures of PSDs, such as
mixtures of Geometric and Negative Binomial distributions, belongs to
our list of future research works.

\section*{Acknowledgments}

The authors thank two anonymous referees for their meticulous reading
and very relevant comments which helped improve our manuscript. This
work was financially supported by the Swiss National Fund Grant
(200021191999).

\section*{Authorship contribution statement}

\par \noindent \textbf{Fadoua Balabdaoui}: Conceptualization, Methodology, Software, Validation, Formal analysis, Investigation, Original draft, Review \& Editing, Supervision, Funding  acquisition \\

\par \noindent \textbf{Harald Besdziek}: Methodology, Formal analysis, Original draft \\

\par \noindent\textbf{Yong Wang}: Software, Validation, Methodology,  Original draft

\appendix

\section{Proofs and auxiliary results for Section 2 }

\par \noindent \textbf{Proof of Theorem 2.1.}   \ 
\noindent Let $\mathcal Q$ denote the set of all mixing distributions defined on $\Theta$. Let $X_1, \ldots, X_n$ be i.i.d. random variables from $\pi_0$ with true (discrete) mixing distribution $Q_0$. We denote by $k_1, \ldots, k_D$ the distinct values taken by the observations and $n_j = \sum_{i=1}^n \mathds{1}_{\{X_i = k_j \}}$.  With $Q \in \mathcal{Q}$, the likelihood function is given by 
\begin{eqnarray*}
L(Q) = \prod_{i=1}^n \int_\Theta f_\theta(X_i) dQ(\theta) = \prod_{j=1}^D \Big( \int_\Theta f_\theta(k_j) dQ(\theta) \Big)^{n_j}.
\end{eqnarray*}
Using the same notation as in Chapter 5 of \cite{lindsay1995}, we write 
\begin{eqnarray*}
L_j(Q) =   \left(\int_\Theta  f_\theta(k_j)  dQ(\theta)  \right)^{n_j}, \  \ j \in \{1, \ldots, D \}.
\end{eqnarray*}
Note that $L(Q_0)  > 0$, which means that the set 
\begin{eqnarray*}
\mathcal{M}  = \left \{  \left(L_1(Q),  \ldots, L_D(Q)  \right):  \ \ Q \in \mathcal Q  \right \}
\end{eqnarray*}
contains at least one interior point with strictly positive likelihood.\\
Using again the same notation as in \cite{lindsay1995}, we define the likelihood curve (including the null vector in $\mathbb{R}^D$) by
\begin{eqnarray*}
\Gamma :=  \left\{ \Big(f_\theta(k_1),\ldots,f_\theta(k_D)\Big) : \theta \in \Theta \right\} \cup \left\{ \big(0,\ldots,0\big)  \right \}.
\end{eqnarray*}
We show now that $\Gamma$ is a compact subset of $\mathbb R^D$. It is clearly bounded since for all $\mathbf{v}  = (v_1, \ldots, v_D) \in \Gamma$ we have that $\max_{1 \le j \le D}  \vert v_j \vert \le 1$. Now we proceed to show that it is also closed.  Recall that $\mathbb K = \mathbb N$, the set of all non-negative integers. This means that $b_k > 0$ for all $k \in \mathbb N$. Consider now a sequence $\mathbf{v}^{(m)}: =  (v^{(m)}_1, \ldots, v^{(m)}_D) \in \Gamma$ such that 
\begin{eqnarray*}
\lim_{m \nearrow  \infty} \mathbf{v}^{(m)}  = \mathbf{\tilde{v}} =  (\tilde{v}_1, \ldots, \tilde{v}_D).
\end{eqnarray*}
If $\tilde{v}_j = 0$ for all $j \in \{1, \ldots, D\}$, then the limit $\mathbf{\tilde{v}}$ is clearly in $\Gamma$. Suppose now that there exists at least one index $j_0 \in \{1, \ldots, D\}$  such that $\tilde{v}_{j_0}  \ne 0$. By definition of $\Gamma$, we can find a sequence $\theta^{(m)}$ such that $v^{(m)}_j =  f_{\theta^{(m)}}(k_j)$ for all $j \in \{1, \ldots, D\}$.\\
Consider first the case $R = \infty$. We start with showing that for any fixed $k \in \mathbb N$, 
\begin{eqnarray*}
\lim_{\theta \nearrow \infty}  f_\theta(k)  =  0.
\end{eqnarray*}
Since $b_k > 0$ for all $k \in \mathbb N$, it follows that $b(\theta)  > b_{k+1}  \theta^{k+1}$. This implies that for all $\theta > 0$,
\begin{eqnarray*}
f_\theta(k)  <  \frac{b_k}{b_{k+1}}  \frac{1}{\theta},
\end{eqnarray*}
from which we conclude the claimed limit.   Suppose now that the sequence $\theta^{(m)}$ is unbounded. This means that we can find a subsequence $\theta^{(m')}$ such that 
$\lim_{m'  \nearrow \infty}  \theta^{(m')} = \infty$. This in turn implies that  $\lim_{m' \nearrow \infty} v^{(m')}_{j_0}=0$, which is in contradiction with our assumption above.   Thus, $\theta^{(m)}$ has to be bounded. This now implies that there exists a subsequence $\theta^{(m')}$ and $\tilde{\theta} $ such that
\begin{eqnarray*}
\lim_{m' \nearrow \infty}  \theta^{(m')} = \tilde{\theta}.
\end{eqnarray*}
Using continuity of the map $\theta \mapsto f_\theta(k)$ for any fixed $k \in \mathbb N$ (at $\theta=0$ we use continuity to the right), it follows that
\begin{eqnarray*}
\left(f_{ \theta^{(m')}}(k_1), \ldots, f_{ \theta^{(m')}}(k_D)  \right)  \to  \left( f_{\tilde{\theta}}(k_1), \ldots,  f_{\tilde{\theta}}(k_D) \right)
\end{eqnarray*}
as $m' \nearrow \infty$. By uniqueness of the limit, we conclude that 
\begin{eqnarray*}
(\tilde{v}_1, \ldots, \tilde{v}_D)) =  \left( f_{\tilde{\theta}}(k_1), \ldots,  f_{\tilde{\theta}}(k_D) \right).
\end{eqnarray*}
This also means that $(\tilde{v}_1, \ldots, \tilde{v}_D) \in \Gamma$.\\
Now consider the case $R < \infty$. Assume first that $b(R) = \infty$. Observe that
for any fixed $k \in \mathbb N$,
\begin{eqnarray*}
\lim_{\theta \nearrow R}  f_\theta(k)  =  0.
\end{eqnarray*}
Indeed, this follows from combining $\lim_{\theta \nearrow R} \theta^k = R^k < \infty$ and $\lim_{\theta \nearrow R}  b(\theta) = \infty$.  As before, let $\mathbf{\tilde{v}}$ denote the limit. If $\tilde{v}_j =0$ for all $j \in \{1, \ldots, D \}$, then we are done since $(0, \ldots, 0) \in \Gamma$. Suppose now that there exists $j_0 \in \{1, \ldots, D\}$ such that $v_{j_0} > 0$. Since $\Theta \subset \overline \Theta = [0,R]$ is compact, the sequence $\theta^{(m)}$ has a subsequence $\theta^{(m')}$ which converges to some $\tilde \theta \in \overline \Theta$. So, by continuity of the function $\theta \mapsto f_\theta(k)$, we have  again
\begin{eqnarray*}
\left(f_{ \theta^{(m')}}(k_1), \ldots, f_{ \theta^{(m')}}(k_D)  \right)  \to  \left( f_{\tilde{\theta}}(k_1), \ldots,  f_{\tilde{\theta}}(k_D) \right),
\end{eqnarray*}
and uniqueness of the limit implies that
\begin{eqnarray*}
\mathbf{\tilde{v}} = (\tilde{v}_1, \ldots, \tilde{v}_D) =  \left( f_{\tilde{\theta}}(k_1), \ldots,  f_{\tilde{\theta}}(k_D) \right).
\end{eqnarray*}
Note that $\tilde{\theta} \ne R$ since otherwise we will reach a contradiction with the assumption that $v_{j_0} > 0$.  Thus, $\tilde{\theta}  \in [0, R)$ and $\mathbf{\tilde{v}} \in \Gamma$.  For the case that $b(R) < \infty$ the argument is even simpler because then, $\overline \Theta = \Theta$.\\
In any case, we have shown that $\Gamma$ is compact. Existence and uniqueness of the NPMLE $\widehat{Q}_n$ now follow from Theorem 18 in Chapter 5 of \cite{lindsay1995} plus the subsequent remark that one may include the zero vector in the likelihood curve because it can never appear in the maximizer. The last statement is clear just by definition of $\widehat{\pi}_n$.  \hfill $\Box$

\medskip

\par \noindent \textbf{Proof of Lemma 2.3.} We prove all the properties separately.
\begin{enumerate}
\item We only consider the case $R < \infty$; the case $R = \infty$ is analogous.Note that $\theta \mapsto b(\theta)$ and $\theta \mapsto f_k(\theta)$ are differentiable on $(0, \tilde \theta) \equiv (0, q_0 R)$. For $\theta \in (0, \tilde \theta)$, w
we compute
\begin{eqnarray*}
f_k'(\theta):=\frac{\partial f_k(\theta)}{\partial \theta}& = & b_k\frac{k\theta^{k-1}b(\theta)-\theta^kb'(\theta)}{(b(\theta))^2} \\
& = &  b_k \theta^{k-1} \frac{kb(\theta)-\theta b'(\theta)}{(b(\theta))^2} \\
& = &  \frac{b_k \theta^{k-1}}{b(\theta)} \left( k -  \frac{\theta b'(\theta)}{b(\theta)} \right) \\
& \ge & 0
\end{eqnarray*}
for all $k \ge U$
where
\begin{eqnarray}\label{U}
U = \Big \lfloor \tilde \theta \sup_{\theta \in (0, \tilde \theta)} \frac{b'(\theta)}{b(\theta)} \Big \rfloor +1.    
\end{eqnarray}

\medskip

\item Define 
\begin{eqnarray}\label{W2}
W =  \min \left \{ w \ge 3:  \max_{k \ge w} \frac{b_{k+1}}{b_k}  \le \frac{t_0}{\tilde \theta}   \right \}. 
\end{eqnarray}
We start with the case $R < \infty$.  The implication (2) in the main manuscript means that there exists an integer $K \ge 1$ such that for all $k \ge K$
\begin{eqnarray*}
\frac{b_{k+1}}{b_k}  \le (1+\epsilon) \frac{1}{R}.
\end{eqnarray*}
Now, choose $\epsilon =  (1/q_0 - 1)/2  > 0$. Then, for all $k \ge K$ 
\begin{eqnarray*}
\frac{b_{k+1}}{b_k}  \le \frac{q_0+1}{2 q_0 R}  =  \frac{t_0}{\tilde \theta}.
\end{eqnarray*}
If we impose that $K \ge 3$, then we can see that $W$ defined in (\ref{W2}) is the smallest such an integer $K$.  Hence, for all $k \ge W$ it holds that
\begin{eqnarray*}
\frac{b_{k+1}}{b_k}  \le \frac{q_0+1}{2 q_0 R}  =  \frac{t_0}{\tilde \theta}.
\end{eqnarray*}
Similarly, if $R = \infty$, we can find an integer $K \ge 3$ such that  
\begin{eqnarray*}
\frac{b_{k+1}}{b_k} \le \frac{1}{2M}  =  \frac{t_0}{\tilde \theta}
\end{eqnarray*}
for all $k \ge K$. Taking the smallest such an integer allows to conclude the claim in both cases.  Note that taking $W \ge 3$ will ensure that $W-1 \ge 2$ and hence $1/t^2_0 \le 1/t_0^{W-1}$ needed below.
\medskip

\item For $K \ge \max(U, W)$, we obtain that
\begin{eqnarray*}
\sum_{k \ge K+1}  \pi_0(k)  & = &   \sum_{k \ge K+1}  \int_\Theta  f_{\theta}(k) dQ_0(\theta)  \\
& \le & \sum_{k \ge K+1} f_{\tilde \theta}(k) \int_\Theta dQ_0(\theta) \: , \\ & \: & \textrm{using that $K \ge U$ and property 1 and (A1)} \\
& = & \sum_{k \ge K+1} f_{\tilde \theta}(k)
=  \frac{b_W \tilde{\theta}^W}{b(\tilde{\theta})} \sum_{k \ge K+1}  \frac{b_k \tilde{\theta}^{k-W}}{b_W} \\
& = & \frac{b_W \tilde{\theta}^W}{b(\tilde{\theta})} \sum_{i \ge 1}  \frac{b_{K+i} \tilde{\theta}^{K-W + i}}{b_W} \\
& \le & \frac{b_W \tilde{\theta}^W}{b(\tilde{\theta})} \sum_{i \ge 1} \left( \frac{t_0}{\tilde{\theta}}\right)^{K-W+i}  \tilde{\theta}^{K-W+i}  \: , \\ & \: & \textrm{using that $K \ge W$ and property 2}\\
& = & \frac{b_W \tilde{\theta}^W}{b(\tilde{\theta})} t^{K-W}_0 \sum_{i \ge 1} t_0^i 
=  \frac{b_W \tilde{\theta}^W}{b(\tilde{\theta})} t^{K-W}_0 \frac{t_0}{1-t_0} \\
& = & A t_0^K,
\end{eqnarray*}
where
\begin{eqnarray}\label{A}
A := \frac{b_W \tilde \theta^W}{b(\tilde{\theta})} \frac{1}{t^{W-1}_0(1-t_0)} =  \frac{f_W(\tilde \theta)}{t^{W-1}_0(1-t_0)}.
\end{eqnarray}

\medskip

 \item  Assume that $k \ge W$. Then,
\begin{eqnarray*}
\pi_0(k+1) -  \pi_0(k)
& = & \int_\Theta f_{\theta}(k+1) dQ_0(\theta) - 
\int_\Theta f_{\theta}(k) dQ_0(\theta)\\
& = & \int_\Theta \Big( f_{\theta}(k+1) - f_{\theta}(k) \Big) dQ_0(\theta) \\
& = & \int_\Theta b(\theta)^{-1} \theta^{k} \Big(b_{k+1}\theta - b_{k} \Big) dQ_0(\theta)\\
& \le & \int_\Theta b(\theta)^{-1} \theta^{k} \Big(t_0  \frac{b_k}{\tilde{\theta}} \tilde{\theta} - b_{k} \Big) dQ_0(\theta) \: , \\ & \: & \textrm{using that $k \ge W$ and Property 2 of Lemma 2.3}\\
& = & \big(t_0  -1 \big) \int_\Theta f_{\theta}(k)  dQ_0(\theta)
= (t_0 - 1) \pi_0(k)  < 0 \: ,
\end{eqnarray*}
from which we conclude the proof. \hfill $\Box$
\end{enumerate} 

\medskip

\par \noindent \textbf{Proof of Lemma 2.4.} It follows from Property 3 of Lemma 2.3 that for all $K \ge  \max(U, W)$ (as in that lemma), we have
\begin{eqnarray*}
 \sum_{k \ge K+1}  \pi_0(k)  \le  A t_0^K.
\end{eqnarray*} 
Hence,
\begin{eqnarray*}
\sum_{k \ge K+1}  \pi_0(k)  \le \frac{(\log n)^{3}}{n}
\end{eqnarray*}
provided that
\begin{eqnarray*}
K   \ge     \frac{1}{\log(1/t_0)}  \log \Big(\frac{A n}{(\log n)^{3}} \Big)  
    = \frac{1}{\log(1/t_0)}  \Big( \log A  + \log n  -  3 \log(\log n)  \Big). 
\end{eqnarray*}
Choosing $n$ such that $n \ge A$ where $A$ is the same constant as in (\ref{A}); i.e., 
\begin{eqnarray}\label{lowbn}
n \ge  \frac{b_W \tilde{\theta}^W  t_0^{1-W}}{b(\tilde{\theta}) (1-t_0)}   
\end{eqnarray}
with $W$ is the same constant as in (\ref{W2}) it is enough to take $K$ such that
\begin{eqnarray*}
K \ge \frac{2 \log n}{\log(1/t_0)}.
\end{eqnarray*}
By definition of $K_n$ as the smallest integer satisfying the bound for the tail, we must have
\begin{eqnarray*}
K_n \le  \Big \lfloor  \frac{2 \log n}{\log(1/t_0)} \Big \rfloor  +  1 =:  \tilde{K}_n,
\end{eqnarray*}
which implies that for $n$ large enough, we have
\begin{eqnarray}\label{tildeKn}
\tilde K_n \le \frac{3 \log n}{\log (1/t_0)}.
\end{eqnarray}
Let us now move on to bound the quantity $\log(1/\tau_n)$.   For $n$ large enough so that $\tilde{K}_n \ge \max(U,V,W)$, where $V$ is from (A3), we have by Property 4 of Lemma 2.3 that
\begin{eqnarray*}
\tau_n  = \inf_{0 \le k \le K_n} \pi_0(k)  \ge \pi_0(\tilde K_n) 
=  \int_\Theta f_{\theta}(\tilde K_n) dQ_0(\theta).
\end{eqnarray*}
Note that $\tilde{K}_n \ge \max(U,V,W)$ is equivalent to 
\begin{eqnarray}\label{ineqn}
\Big \lfloor  \frac{2 \log n}{\log(1/t_0)} \Big \rfloor   \ge \max(U,V,W) -1,
\end{eqnarray}
where $U$ is from (\ref{U}).  If $Q_0(\{0\}) > 0$, it follows from Assumption (A2) that $Q_0([0,\delta_0)) \le 1-\eta_0$. This implies that $Q_0([\delta_0, \infty))  \ge \eta_0$.  Hence, by Property 1 of Lemma 2.3, we obtain that 
\begin{eqnarray*}
\tau_n \ge \eta_0 f_{\delta_0}(\tilde K_n). 
\end{eqnarray*}
In the case that $Q_0(\{0\}) = 0$, we know from Assumption (A2) that $Q_0([0,\delta_0)) = 0$. Invoking again Property $1$ of Lemma 2.3, we see that $\tau_n \ge f_{\delta_0}(\tilde K_n)$. So, in any case, 
\begin{eqnarray*}
\tau_n \ge \eta_0 f_{\delta_0}(\tilde K_n)
= \eta_0\Big( \frac{b_{\tilde{K}_n} \delta^{\tilde K_n}_0}{b(\delta_0)} \Big)
\ge \eta_0  \Big(\frac{b_0 \tilde{K}_n^{-\tilde{K}_n} \delta^{\tilde K_n}_0}{ b(\delta_0)}\Big),
\end{eqnarray*}
where the last step applied Assumption (A3). Thus, we obtain for $n$ large enough  (we shall make this statement more precise)
\begin{eqnarray}\label{ineqtaun}
\log(1/\tau_n) & \le & \log\left(\frac{b(\delta_0)}{b_0(1-\eta_0)} (\tilde{K}_n^{\tilde{K}_n} \delta^{-\tilde K_n}_0)\right) \notag \\
& \le   &   \log \left(\frac{b(\delta_0)}{b_0(1-\eta_0)}\right) + \tilde{K}_n \log(\tilde K_n) -  \tilde{K}_n \log \delta_0  \notag  \\
& \le   &  3 \tilde{K}_n \log(\tilde K_n) \le  3 \tilde{K}^{2}_n \le 3  \left(\frac{3}{\log (1/t_0)}\right)^{2}  (\log n)^{2},
\end{eqnarray}
implying that
\begin{eqnarray}\label{ineqn2}
(K_n +1) \log(1/\tau_n)
& \le &  \frac{27 (\log n)^{2}}{\log (1/t_0)^{2}} \left(\frac{2\log n}{\log(1/t_0)} + 2 \right) \notag \\
& \le  &  \frac{81 (\log n)^{3}}{\log (1/t_0)^{3}}.
\end{eqnarray}
Note that in order for the inequalities in (\ref{tildeKn}), (\ref{ineqn}),  (\ref{ineqtaun}) and (\ref{ineqn2}) to hold, it is enough that $n$ satisfies
\begin{eqnarray*}
\frac{2 \log n}{\log(1/t_0)}   +  1 \le  \frac{3 \log n}{\log (1/t_0)} , 
\end{eqnarray*}
\begin{eqnarray*}
\frac{2 \log n}{\log(1/t_0)} \ge \max(U, V, W),    
\end{eqnarray*}
\begin{eqnarray*}
\frac{b(\delta_0)}{b_0 \eta_0}  \le \frac{2\log n}{\log(1/t_0)}, \  \textrm{and}  \  \ \ \frac{1}{\delta_0}   \le  \frac{2 \log n}{\log(1/t_0)}
\end{eqnarray*}
and 
\begin{eqnarray*}
\frac{2\log n}{\log(1/t_0)} + 2 \le \frac{3\log n}{\log(1/t_0)}.
\end{eqnarray*}
Also, using the fact that $f_W(\tilde \theta) \in (0,1)$,  we see that the inequality in (\ref{lowbn}) holds if we take 
$$
n \ge \frac{1}{t_0^{W-1} (1-t_0)}.
$$
Combining this with the conditions above yields
\begin{eqnarray*}
n & \ge  &    \frac{1}{t^2_0} \vee  \exp\left \{\log(t_0^{-1/2}) \left( U \vee V \vee W \vee \frac{b(\delta_0)}{b_0 \eta_0} \vee \frac{1}{\delta_0} \right)   \right \} \vee  \frac{1}{t_0^{W-1}(1-t_0)}  \\
& = &   \exp\left \{\log(t_0^{-1/2}) \left( U \vee V \vee W \vee \frac{b(\delta_0)}{b_0 \eta_0} \vee \frac{1}{\delta_0} \right)   \right \} \vee  \frac{1}{t_0^{W-1}(1-t_0)}
\end{eqnarray*}
by the fact that $t_0 \in (0,1)$ and $W \ge 3$. \hfill $\Box$

\medskip
\medskip

\par \noindent \textbf{Proof of Lemma 2.6 (the basic inequality).} The class $\mathcal M$ is convex and hence $(\widehat \pi_n + \pi_0)/2 \in \mathcal M$. Combining this with the definition of the NPMLE 
\begin{eqnarray*}
0 \le \int \log \left(  \frac{2 \widehat \pi_n}{\widehat \pi_n + \pi_0} \right)  d\mathbb P_n.
\end{eqnarray*}
Now, using concavity of the logarithm, we have $\log(x)  \le x-1$ for all $x \in (0, \infty)$ and hence
\begin{eqnarray*}
0 \le \int \log \left(  \frac{2 \widehat \pi_n}{\widehat \pi_n + \pi_0} \right)  d\mathbb P_n   & \le   & \int \left(\frac{2 \widehat \pi_n}{\widehat \pi_n + \pi_0} -1 \right)  d\mathbb P_n  \\
& = &  \int \frac{2 \widehat \pi_n}{\widehat \pi_n + \pi_0}  d(\mathbb P_n  - \mathbb P)  +   \int \frac{2 \widehat \pi_n}{\widehat \pi_n + \pi_0}  d\mathbb P -  1   \\
& =  &  \int \frac{2 \widehat \pi_n}{\widehat \pi_n + \pi_0}  d(\mathbb P_n  - \mathbb P)  +   \int \left(\frac{2 \widehat \pi_n}{\widehat \pi_n + \pi_0}  -1  \right) d\mathbb P  \\
& = &   \int \frac{2 \widehat \pi_n}{\widehat \pi_n + \pi_0}  d(\mathbb P_n  - \mathbb P)  +   \int \frac{\widehat \pi_n  - \pi_0}{\widehat \pi_n + \pi_0}   d\mathbb P.
\end{eqnarray*}
We have
\begin{eqnarray*}
\int \frac{\pi_0 - \widehat \pi_n}{\widehat \pi_n + \pi_0}   d\mathbb P  & =   &    \int \frac{\pi_0 - \widehat \pi_n}{\widehat \pi_n + \pi_0}  \pi_0 d\mu    \\
& =  &  \frac{1}{2} \int \frac{\pi_0 - \widehat \pi_n}{\widehat \pi_n + \pi_0}  (\pi_0  + \widehat \pi_n) d\mu   +  \frac{1}{2}  \int \frac{\pi_0 - \widehat \pi_n}{\widehat \pi_n + \pi_0}  (\pi_0  - \widehat \pi_n) d\mu  \\
& = &   \frac{1}{2} \int   (\pi_0 -  \widehat \pi_n)  d\mu  +  \frac{1}{2}  \int  \frac{(\pi_0 - \widehat \pi_n)^2}{\widehat \pi_n + \pi_0} d\mu  \\
& =  &  \frac{1}{2} \int \frac{(\pi_0 - \widehat \pi_n)^2}{\widehat \pi_n + \pi_0} d\mu  \\
& \ge & \frac{1}{2} \int \frac{(\pi_0 - \widehat \pi_n)^2}{\widehat \pi_n + \pi_0 + 2\sqrt{\widehat \pi_n \pi_0}} d\mu  \\
& = & \frac{1}{2} \int \frac{(\pi_0 - \widehat \pi_n)^2}{(\sqrt{\widehat \pi_n}+\sqrt{\pi_0})^2} d\mu  \\
& = & \frac{1}{2} \int (\sqrt{\widehat \pi_n}-\sqrt{\pi_0})^2 d\mu \\
& = & h^2(\widehat \pi_n, \pi_0),
\end{eqnarray*}
from which we conclude that 
\begin{eqnarray*}
h^2(\widehat \pi_n, \pi_0)  & \le  &   \int \frac{2 \widehat \pi_n}{\widehat \pi_n + \pi_0}  d(\mathbb P_n  - \mathbb P) \\
& = &  \int \left(\frac{2 \widehat \pi_n}{\widehat \pi_n + \pi_0}  - 1 \right) d(\mathbb P_n  - \mathbb P) \\
& = & \int \frac{\widehat{\pi}_n -  \pi_0}{\widehat{\pi}_n +  \pi_0}  d(\mathbb P_n -  \mathbb P).
\end{eqnarray*}
\hfill $\Box$

\section{Proofs and auxiliary results for Section 3}
 
\begin{lemma}\label{2in1}
Suppose that Assumptions (A1) and (A2) hold.  Let $\theta_m$ denote the largest point in the support of $Q_0$.  There exist $U \in \mathbb N$  and $\tilde{\delta}  > 0$  such that with probability $1$, there exists $n_0 \in \mathbb{N}$ such that for all $n \geq n_0$ and for all $k  \ge U$, we have that 
\begin{eqnarray*}
\frac{\widehat{\pi}_n(k)}{\pi_0(k)}  \le \frac{2}{\nu_0}  \left(\frac{\theta _m + \tilde{\delta}}{\theta_m- \tilde{\delta}}\right)^k
\end{eqnarray*}
and \begin{eqnarray*}
\frac{\theta_m  (\theta_m + \tilde{\delta})}{\theta_m- \tilde{\delta}}  \le  \left \{
\begin{array}{ll}
 \frac{(1 + q_0) R}{2}, \ \textrm{if $R < \infty$}  \\
 2 M, \  \  \  \  \ \  \textrm{if $R = \infty$}.
\end{array}
\right.
\end{eqnarray*}
 Above, $q_0$ and $M$ are the same constants as in Assumption (A1), $\delta_0$ the same as in Assumption (A2) and 
\begin{eqnarray*}
\nu_0  =  Q_0((\theta_m - \tilde{\delta}, \theta_m +\tilde{\delta}]).
\end{eqnarray*}

\end{lemma}

\medskip

\begin{proof}
Assumption (A1) implies that $\theta_m  < R$. Hence, for any small $\delta > 0$ such that $\theta_m + \delta < R$ we can use arguments similar to those of the proof of Lemma 2.2 of the main manuscript to show that there exists $U \in \mathbb N$ (depending on $\delta$) such that the map $\theta \mapsto f_\theta(k)$ is non-decreasing on $[0, \theta_m + \delta]$ for all $k \ge U$. Thus,
\begin{eqnarray*}
\pi_0(k)   \le  f_{\theta_m + \delta}(k), \  \ \textrm{for all $k \ge U$}.
\end{eqnarray*}
Since the set of discontinuity points of any non-negative measure is countable, and hence we can find $\delta  > 0$ such that  $\theta_m - \delta$ is a continuity point of the distribution function associated with the measure $Q_0$. With $\widehat{Q}_n$ being the NPMLE of $Q_0$, the latter means that with probability $1$, there exists $n_0 \in \mathbb N$ such that for all $n \ge n_0$
\begin{eqnarray*}
\widehat{Q}_n((\theta_m - \delta, \theta_m + \delta])  \ge \frac{1}{2}  Q_0((\theta_m - \delta, \theta_m + \delta])  = \frac{\nu_0}{2}.
\end{eqnarray*}
It follows that 
\begin{eqnarray*}
\widehat{\pi}_n(k)  &  \ge   &   \int_{(\theta_m - \delta, \theta_m + \delta]}  f_\theta(k)   d\widehat{Q}_n(\theta)  \\
& \ge &  \frac{\nu_0}{2} f_{\theta_m - \delta}(k)  
\end{eqnarray*}
for all $k \ge U$ and $n$ large enough almost surely.  Thus, with probability $1$ there exists $n_0$ such that for all $n \ge n_0$ and $k \ge U$
\begin{eqnarray*}
\frac{\pi_0(k)}{\widehat \pi_n(k)}  \le \frac{2}{\nu_0}  \left(\frac{\theta_m + \delta}{\theta_m - \delta}\right)^k.
\end{eqnarray*}
Suppose that $R < \infty$, and set 
\begin{eqnarray*}
\beta =  \frac{1-q_0}{2q_0}. 
\end{eqnarray*}
 Since $\delta$ can be arbitrarily small, we take in the following
\begin{eqnarray*}
\delta =\tilde{\delta}  \in \left(0, \frac{\beta \delta_0}{\delta_0 +2}\right],
\end{eqnarray*}
where $\delta_0$ is the same from Assumption (A2).   We will show now that 
\begin{eqnarray*}
\theta_m \frac{\theta_m + \tilde{\delta}}{\theta_m - \tilde{\delta}}  \le \frac{(1+q_0) R}{2}. 
\end{eqnarray*}
Since  $R < \infty$ we must have $\theta_m \le q_0 R$, and hence  
\begin{eqnarray*}
\theta_m \frac{\theta_m + \tilde{\delta}}{\theta_m - \tilde{\delta}} \le q_0  R  \frac{\theta_m + \tilde{\delta}}{\theta_m - \tilde{\delta}}.
\end{eqnarray*}
Thus, it is enough to show that 
\begin{eqnarray*}
\frac{\theta_m + \tilde{\delta}}{\theta_m - \tilde{\delta}}   \le \frac{1 + q_0}{2 q_0}.
\end{eqnarray*}
This is equivalent to showing that 
\begin{eqnarray*}
\frac{2\tilde{\delta}}{\theta_m - \tilde{\delta}}   \le \frac{1-q_0}{2q_0}  =  \beta
\end{eqnarray*}
or that
\begin{eqnarray*}
\tilde{\delta} \le \frac{\beta \theta_m /2}{1 + \beta/2}. 
\end{eqnarray*}
Now, since $\theta_m \ge \delta_0$, the previous inequality is fulfilled if 
\begin{eqnarray*}
\tilde{\delta} \le \frac{\beta \delta_0 /2}{1 + \beta/2}  =  \frac{\beta \delta_0}{2 + \beta}
\end{eqnarray*}
which is true by definition of $\tilde{\delta}$.  If $R = \infty$, we take $\delta = \tilde{\delta}  \in (0, \delta_0/3]$.  We will show that in this case 
\begin{eqnarray*}
\theta_m \frac{\theta_m + \tilde{\delta}}{\theta_m - \tilde{\delta}}   \le 2 M.
\end{eqnarray*}
Since  $\theta_m \le M$ it is enough to show that 
\begin{eqnarray*}
\frac{\theta_m + \tilde{\delta}}{\theta_m - \tilde{\delta}}   \le 2,
\end{eqnarray*}
or equivalently
\begin{eqnarray*}
3 \tilde{\delta}  \le \theta_m
\end{eqnarray*}
which is true since $3 \tilde{\delta} \le \delta_0  \le \theta_m$. 
\end{proof}

\medskip

\par \noindent \textbf{Proof of Theorem 3.1.}    Let us start with the first rate. Fix $\epsilon \in (0,1/2)$, choose $\gamma>0$ such that $\gamma \in (1-2\epsilon,1)$. By Funbin's theorem and Lemma 2.3 in the main manuscript, we have for some $U \in \mathbb{N}$ and any $t>0$
\begin{eqnarray*}
\sum_{k \in \mathbb{N}} \pi_{0}(k) \exp(tk)
& = & \sum_{k \in \mathbb{N}} \int_{\Theta}  f_\theta(k) \exp(tk)  dQ_0(\theta)  \\
& = & \sum_{k \in \mathbb{N} \cap \{k < U\}}  \int_{\Theta}  f_\theta(k) \exp(tk)  dQ_0(\theta) \\
&& \ + \sum_{k \in \mathbb{N} \cap \{k \ge U\}}  \int_{\Theta}  f_\theta(k) \exp(tk)  dQ_0(\theta) \\
& = & \textrm{cst.} + \sum_{k \in \mathbb{N} \cap \{k \ge U\}} \int_{\Theta}  f_\theta(k) \exp(tk)  dQ_0(\theta) \\
& \leq & \textrm{cst.} +  \sum_{k \in \mathbb{N} \cap \{k \ge U\}}  f_{\tilde{\theta}}(k) \exp(tk) \\
& = & \textrm{cst.} + \sum_{k \in \mathbb{N}} \frac{b_k (\tilde{\theta} e^t)^k}{b(\tilde{\theta})}
\end{eqnarray*}
where $\tilde{\theta}  =  (q_0 R) \mathds{1}_{\{R < \infty\}}  +  M \mathds{1}_{\{R = \infty\}}$  ($q_0 \in (0,1)$ and $M > 0$ are from Assumption (A1)).   Hence, if $R < \infty$, the sum on the right side of the previous display is finite for $t > 0$ such that $e^t < 1/q_0$; i.e., $t \in (0, \log(1/q_0))$. A possible choice is $t =  \log(1/\sqrt{q_0})$. If $R = \infty$, then we can choose $t =1$.\\

\par \noindent Let us write $a=e^t$, that is, $a = 1/\sqrt{q_0}$ if $R < \infty$ and $a=e$ otherwise. Note that $a>1$. Then, it follows from the calculations above that $  \sum_{k \in \mathbb{N}} \pi_{0}(k)a^{k} < \infty$.   Now, for this choice of $a$, set $A_a:=a^{(1-\gamma)/3} >1$. We have that
\begin{eqnarray*}
n^{1-2\epsilon} \sum_{k \in \mathbb{N}} \frac{(\bar \pi_n(k)-\pi_{0}(k))^2}{\pi_{0}(k)}
= n^{1-2\epsilon} \sum_{k \in \mathbb{N}} A_a^{-k} A_a^{k} \frac{(\bar \pi_n(k)-\pi_{0}(k))^2}{\pi_{0}(k)}.
\end{eqnarray*}
Define the positive measure $\mu$ on $\mathbb{N}$ through 
\begin{eqnarray*}
\mu(A) := \sum_{k \in A} \frac{(\bar \pi_n(k)-\pi_{0}(k))^2}{\pi_{0}(k)}
\end{eqnarray*}
 for $A \in 2^{\mathbb{N}}$. By applying the H\"{o}lder inequality to the functions $ k \mapsto A_a^{-k}$ and $ k \mapsto A_a^{k}$ and the measure $\mu$, we are able to bound the upper expression by
\begin{eqnarray*}
&&  n^{1-2\epsilon} \Big[\sum_{k \in \mathbb{N}} (A_a^{-k})^{1/\gamma} \frac{(\bar \pi_n(k)-\pi_{0}(k))^2}{\pi_{0}(k)}\Big]^\gamma \Big[\sum_{k \in \mathbb{N}} (A_a^{k})^{1/(1-\gamma)} \frac{(\bar \pi_n(k)-\pi_{0}(k))^2}{\pi_{0}(k)}\Big]^{1-\gamma} \\
&&  =n^{1-2\epsilon} \Big[\sum_{k \in \mathbb{N}} \frac{(\bar \pi_n(k)-\pi_{0}(k))^2}{\pi_{0}(k)} A_a^{-k/\gamma}\Big]^\gamma \Big[\sum_{k \in \mathbb{N}} \frac{(\bar \pi_n(k)-\pi_{0}(k))^2}{\pi_{0}(k)} A_a^{k/(1-\gamma)}\Big]^{1-\gamma}.
\end{eqnarray*}
We can write
\begin{eqnarray*}
\mathbb{E}\left[n \sum_{k \in \mathbb{N}} \frac{(\bar \pi_n(k)-\pi_{0}(k))^2}{\pi_{0}(k)} A_a^{-k/\gamma}\right]
& = & \sum_{k \in \mathbb{N}} n \mathbb{E}\left[ \frac{(\bar \pi_n(k)-\pi_{0}(k))^2}{\pi_{0}(k)}\right] A_a^{-k/\gamma}  \\
& = & \sum_{k \in \mathbb{N}} n  \frac{\pi_0(k)(1-\pi_0(k))}{n \pi_0(k)} A_a^{-k/\gamma} \\
& < & \sum_{k \in \mathbb{N}} A_a^{-k/\gamma}  < \infty
\end{eqnarray*}
where we used $A_a >1$. Hence, $n \sum_{k \in \mathbb{N}} \frac{(\bar \pi_n(k)-\pi_{0}(k))^2}{\pi_{0}(k)} A_a^{-k/\gamma}  =  O_{\mathbb P}(1)$.  Consider now the term in the second brackets. We have  that
\begin{eqnarray*}
\sum_{k \in \mathbb{N}} \frac{(\bar \pi_n(k)-\pi_{0}(k))^2}{\pi_{0}(k)} A_a^{k/(1-\gamma)}
& = & \sum_{k \in \mathbb{N}} \frac{(\bar \pi_n(k)-\pi_{0}(k))^2}{\pi_{0}(k)} a^{k/3} \\
& \leq & \sum_{k \in \mathbb{N}} \frac{\bar \pi_n(k)^2}{\pi_{0}(k)} a^{k/3} + \sum_{k \in \mathbb{N}} \pi_{0}(k) a^{k/3} \\
& \leq & \sum_{k \in \mathbb{N}} \frac{\bar \pi_n(k)^2}{\pi_{0}(k)} a^{k/3} + \sum_{k \in \mathbb{N}} \pi_{0}(k) a^k.
\end{eqnarray*}
By the above calculations, we know that the second term is finite. For the first term, we proceed as follows. For any fixed $l > 0$, define
\begin{eqnarray*}
E_n(l):= \Big \{\exists \ k \text{ such that } \bar \pi_n(k)>l\pi_{0}(k) a^{k/3}   \Big\}.
\end{eqnarray*}
Using the union bound, the Markov's inequality and $\mathbb{E}[\bar \pi_n(k)]=\pi_{0}(k)$, we obtain
\begin{eqnarray*}
\mathbb{P}(E_n(l))
&\leq  &  \sum_{k \in \mathbb{N}} \mathbb{P}\left(\bar \pi_n(k)> l\pi_{0}(k) a^{k/3}\right) \\
&\leq  &  l^{-1}\sum_{k \in \mathbb{N}} \frac{\mathbb{E}[\bar \pi_n(k)]}{\pi_{0}(k)a^{k/3}}\\
& =   & l^{-1}\sum_{k \in \mathbb{N}} a^{-k/3} \\
& =   &  l^{-1} (1- a^{-1/3})^{-1}  
\end{eqnarray*}
where
\begin{eqnarray*}
(1- a^{-1/3})^{-1} = \left \{
\begin{array}{ll}
\frac{1}{1- q^{2/3}_0}, \  \ \textrm{if $R < \infty$}  \\
\smallskip \\
\frac{1}{1 - e^{-1/3}}, \  \  \textrm{otherwise}.
\end{array}
\right.
\end{eqnarray*}
Hence, $\mathbb{P}(E_n(l))$ can be made arbitrarily small by choosing $l$ sufficiently large. Now note that on  the complement  $E_n(l)^c$, we have that
\begin{eqnarray*}
\sum_{k \in \mathbb{N}} \frac{\bar \pi_n(k)^2}{\pi_{0}(k)} a^{k/3} \leq l^2 \sum_{k \in \mathbb{N}} \pi_{0}(k) a^k < \infty.
\end{eqnarray*}
This implies that 
\begin{eqnarray*}
\sum_{k \in \mathbb{N}} \frac{(\bar \pi_n(k)-\pi_{0}(k))^2}{\pi_{0}(k)} A_a^{k/(1-\gamma)}  = O_{\mathbb P}(1),
\end{eqnarray*}
from which we conclude that 
\begin{eqnarray*}
n^{1-2\epsilon} \sum_{k \in \mathbb{N}} \frac{(\bar \pi_n(k)-\pi_{0}(k))^2}{\pi_{0}(k)}  =  O_{\mathbb P}(n^{1-2\epsilon - \gamma})  = o_{\mathbb P}(1)
\end{eqnarray*}
using the fact that $\gamma > 1-2\epsilon$. This proves the rate in the first claim of the theorem.    Let us move to  proving the second claim. To this  goal, we will use Lemma \ref{2in1}. Note that for any integer $U \ge 0$
\begin{eqnarray*}
n^{1-2\epsilon}\sum_{k  \le U} \frac{(\bar \pi_n(k)-\pi_{0}(k))^2}{\widehat \pi_n(k)}  =  O_{\mathbb P}(n^{-2\epsilon})
\end{eqnarray*}
which follows from the $1/\sqrt n$-rate of $\bar{\pi}_n$ and uniform consistency of the NPMLE $\widehat{\pi}_n$ on a finite set of integers. Thus, and without loss of generality, we can take $U =0$ in Lemma \ref{2in1}. Define now
\begin{eqnarray*}
a:=  \left(\frac{\theta_m + \tilde{\delta}}{\theta_m - \tilde{\delta}}\right)^{3/(1-\gamma)}
\end{eqnarray*}
with $\tilde{\delta}$ the same as in Lemma \ref{2in1}.  Then, $A_a= a^{(1-\gamma)/3}=(\theta_m + \tilde{\delta})/(\theta_m - \tilde{\delta})  >1$.
Fix $\epsilon \in (0,1/2)$, and let $\gamma \in (0,1)$ to be chosen later. For all $n \geq n_0$, we obtain
\begin{eqnarray*}
&& n^{1-2\epsilon}\sum_{k \in \mathbb{N}} \frac{(\bar \pi_n(k)-\pi_{0}(k))^2}{\widehat \pi_n(k)}\\
&&= n^{1-2\epsilon}\sum_{k \in \mathbb{N}} \frac{(\bar \pi_n(k)-\pi_{0}(k))^2}{\pi_0(k)} \frac{\pi_{0}(k)}{\widehat \pi_n(k)} \\
&& \leq \frac{2}{\nu_0} n^{1-2\epsilon}  \sum_{k \in \mathbb{N}} \frac{(\bar \pi_n(k)-\pi_{0}(k))^2}{\pi_0(k)} \Big(\frac{\theta_m + \tilde{\delta}}{\theta_m - \tilde{\delta}}\Big)^k \\
&& =\frac{2}{\nu_0}  n^{1-2\epsilon} \sum_{k \in \mathbb{N}} \frac{(\bar \pi_n(k)-\pi_{0}(k))^2}{\pi_0(k)} A_a^k \\
&& \leq \frac{2}{\nu_0}  n^{1-2\epsilon} \Big[\sum_{k \in \mathbb{N}} \frac{(\bar \pi_n(k)-\pi_{0}(k))^2}{\pi_{0}(k)} \Big]^{\gamma} \Big[\sum_{k \in \mathbb{N}} \frac{(\bar \pi_n(k)-\pi_{0}(k))^2}{\pi_{0}(k)}A_a^{k/(1-\gamma)}\Big]^{1-\gamma},  \\
&&  \  \ \textrm{using the H\"{o}lder inequality}.
\end{eqnarray*}
Now, fix a small $\epsilon' \in (0, \epsilon)$. Using the first convergence result above, we have that 
\begin{eqnarray*}
\sum_{k \in \mathbb{N}} \frac{(\bar \pi_n(k)-\pi_{0}(k))^2}{\pi_{0}(k)}  =  o_{\mathbb P}(n^{-1+2\epsilon'}).
\end{eqnarray*}
Now,   let $\gamma \in ((1-2\epsilon)/(1-2\epsilon'), 1)$. Then, 
\begin{eqnarray*}
n^{1-2\epsilon} \Big[\sum_{k \in \mathbb{N}} \frac{(\bar \pi_n(k)-\pi_{0}(k))^2}{\pi_{0}(k)} \Big]^{\gamma} = o_{\mathbb P}( n^{1-2\epsilon  - \gamma (1-2 \epsilon'))}  =  o_{\mathbb P}(1).
\end{eqnarray*}
Using Lemma 2.2 of the main manuscript (with $U=0$) and Lemma \ref{2in1}, we can write 
\begin{eqnarray*}
\sum_{k \in \mathbb N}  \pi_0(k)  \left(\frac{\theta_m + \tilde{\delta}}{\theta_m -\tilde{\delta}}\right)^k  & \le  &    \sum_{k \in \mathbb N}  f_{\theta_m}(k)  \left(\frac{\theta_m + \tilde{\delta}}{\theta_m -\tilde{\delta}}\right)^k  \\
& = & \frac{1}{b(\theta_m)}  \sum_{k \in \mathbb K}  b_k  \left(\frac{\theta_m (\theta_m + \tilde{\delta})}{\theta_m -\tilde{\delta}}\right)^k   \\
& \le &  \frac{1}{b(\delta_0)} \left \{
\begin{array}{ll}
\sum_{k \in \mathbb K}  b_k  \left(\frac{(1+ q_0)R}{2} \right)^k, \  \ \textrm{if $R < \infty$}  \\
\sum_{k \in \mathbb K}  b_k  (2M)^k, \  \hspace{1cm} \textrm{if $R = \infty$}  
\end{array}
\right.  \\
& < &  \infty. 
\end{eqnarray*}
Thus,  similar arguments as above can be used to show that 
$$
\sum_{k \in \mathbb{K}} \frac{(\bar \pi_n(k)-\pi_{0}(k))^2}{\pi_{0}(k)}A_a^{k/(1-\gamma)}  = O_{\mathbb P}(1).
$$
 This completes the proof.  \hfill $\Box$

\medskip

\par \noindent \textbf{Proof of Proposition 3.3.}   The claim is trivial for $\alpha=0$. Let $\alpha \in (0, 1)$. Then,  
\begin{eqnarray*}
\Vert \pi_0\Vert^2_{n,\alpha}  &  =   & \sum_{k \ge 0}  \frac{(\pi_0(k))^2}{\widehat{\pi}^\alpha_n(k)}  =   \sum_{k \ge 0}  \frac{(\pi_0(k))^2}{\widehat{\pi}_n(k)} \\
                                               & =  &   \sum_{k \ge 0}  \frac{(\pi_0(k)  - \bar{\pi}_n(k) )\pi_0(k) + \bar{\pi}_n(k)  \pi_0(k)}{\widehat{\pi}_n(k)} \\
& \le & 1 +  \sum_{k \ge 0}   \frac{\vert \bar{\pi}_n(k)  - \pi_0(k) \vert  \pi_0(k) }{\widehat{\pi}_n(k)}
\end{eqnarray*}
using the triangle inequality and the fact that $\sum_{k \ge 0} \bar{\pi}_n(k) \pi_0(k) /\widehat{\pi}_n(k) \le 1$. The latter follows follows from the optimization properties of the NPMLE. Thus,
\begin{eqnarray*}
\sum_{k \ge 0}  \frac{(\pi_0(k))^2}{\widehat{\pi}_n(k)}  & \le  &  1 +  \sum_{k \ge 0}   \frac{\vert \bar{\pi}_n(k)  - \pi_0(k)  \vert}{\sqrt{\widehat{\pi}_n(k)}} \frac{\pi_0(k)}{\sqrt{\widehat{\pi}_n(k)}} \\
& \le &   1 +  \left( \sum_{k \ge 0}  \frac{\vert \bar{\pi}_n(k)  - \pi_0(k)  \vert^2}{\widehat{\pi}_n(k)}\right)^{1/2} \left(\sum_{k \ge 0}  \frac{(\pi_0(k))^2}{\widehat{\pi}_n(k)}\right)^{1/2}, \\
&& \ \  \ \textrm{by the Cauchy-Schwarz inequality}.   
\end{eqnarray*}
Put 
\begin{eqnarray*}
A_n =   \left(\sum_{k \ge 0}  \frac{(\pi_0(k))^2}{\widehat{\pi}_n(k)}\right)^{1/2}, \ \ \textrm{and}  \  \  B_n =  \left( \sum_{k \ge 0}  \frac{\vert \bar{\pi}_n(k)  - \pi_0(k)  \vert^2}{\widehat{\pi}_n(k)}\right)^{1/2}.
\end{eqnarray*}
By Theorem 3.4 of the main manuscript,  we know that $B_n = o_{\mathbb P}(n^{-1/2 + \epsilon})$ for any $\epsilon > 0$. In particular this implies that $B_n \le 1$  with probability tending to $1$. Hence, with probability tending to $1$ we have that 
\begin{eqnarray*}
A_n^2  - A_n  - 1 \le 0
\end{eqnarray*}
or equivalently
\begin{eqnarray*}
\left(A_n - \frac{1}{2} \right)^2 \le \frac{5}{4}
\end{eqnarray*}
which implies that $ A_n  \in (0, (1+\sqrt{5})/2]$ with probability tending to $1$. This shows that $ \Vert \pi_0\Vert_{n,\alpha}$ is finite in probability.    \hfill $\Box$

\medskip
\medskip

\par \noindent \textbf{Proof of Theorem 3.4.}    By definition of $\widetilde{\pi}_{n , \alpha}$ it holds that 
\begin{eqnarray*}
\sum_{k \ge 0}  \frac{(\widetilde{\pi}_{n, \alpha}(k)  - \bar{\pi}_n(k))^2 }{\widehat{\pi}^\alpha_{n}(k)}  \le \sum_{k \ge 0}  \frac{(\bar{\pi}_n(k)  - \pi_0(k))^2 }{\widehat{\pi}^\alpha_{n}(k)}
\end{eqnarray*}
which, in combination with 
\begin{eqnarray*}
\sum_{k \ge 0}  \frac{(\widetilde{\pi}_{n, \alpha}(k)  - \pi_0(k))^2 }{\widehat{\pi}^\alpha_{n}(k)}  \le  2 \sum_{k \ge 0}  \frac{(\widetilde{\pi}_{n, \alpha}(k)  - \bar{\pi}_n(k))^2 }{\widehat{\pi}^\alpha_{n}(k)}  + 2 \sum_{k \ge 0}  \frac{(\bar{\pi}_n(k)  - \pi_0(k))^2 }{\widehat{\pi}^\alpha_{n}(k)},
\end{eqnarray*}
yields
\begin{eqnarray*}
\sum_{k \ge 0}  \frac{(\widetilde{\pi}_{n, \alpha}(k)  - \pi_0(k))^2 }{\widehat{\pi}^\alpha_{n}(k)}  \le  4 \sum_{k \ge 0}  \frac{(\bar{\pi}_n(k)  - \pi_0(k))^2 }{\widehat{\pi}^\alpha_{n}(k)}.
\end{eqnarray*}
Using $\widehat{\pi}^\alpha_n \ge \widehat{\pi}^{1/2}_n$ it follows that
\begin{eqnarray*}
\sum_{k \ge 0}  \frac{(\bar{\pi}_{n}(k)  - \pi_0(k))^2 }{\widehat{\pi}^\alpha_{n}(k)}   & \le   &  \sum_{k \ge 0}  \frac{(\bar{\pi}_{n}(k)  - \pi_0(k))^2 }{\sqrt{\widehat{\pi}_{n}(k)}}   \\
&  \le  &   \sum_{k \ge 0}  \frac{(\bar{\pi}_n(k)  - \pi_0(k))^2 }{\sqrt{\pi_0(k)}}  +   \sum_{k \ge 0}  (\bar{\pi}_n(k)  - \pi_0(k))^2 \left \vert \frac{1}{\sqrt{\pi_0(k)}}   -  \frac{1}{\sqrt{\widehat{\pi}_{n}(k)}} \right \vert \\
& = & I_n + II_n
\end{eqnarray*}
with $I_n = O_{\mathbb P}(1/n)$ and 
\begin{eqnarray*}
II_n & = & \sum_{k \ge 0}  \frac{\vert \bar{\pi}_n(k)  - \pi_0(k)  \vert}{\sqrt{\pi_0(k)}} \frac{\vert \bar{\pi}_n(k)  - \pi_0(k)  \vert}{\sqrt{\widehat{\pi}_n(k)}} \left \vert \sqrt{\widehat{\pi}_n(k)} - \sqrt{\pi_0(k)} \right \vert \\
& \le &  \left( \sum_{k \ge 0}  \frac{(\bar{\pi}_n(k)  - \pi_0(k))^2}{\pi_0(k)}\right)^{1/2} \left(\sum_{k \ge 0} \frac{(\bar{\pi}_n(k)  - \pi_0(k))^2}{\widehat \pi_n(k)} \left (\sqrt{\widehat{\pi}_n(k)} - \sqrt{\pi_0(k)} \right)^2  \right)^{1/2}, \\
&&  \  \textrm{by the Cauchy-Schwarz inequality}\\
& \le &  \left( \sum_{k \ge 0}  \frac{(\bar{\pi}_n(k)  - \pi_0(k))^2}{\pi_0(k)}\right)^{1/2}\left(\sum_{k \ge 0} \frac{(\bar{\pi}_n(k)  - \pi_0(k))^2}{\widehat \pi_n(k)}\right)^{1/2}  \sup_{k \ge 0}  \left \vert \sqrt{\widehat{\pi}_n(k)} - \sqrt{\pi_0(k)} \right \vert \\
& = &  o_{\mathbb P}(n^{-1/2 + \epsilon})   o_{\mathbb P}(n^{-1/2 + \epsilon})  O_{\mathbb P}( n^{-1/2} \log n)^{3/2} )
\end{eqnarray*}
using the convergence rates of Theorem 3.4 and Theorem 2.1 (see the main manuscript). We conclude that
\begin{eqnarray*}
II_n =  o_{\mathbb P}(n^{-3/2 + 2\epsilon}   (\log n)^{3/2})  =  o_{\mathbb P}(n^{-1}).
\end{eqnarray*}
It follows that 
\begin{eqnarray*}
\sum_{k \ge 0}  \frac{(\widetilde{\pi}_{n, \alpha}(k)  - \pi_0(k))^2 }{\widehat{\pi}^\alpha_{n}(k)} =  O_{\mathbb P}\left(\frac{1}{n}\right)
\end{eqnarray*}
uniformly in $\alpha \in [0,1/2]$, which concludes the proof.   For the second statement of the theorem, we use that $1/\widehat{\pi}_{n, \alpha}  \ge 1$ to conclude that for all $p \in [2, \infty]$
\begin{eqnarray*}
\left(\sum_{k \ge 0} \left \vert \widetilde{\pi}_{n, \alpha}(k) -\pi_0(k) \right \vert^p \right)^{1/p}  & \le  &   \left(\sum_{k \ge 0} \left \vert \widetilde{\pi}_{n, \alpha}(k) -\pi_0(k) \right \vert^2\right)^{1/2} \\
&  \le  &  \left(\sum_{k \ge 0}  \frac{(\widetilde{\pi}_{n, \alpha}(k)  - \pi_0(k))^2 }{\widehat{\pi}^{1/2}_n(k)}\right)^{1/2}
\end{eqnarray*}
and the proof is complete by taking the supremum over $\alpha \in [0,1/2]$ in the three terms of the previous display. \hfill $\Box$

\medskip

\par \noindent \textbf{Proof of Property 1 of Proposition 3.7.}
We know from Theorem 2.2 that
\begin{eqnarray*}
\sum_{k \in \mathbb{N}}  \vert \widehat \pi_n(k) - \pi_0(k) \vert = O_{\mathbb P}\Big(\frac{(\log n)^{3/2}}{\sqrt{n}}\Big) = o_{\mathbb P}\Big(\frac{1}{ (\log n)^{3}}\Big).
\end{eqnarray*}
This implies
\begin{eqnarray*}
\sum_{k > \tilde{K}_n}  \pi_0(k)
\le \sum_{k \in \mathbb{N}}  \vert \widehat \pi_n(k) - \pi_0(k) \vert + \sum_{k > \tilde{K}_n}  \widehat \pi_n(k)   \le \frac{2}{(\log n)^{3}}.
\end{eqnarray*}
From Property 3 of Lemma 2.3, we know that we can find an integer $K \ge 1$ such that \begin{eqnarray*}
\sum_{k \ge K+1} \pi_0(k) \le A t_0^{K},
\end{eqnarray*}
where  $t_0  = (q_0+1)/2  \mathds{1}_{\{R < \infty\}} +  (1/2) \mathds{1}_{\{R = \infty\}} \in (0,1)$ and $A > 0$ is the same constant defined in (\ref{A}). Note that 
\begin{eqnarray*}
A t_0^{K} \le \frac{2}{(\log n)^{3}}.
\end{eqnarray*}
implies that
\begin{eqnarray*}
K \ge \frac{1}{\log (1/t_0)} \log\Big(\frac{A}{2}(\log n)^{3}\Big).
\end{eqnarray*}
Thus, for such $K$ we must have that
\begin{eqnarray*}
\sum_{k \ge K+1} \pi_0(k) \le \frac{2}{(\log n)^{3}}.
\end{eqnarray*}
Now, note that
\begin{eqnarray*}
\frac{1}{\log(1/t_0)} \log\Big(\frac{A}{2}(\log n)^{3}\Big)
& = &  \frac{1}{\log(1/t_0)} \log\Big(\frac{A}{2}\Big) + \frac{3}{\log(1/t_0)}  \log(\log n)  \\
&\le & \frac{4}{\log(1/t_0)} \log(\log n),
\end{eqnarray*}
for $n$ large enough. Thus, by definition of $\tilde{K}_n$, we have for large enough $n$
\begin{eqnarray*}
\tilde{K}_n  + 1  \le \frac{4}{\log(1/t_0)} \log(\log n)  +1  \le \frac{5}{\log(1/t_0)} \log(\log n)  =: d \log(\log n).
\end{eqnarray*}
Without loss of generality (and also for convenience), we assume that $d \log(\log n)$ is an integer. We assume in the following that $Q_0(\{0\})=0$, which means by Assumption (A2) that $Q_0([0, \delta_0]) =0$ (if $Q_0(\{0\}) > 0$, a similar reasoning yields the same conclusions).  Using Property 1 and 4 of Lemma 2.3 we can write 
\begin{eqnarray*}
\Big(1-\pi_0(\tilde{K}_n)\Big)^n  &\le & \Big(1-\pi_0(d \log(\log n))\Big)^n   = \left( 1- \int_\Theta f_{\theta}(d \log(\log n)) dQ_0(\theta) \right)^n\\
& \le &  \Big(1-  f_{\delta_0}(d \log(\log n)) \Big)^ n  =  \left( 1-  \frac{b_{d \log(\log n)} \delta_0^{d \log(\log n)}}{b(\delta_0)}  \right)^n.
\end{eqnarray*}
Using Assumption (A3), we have that 
\begin{eqnarray*}
\frac{b_{d \log(\log n)} \delta_0^{d \log(\log n)}}{b(\delta_0)} \ge  \frac{b_0}{b(\delta_0)}  (d \log(\log n))^{-d \log(\log(n))} \delta_0^{d \log(\log n)},
\end{eqnarray*}
and hence for $n$ large enough
\begin{eqnarray*}
\Big(1-\pi_0(\tilde{K}_n)\Big)^n & \le &  \left(  1-  \frac{b_0}{b(\delta_0)}  (d \log(\log n))^{-d \log(\log(n))} \delta_0^{d \log(\log n)}  \right)^n. 
\end{eqnarray*}
Hence, we have for $n$ large enough
\begin{eqnarray*}
&& \Big(\tilde{K}_n +1\Big) \Big(1-\pi_0(\tilde{K}_n)\Big)^n\\
&&\le d \log(\log n)  \left(  1-  \frac{b_0}{b(\delta_0)}  (d \log(\log n))^{-d \log(\log(n))} \delta_0^{d \log(\log n)}  \right)^n \\
&& =  \exp\left\{ \log d + \log(\log(\log n)) + n \log\Big(1- \frac{b_0}{b(\delta_0)}  (d \log(\log n))^{-d \log(\log(n))} \delta_0^{d \log(\log n)} \Big)\right \} \\
&& \le \exp\left \{\log d + \log(\log(\log n)) - n  \frac{b_0}{b(\delta_0)} \left(\frac{\delta_0}{d \log(\log n)}\right)^{d \log(\log(n))}   \right \}. 
\end{eqnarray*}
Now, note that 
\begin{eqnarray*}
n  \left(\frac{\delta_0}{d \log(\log n)}\right)^{d \log(\log(n))} & =  & \exp \left\{ \log n + d \log(\log n) \log(\delta_0)  - d \log (\log n) \log( d \log (\log n))  \right\}  \\
 &\ge & \exp((\log n)/2) =   \sqrt n 
\end{eqnarray*}
for $n$ large enough, which implies that
\begin{eqnarray*}
\log(\log(\log n)) - n  \frac{b_0}{b(\delta_0)} \left(\frac{\delta_0}{d \log(\log n)}\right)^{d \log(\log(n))}  \le  \log(\log(\log n))  -  \frac{b_0}{b(\delta_0)} \sqrt n  \searrow -\infty.
\end{eqnarray*}
This completes the proof.
\hfill $\Box$

\section{Additional theorems}

\begin{theorem}\label{finitemix}
Fix $\alpha \in [0, 1/2]$.   Consider  $\check{\pi}_{n, \alpha}$  to be the unique minimizer of $Q_{n, \alpha}$ over the space $\mathcal M'$ of 
\begin{eqnarray*}
\pi(k; Q)  =  \int_{\Theta}  f_\theta(k)  dQ(\theta), \  \ k \in \mathbb N,
\end{eqnarray*} 
where $Q$ is a positive and finite measure on $\Theta$.  Then, as $n \to \infty$, $\check{\pi}_{n, \alpha}$ is a finite mixture with at most $\max_{1 \le i \le n} X_i+1$ components with probability tending to $1$.
\end{theorem}

\medskip

\begin{proof}  In the following, we will most of the time drop the statement \lq\lq with probability tending to $1$\rq\rq \ while keeping in mind that all the properties  proved below are only true in probability.  Let $\theta \in [0, R)$.  If $\theta$ belongs to the support of the mixing measure of $\check{\pi}_{n, \alpha}$, $\tilde{F}_{n, \alpha}$, then we must have 
\begin{eqnarray*}\label{eq}
\lim_{\epsilon \to 0} \frac{1}{\epsilon} \Big( Q_{n, \alpha}((1-\epsilon) \check{\pi}_{n, \alpha}  + \epsilon f_{\theta})  -  Q_{n, \alpha}(\check{\pi}_{n, \alpha}) \Big) = 0 
\end{eqnarray*}
which can be rewritten as
\begin{eqnarray*}
\sum_{k \ge 0} \frac{\left(\check{\pi}_{n, \alpha}(k) - \bar{\pi}_n(k) \right) \left(f_{\theta}(k)  -  \check{\pi}_{n, \alpha}(k) \right)}{\widehat{\pi}^\alpha_n(k)}  & = &  0
\end{eqnarray*}
or equivalently
\begin{eqnarray}\label{eqFench}
\sum_{k \ge 0} \frac{\left(\check{\pi}_{n, \alpha}(k) - \bar{\pi}_n(k) \right) f_{\theta}(k) }{\widehat{\pi}^\alpha_n(k)}  & =  &   \sum_{k \ge 0} \frac{\left(\check{\pi}_{n, \alpha}(k) - \bar{\pi}_n(k) \right) \check{\pi}_{n, \alpha}(k)  }{\widehat{\pi}^\alpha_n(k)} \\
& = & 0 \notag
\end{eqnarray}
where the last equality follows from seeing that the term in (\ref{eqFench})  is equal to 
\begin{eqnarray*}
\lim_{\epsilon \to 0} \frac{1}{\epsilon} \Big( Q_{n, \alpha}(\check{\pi}_{n, \alpha}  + \epsilon \check{\pi}_{n, \alpha})  -  Q_{n, \alpha}(\check{\pi}_{n, \alpha}) \Big). 
\end{eqnarray*}
Write 
\begin{eqnarray*}
w_{n, \alpha}(k)  =  \frac{\left(\check{\pi}_{n, \alpha}(k)  - \bar{\pi}_n(k) \right) b_k}{\widehat{\pi}^\alpha_n}.
\end{eqnarray*}
It follows from the calculations above that for $\theta$ in the support of $\tilde{F}_{n, \alpha}$
\begin{eqnarray}\label{id}
\sum_{k \ge 0}  w_{n, \alpha}(k)  \theta^k =0.
\end{eqnarray}
Now, suppose that $\check{\pi}_{n , \alpha}$ has at least $X_{(n)}+2$ components in $[0, R)$. Write $N = X_{(n)}+1$. Then, using the well-known fact that a power series on $[0, R)$ is smooth on $(0, R)$ and the mean value theorem it follows that there exists $\theta_*$ which a root of the $N$-th derivative  of the function  on the left-hand side of (\ref{id}). In other words, we must have 
\begin{eqnarray*}
\sum_{k \ge N}  w_{n, \alpha}(k)  k (k-1) \ldots (k - N+1) \theta^{k- N}_{*}   & =  &   0.
\end{eqnarray*}
Using the fact that $\bar{\pi}(k) =0$ for $k \ge N$, it follows that
\begin{eqnarray*}
\sum_{k \ge N}  \frac{\check{\pi}_{n, \alpha}(k)}{\widehat{\pi}_n(k)} b_k  k (k-1) \ldots (k - N+1) \theta^{k- N}_{*}  & = &  0
\end{eqnarray*}
which is impossible.  
\end{proof}

\medskip
\medskip

\begin{theorem}
If $\mathbb K$ is finite, then
\begin{eqnarray*}
h(\widehat \pi_n, \pi_0)   = O_{\mathbb P} \left( \frac{1}{\sqrt n}\right)
\end{eqnarray*}
\end{theorem}
\begin{proof}
We are interested in the class of functions
\begin{eqnarray*}
\mathcal{G}(\delta)  :=  \left \{ k \mapsto g(k)  =  \frac{\pi(k)  - \pi_0(k)}{\pi(k)  +  \pi_0(k)}, k \in \mathbb K:  h(\pi, \pi_0) \le \delta\right\}.
\end{eqnarray*}
If $K$ is the cardinality of $\mathbb K$, then the $\nu$-bracketing entropy of this class can be easily shown to be upper bounded by  
\begin{eqnarray*}
K \log \left(\frac{c \delta}{\nu} \right)
\end{eqnarray*}
for some constant $c > 0$ which depends only on the $\inf_{k \in \mathbb K} \pi_0(k)  > 0$.  Then, 
\begin{eqnarray*}
\tilde{J}_B(\delta, \mathcal G, \mathbb P) \le  \int_0^\delta  \sqrt{1 +  K \log \left(\frac{c \delta}{u} \right)}  du  \le \delta  + \sqrt K  \int_0^\delta \sqrt{\log \left(\frac{c \delta}{u} \right)}  du  \lesssim \delta
\end{eqnarray*}
using our calculations from above.  When solving $\sqrt n \delta^2_n  = \delta_n$, we find $\delta_n = 1/\sqrt n$. The proof for showing that this is indeed the rate for $h(\widehat \pi_n, \pi_0)$ will go along the same lines as for bounding the probability $P_2$ as in the proof of Theorem 2.1  in the main manuscript.
\end{proof}

\bibliographystyle{Chicago}

\begin{thebibliography}{99}

\bibitem[\protect\citeauthoryear{Balabdaoui and de~Fournas-Labrosse}{Balabdaoui
  and de~Fournas-Labrosse}{2020a}]{fadouagab}
Balabdaoui, F. and G.~de~Fournas-Labrosse (2020a).
\newblock Least squares estimation of a completely monotone pmf: from analysis
  to statistics.
\newblock {\em J. Statist. Plann. Inference\/},~{\bf 204}, 55--71.

\bibitem[\protect\citeauthoryear{Balabdaoui and de~Fournas-Labrosse}{Balabdaoui
  and de~Fournas-Labrosse}{2020b}]{bdF2019}
Balabdaoui, F. and G.~de~Fournas-Labrosse (2020b).
\newblock Least squares estimation of a completely monotone pmf: From analysis
  to statistics.
\newblock {\em JSPI\/},~{\bf 204}, 55--71.

\bibitem[\protect\citeauthoryear{Balabdaoui, Durot, and Koladjo}{Balabdaoui
  et~al.}{2017}]{balabd2017}
Balabdaoui, F., C.~Durot, and F.~Koladjo (2017).
\newblock On asymptotics of the discrete convex {LSE} of a p.m.f.
\newblock {\em Bernoulli\/}~{\em 23\/}(3), 1449--1480.

\bibitem[\protect\citeauthoryear{Balabdaoui and Jankowski}{Balabdaoui and
  Jankowski}{2016}]{fadouahanna}
Balabdaoui, F. and H.~Jankowski (2016).
\newblock Maximum likelihood estimation of a unimodal probability mass
  function.
\newblock {\em Statist. Sinica\/}~{\em 26\/}(3), 1061--1086.

\bibitem[\protect\citeauthoryear{Balabdaoui and Kulagina}{Balabdaoui and
  Kulagina}{2020}]{balabdkulagina}
Balabdaoui, F. and Y.~Kulagina (2020).
\newblock Completely monotone distributions: Mixing, approximation and
  estimation of number of species.
\newblock {\em Computational Statistics \& Data Analysis\/},~{\bf 150}, 107014,
  26.

\bibitem[\protect\citeauthoryear{Balabdaoui and Wellner}{Balabdaoui and
  Wellner}{2007}]{Balabdaoui07}
Balabdaoui, F. and J.~A. Wellner (2007).
\newblock Estimation of a {$k$}-monotone density: limit distribution theory and
  the spline connection.
\newblock {\em Ann. Statist.\/}~{\em 35\/}(6), 2536--2564.

\bibitem[\protect\citeauthoryear{B\"{o}hning and Patilea}{B\"{o}hning and
  Patilea}{2005}]{Patilea2005}
B\"{o}hning, D. and V.~Patilea (2005).
\newblock Asymptotic normality in mixtures of power series distributions.
\newblock {\em Scand. J. Statist.\/}~{\em 32\/}(1), 115--131.

\bibitem[\protect\citeauthoryear{Chee and Wang}{Chee and
  Wang}{2016}]{CheeWang2016}
Chee, C.-S. and Y.~Wang (2016).
\newblock Nonparametric estimation of species richness using discrete
  k-monotone distributions.
\newblock {\em Computational Statistics \& Data Analysis\/},~{\bf 93},
  107--118.

\bibitem[\protect\citeauthoryear{Chen}{Chen}{1995}]{chen1995}
Chen, J. (1995).
\newblock Optimal rate of convergence for finite mixture models.
\newblock {\em The Annals of Statistics\/}~{\em 23\/}(1), 221--233.

\bibitem[\protect\citeauthoryear{Davison and Hinkley}{Davison and
  Hinkley}{1997}]{davison-hinkley-1997}
Davison, A.~C. and D.~Hinkley (1997).
\newblock {\em Bootstrap Methods and Their Application}.
\newblock Cambridge University Press.

\bibitem[\protect\citeauthoryear{Durot}{Durot}{2007}]{durot07}
Durot, C. (2007).
\newblock On the {$\Bbb L_p$}-error of monotonicity constrained estimators.
\newblock {\em Ann. Statist.\/}~{\em 35\/}(3), 1080--1104.

\bibitem[\protect\citeauthoryear{Gao and Wellner}{Gao and
  Wellner}{2009}]{GaoWell09}
Gao, F. and J.~A. Wellner (2009).
\newblock On the rate of convergence of the maximum likelihood estimator of a
  {$k$}-monotone density.
\newblock {\em Sci. China Ser. A\/}~{\em 52\/}(7), 1525--1538.

\bibitem[\protect\citeauthoryear{Genovese and Wasserman}{Genovese and
  Wasserman}{2000}]{GW00}
Genovese, C.~R. and L.~Wasserman (2000).
\newblock Rates of convergence for the {G}aussian mixture sieve.
\newblock {\em Ann. Statist.\/}~{\em 28\/}(4), 1105--1127.

\bibitem[\protect\citeauthoryear{Ghosal and van~der Vaart}{Ghosal and van~der
  Vaart}{2001}]{GvdV01}
Ghosal, S. and A.~W. van~der Vaart (2001).
\newblock Entropies and rates of convergence for maximum likelihood and {B}ayes
  estimation for mixtures of normal densities.
\newblock {\em Ann. Statist.\/}~{\em 29\/}(5), 1233--1263.

\bibitem[\protect\citeauthoryear{Giguelay}{Giguelay}{2017}]{giguelay}
Giguelay, J. (2017).
\newblock Estimation of a discrete probability under constraint of
  {$k$}-monotonicity.
\newblock {\em Electron. J. Stat.\/}~{\em 11\/}(1), 1--49.

\bibitem[\protect\citeauthoryear{Giguelay and Huet}{Giguelay and
  Huet}{2018}]{giguelay2018}
Giguelay, J. and S.~Huet (2018).
\newblock Testing {$k$}-monotonicity of a discrete distribution. {A}pplication
  to the estimation of the number of classes in a population.
\newblock {\em Comput. Statist. Data Anal.\/},~{\bf 127}, 96--115.

\bibitem[\protect\citeauthoryear{Groeneboom}{Groeneboom}{1985}]{piet85}
Groeneboom, P. (1985).
\newblock Estimating a monotone density.
\newblock In {\em Proceedings of the {B}erkeley conference in honor of {J}erzy
  {N}eyman and {J}ack {K}iefer, {V}ol. {II} ({B}erkeley, {C}alif., 1983)},
  Wadsworth Statist./Probab. Ser., pp.\  539--555. Wadsworth, Belmont, CA.

\bibitem[\protect\citeauthoryear{Groeneboom, Jongbloed, and Wellner}{Groeneboom
  et~al.}{2001}]{GJW01}
Groeneboom, P., G.~Jongbloed, and J.~A. Wellner (2001).
\newblock Estimation of a convex function: characterizations and asymptotic
  theory.
\newblock {\em Ann. Statist.\/}~{\em 29\/}(6), 1653--1698.

\bibitem[\protect\citeauthoryear{Hengartner}{Hengartner}{1997}]{heng97}
Hengartner, N.~W. (1997).
\newblock Adaptive demixing in poisson mixture models.
\newblock {\em The Annals of Statistics\/}~{\em 25\/}(3), 917--928.

\bibitem[\protect\citeauthoryear{Jankowski and Wellner}{Jankowski and
  Wellner}{2009}]{hannajon}
Jankowski, H.~K. and J.~A. Wellner (2009).
\newblock Estimation of a discrete monotone distribution.
\newblock {\em Electron. J. Stat.\/},~{\bf 3}, 1567--1605.

\bibitem[\protect\citeauthoryear{Jewell}{Jewell}{1982}]{jewell}
Jewell, N.~P. (1982).
\newblock Mixtures of exponential distributions.
\newblock {\em Ann. Statist.\/}~{\em 10\/}(2), 479--484.

\bibitem[\protect\citeauthoryear{Kiefer and Wolfowitz}{Kiefer and
  Wolfowitz}{1956}]{kiefer1956}
Kiefer, J. and J.~Wolfowitz (1956).
\newblock Consistency of the maximum likelihood estimator in the presence of
  infinitely many incidental parameters.
\newblock {\em The Annals of Mathematical Statistics\/},~{\bf 27}, 887--906.

\bibitem[\protect\citeauthoryear{Lambert and Tierney}{Lambert and
  Tierney}{1984}]{Lambert}
Lambert, D. and L.~Tierney (1984).
\newblock Asymptotic properties of maximum likelihood estimates in the mixed
  {P}oisson model.
\newblock {\em Ann. Statist.\/}~{\em 12\/}(4), 1388--1399.

\bibitem[\protect\citeauthoryear{Lindsay}{Lindsay}{1983a}]{LindsayI}
Lindsay, B.~G. (1983a).
\newblock The geometry of mixture likelihoods: A general theory.
\newblock {\em The Annals of Statistics\/}~{\em 11\/}(1), 86--94.

\bibitem[\protect\citeauthoryear{Lindsay}{Lindsay}{1983b}]{LindsayII}
Lindsay, B.~G. (1983b).
\newblock The geometry of mixture likelihoods: The exponential family.
\newblock {\em The Annals of Statistics\/}~{\em 11\/}(3), 783--792.

\bibitem[\protect\citeauthoryear{Lindsay}{Lindsay}{1995}]{lindsay1995}
Lindsay, B.~G. (1995).
\newblock {\em Mixture models: Theory, geometry, and applications}.
\newblock Institute of Mathematical Statistics.

\bibitem[\protect\citeauthoryear{Lindsay and Lesperance}{Lindsay and
  Lesperance}{1995}]{Lindsay}
Lindsay, B.~G. and M.~L. Lesperance (1995).
\newblock A review of semiparametric mixture models.
\newblock {\em J. Statist. Plann. Inference\/}~{\em 47\/}(1-2), 29--39.
\newblock Statistical modelling (Leuven, 1993).

\bibitem[\protect\citeauthoryear{Loh and Zhang}{Loh and
  Zhang}{1996}]{loh1996global}
Loh, W.-L. and C.-H. Zhang (1996).
\newblock Global properties of kernel estimators for mixing densities in
  discrete exponential family models.
\newblock {\em Statistica Sinica\/},~{\bf 6}, 561--578.

\bibitem[\protect\citeauthoryear{McLachlan and Peel}{McLachlan and
  Peel}{2000}]{McLachlan}
McLachlan, G. and D.~Peel (2000).
\newblock {\em Finite Mixture Models}.
\newblock Wiley Series in Probability and Statistics: Applied Probability and
  Statistics. Wiley-Interscience, New York.

\bibitem[\protect\citeauthoryear{Norris~III and Pollock}{Norris~III and
  Pollock}{1998}]{NorrisIII1998391}
Norris~III, J.~L. and K.~H. Pollock (1998).
\newblock Non-parametric mle for poisson species abundance models allowing for
  heterogeneity between species.
\newblock {\em Environmental and Ecological Statistics\/}~{\em 5\/}(4), 391 –
  402.
\newblock Cited by: 64.

\bibitem[\protect\citeauthoryear{Patilea}{Patilea}{2001}]{Patilea}
Patilea, V. (2001).
\newblock Convex models, {MLE} and misspecification.
\newblock {\em Ann. Statist.\/}~{\em 29\/}(1), 94--123.

\bibitem[\protect\citeauthoryear{Polyanskiy and Wu}{Polyanskiy and
  Wu}{2021}]{polyanskiy2021sharp}
Polyanskiy, Y. and Y.~Wu (2021).
\newblock Sharp regret bounds for empirical bayes and compound decision
  problems.
\newblock arXiv preprint arXiv:2109.03943.

\bibitem[\protect\citeauthoryear{Roueff and Ryd{\'e}n}{Roueff and
  Ryd{\'e}n}{2005}]{roueff2005}
Roueff, F. and T.~Ryd{\'e}n (2005).
\newblock Nonparametric estimation of mixing densities for discrete
  distributions.
\newblock {\em The Annals of Statistics\/}~{\em 33\/}(5), 2066--2108.

\bibitem[\protect\citeauthoryear{Simar}{Simar}{1976}]{Simar}
Simar, L. (1976).
\newblock Maximum likelihood estimation of a compound {P}oisson process.
\newblock {\em Ann. Statist.\/}~{\em 4\/}(6), 1200--1209.

\bibitem[\protect\citeauthoryear{Szeg}{Szeg}{1939}]{szeg1939orthogonal}
Szeg, G. (1939).
\newblock {\em Orthogonal polynomials}, Volume~23.
\newblock American Mathematical Soc.

\bibitem[\protect\citeauthoryear{Titterington, Smith, and Makov}{Titterington
  et~al.}{1985}]{Titterington}
Titterington, D.~M., A.~F.~M. Smith, and U.~E. Makov (1985).
\newblock {\em Statistical Analysis of Finite Mixture Distributions}.
\newblock Wiley Series in Probability and Mathematical Statistics: Applied
  Probability and Statistics. John Wiley \& Sons.

\bibitem[\protect\citeauthoryear{{United States Geological Survey}}{{United
  States Geological Survey}}{2022}]{usgs}
{United States Geological Survey} (2022).
\newblock Earthquakes.
\newblock \url{https://www.usgs.gov/programs/earthquake-hazards/earthquakes}.
\newblock Accessed: 2022-05-10.

\bibitem[\protect\citeauthoryear{van~de Geer}{van~de Geer}{2000}]{sara}
van~de Geer, S.~A. (2000).
\newblock {\em Applications of empirical process theory}, Volume~6 of {\em
  Cambridge Series in Statistical and Probabilistic Mathematics}.
\newblock Cambridge University Press, Cambridge.

\bibitem[\protect\citeauthoryear{van~der Vaart and Wellner}{van~der Vaart and
  Wellner}{1996}]{aadbook}
van~der Vaart, A.~W. and J.~A. Wellner (1996).
\newblock {\em Weak convergence and empirical processes}.
\newblock Springer Series in Statistics. New York: Springer-Verlag.
\newblock With applications to statistics.

\bibitem[\protect\citeauthoryear{Wang}{Wang}{2007}]{wang-2007}
Wang, Y. (2007).
\newblock On fast computation of the non-parametric maximum likelihood estimate
  of a mixing distribution.
\newblock {\em Journal of the Royal Statistical Society, Ser. B\/},~{\bf 69},
  185--198.

\bibitem[\protect\citeauthoryear{Wang}{Wang}{2010}]{wang-2010}
Wang, Y. (2010).
\newblock Maximum likelihood computation for fitting semiparametric mixture
  models.
\newblock {\em Statistics and Computing\/},~{\bf 20}, 75--86.

\bibitem[\protect\citeauthoryear{Woo and Sriram}{Woo and Sriram}{2007}]{hell}
Woo, M.-J. and T.~Sriram (2007).
\newblock Robust estimation of mixture complexity for count data.
\newblock {\em Computational Statistics \& Data Analysis\/}~{\em 51\/}(9), 4379
  -- 4392.

\bibitem[\protect\citeauthoryear{Yu}{Yu}{1997}]{yu1997}
Yu, B. (1997).
\newblock Assouad, {F}ano, and {L}e {C}am.
\newblock In {\em Festschrift for Lucien Le Cam: {R}esearch papers in
  probability and statistics}, pp.\  423--435. Springer.

\bibitem[\protect\citeauthoryear{Zhang}{Zhang}{2009}]{Zha09}
Zhang, C.-H. (2009).
\newblock Generalized maximum likelihood estimation of normal mixture
  densities.
\newblock {\em Statist. Sinica\/}~{\em 19\/}(3), 1297--1318.

\bibitem[\protect\citeauthoryear{Zucchini, MacDonald, and Langrock}{Zucchini
  et~al.}{2016}]{zucchini-macdonald-langrock-2016}
Zucchini, W., I.~L. MacDonald, and R.~Langrock (2016).
\newblock {\em Hidden Markov Models for Time Series: An Introduction Using R\/}
  (2nd ed.), Volume 150.
\newblock CRC Press.

\end{thebibliography}

\end{document}